\documentclass[10pt,a4paper]{amsart}
\usepackage[a4paper,marginratio=1:1]{geometry}
\usepackage[protrusion=all]{microtype}
\usepackage{amssymb}
\usepackage[initials,non-sorted-cites]{amsrefs}
\usepackage{mybibspec}
\usepackage[shortlabels]{enumitem}
\usepackage[symbol,perpage]{footmisc}
\usepackage{graphicx}
\usepackage{braket,tensor,bm,stmaryrd,fouridx}
\usepackage{tikz-cd}

\numberwithin{equation}{section}
\usepackage{mathtools}\mathtoolsset{showonlyrefs}
\usepackage{stackengine}
\usepackage{rotating}

\newcommand{\abs}[1]{\lvert#1\rvert}

\newcommand{\conj}[1]{\overline{#1}}

\DeclareMathOperator{\tf}{tf}
\DeclareMathOperator{\Ad}{Ad}
\DeclareMathOperator{\ad}{ad}
\DeclareMathOperator{\gr}{gr}
\DeclareMathOperator{\im}{im}
\DeclareMathOperator{\End}{End}

\DeclareMathOperator{\Fl}{Fl}

\DeclareMathOperator{\Ric}{Ric}

\DeclareMathOperator{\id}{id}
\DeclareMathOperator{\tr}{tr}
\DeclareMathOperator{\rank}{rank}
\let\Re\relax\DeclareMathOperator{\Re}{Re}
\newcommand{\sym}{\mathrm{sym}}
\renewcommand{\skew}{\mathrm{skew}}
\newcommand{\proj}{\mathrm{proj}}
\newcommand{\SU}{\mathit{SU}}
\newcommand{\PSU}{\mathit{PSU}}

\newcommand{\GL}{\mathit{GL}}
\newcommand{\CU}{\mathit{CU}}
\newcommand{\intprod}{\mathbin{\raisebox{1pt}{\scalebox{1.3}{$\lrcorner$}}}}
\newcommand{\assocbracket}[1]{\llbracket#1\rrbracket}
\newcommand{\transpose}[1]{\fourIdx{t}{}{}{}{#1}}
\newcommand{\weylcomp}[1]{\mathsf{#1}}

\newtheorem{thm}{Theorem}[section]
\newtheorem{prop}[thm]{Proposition}
\newtheorem{lem}[thm]{Lemma}

\theoremstyle{definition}
\newtheorem{dfn}[thm]{Definition}
\theoremstyle{remark}
\newtheorem{rem}[thm]{Remark}

\newtheorem{notation}[thm]{Notation}

\title[]{The CR Killing operator and\\ Bernstein--Gelfand--Gelfand construction in CR geometry}
\author{Yoshihiko Matsumoto}
\address{Department of Mathematics, Graduate School of Science, Osaka University, Toyonaka, Osaka 560-0043, Japan}
\email{matsumoto@math.sci.osaka-u.ac.jp}
\subjclass[2020]{Primary 32V05; Secondary 53A40, 53B15, 53D10, 58H15.}

\begin{document}

\maketitle

\begin{abstract}
	We elaborate the tractor calculus for compatible almost CR structures
	(also known as strictly pseudoconvex partially integrable almost CR structures)
	on contact manifolds,
	and as an application, express the first BGG invariant differential operator $D_0$ explicitly in some cases,
	i.e., for some tractor connections.
	An interesting outcome is the fact that
	the ``modified'' adjoint tractor connection $\tilde{\nabla}$
	governing infinitesimal deformations of parabolic geometries
	generates what we call the CR Killing operator as its first BGG operator,
	and actually, it does not agree with the first BGG operator of the normal (unmodified) adjoint tractor connection.
	The relationship between the CR Killing operator and analysis of
	ACHE (asymptotically complex hyperbolic Einstein) metrics, or more specifically the CR obstruction tensor,
	is also discussed.
\end{abstract}

\section{Introduction}

Theory of parabolic geometries,
as partly summarized in the standard reference of \v{C}ap--Slov\'ak \cite{Cap-Slovak-09},
has proven to be a useful tool in studying
certain kinds of differential-geometric structures.
The fundamental idea of the theory is to single out one preferred Cartan connection
among the ones admitted by the underlying geometric structure
by introducing a normality condition described in the Lie-theoretic language,
thereby constructing preferable linear connections of naturally defined vector bundles (called tractor bundles),
which fit together well with Lie algebra homology/cohomology theory.
In this article, we describe this approach in detail in the case of CR geometry,
and in particular, how a specific linear differential operator $D$, which we suggest calling the CR Killing operator,
can be put in its context.

As pointed out by \v{C}ap--Schichl \cite{Cap-Schichl-00} and perhaps also intended by earlier authors,
the scope of the approach toward CR geometry from the theory of parabolic geometries
is not limited to CR structures in the classical sense; the formal integrability condition can be relaxed.
Stated in a more precise language, the category of normal regular parabolic geometries of type $(G,P)$,
where $G=\PSU(n+1,1)$ and $P$ is its parabolic subgroup that consists of all the elements preserving
some fixed null complex line in $\mathbb{C}^{n+1,1}$, is equivalent to
the category of so-called strictly pseudoconvex partially integrable almost CR structures on contact manifolds,
or compatible almost CR structures by the term coined in \cite{Matsumoto-21}.
The subtleties about the CR Killing operator that we are going to discuss
become visible only by taking those non-integrable CR structures into consideration.

We define the \emph{CR Killing operator}
\begin{equation}
	\label{eq:CR-Killing-operator-domain-and-target}
	D\colon \Re\mathcal{E}(1,1)\to
	\Re(\mathcal{E}_{(\alpha\beta)}(1,1)\oplus\mathcal{E}_{(\conj{\alpha}\conj{\beta})}(1,1))
\end{equation}
on a contact manifold equipped with a compatible almost CR structure by
\begin{equation}
	\label{eq:CR-Killing-operator}
	Df=
	(i\nabla_{(\alpha}\nabla_{\beta)}f-A_{\alpha\beta}f-iN_{(\alpha\beta)\gamma}\nabla^\gamma f,
	-i\nabla_{({\conj{\alpha}}}\nabla_{{\conj{\beta}})}f-A_{\conj{\alpha}\conj{\beta}}f
	+iN_{(\conj{\alpha}\conj{\beta})\conj{\gamma}}\nabla^{\conj{\gamma}}f)
\end{equation}
in terms of any Tanaka--Webster connection,
where the necessary notation to interpret \eqref{eq:CR-Killing-operator-domain-and-target} and
\eqref{eq:CR-Killing-operator} is recalled in the next section.
One can verify that the operator given by \eqref{eq:CR-Killing-operator} does not
depend on a particular choice of the Tanaka--Webster connection (i.e., a choice of a contact form),
or in other words, that it is determined purely by the compatible almost CR structure.
In fact, as we will see in Section \ref{sec:CR-Killing-operator},
this operator has a CR-geometric meaning: it describes trivial infinitesimal deformations
of compatible almost CR structures induced by contact Hamiltonian vector fields.

As already mentioned, one major purpose of this article is to illustrate that
the operator \eqref{eq:CR-Killing-operator} also pops up from the theory of parabolic geometries.
It is actually what is called the first BGG (Bernstein--Gelfand--Gelfand) operator associated with
the ``modified'' adjoint tractor connection introduced by \v{C}ap in \cite{Cap-08}.

\begin{thm}
	The first BGG operator $D_0^{\tilde{\nabla}}$ of the modified adjoint tractor connection $\tilde{\nabla}$
	equals the CR Killing operator $D$.
\end{thm}

Actually, it is also very important to note that, for general non-integrable compatible almost CR structures,
the ``unmodified'' normal adjoint tractor connection $\nabla$ does not play the same role;
using the modified connection $\tilde{\nabla}$ is really necessary.
The author believes that this observation gives some insight toward
the idea of ``holographic'' reconstruction of normal parabolic geometries
associated with compatible almost CR structures
on the conformal infinity of ACHE (asymptotically complex hyperbolic Einstein) spaces.

In order to compute the first BGG operators $D_0^\nabla$ and $D_0^{\tilde{\nabla}}$
and to show that the latter agrees with $D$,
we derive the formulae of $\nabla$ and $\tilde{\nabla}$ relative to
an arbitrarily chosen exact Weyl structure (which is equivalent to a contact form;
see Section \ref{subsec:exact-Weyl-structures-and-normal-Weyl-form}).
The derivation is given starting mostly from scratch.
The necessary general background is summarized in Section \ref{sec:general-theory-on-parabolic-geometry},
which contains an account of the correspondence of normal regular parabolic geometries
and underlying infinitesimal flag structures
and a sketch of the theory of tractor connections and BGG operators.
Section \ref{sec:Lie-algebra-homology-for-CR-geometry} is devoted to Lie algebra homology computations
that we need to specify the normality condition of the CR Cartan connection
and also to invoke the BGG construction later.
In Section \ref{sec:construction-of-normal-tractor-connections},
we first recall the general notion of Weyl structures and Weyl forms
and how they are used to describe tractor connections.
They are followed by a description of frame bundles associated with CR geometry,
and in the latter part of this section, the determination of CR normal Weyl forms is given.
Finally, in Section \ref{sec:computation-of-CR-BGG-operators}, we obtain the promised formulae
of the adjoint tractor connections and their associated first BGG operators.
Here we also discuss the case of the standard tractor connection;
for this purpose, in Section \ref{sec:construction-of-normal-tractor-connections}, we also consider
the frame bundle corresponding to the group $G^\sharp=\SU(n+1,1)$,
which is an $(n+2)$-sheeted covering of $G=\PSU(n+1,1)$.

In the literature, CR tractor bundles \emph{over integrable CR structures} have been discussed in several ways.
The most explicit, hands-on approach is taken in a work of Gover--Graham \cite{Gover-Graham-05},
which gives the formula of the standard (co)tractor connection relative to exact Weyl structures
and actually \emph{uses it as the definition} of the connection, without referencing to
normal Cartan connections or normal Weyl forms.
Another description of standard tractors using
the ambient metric construction of Fefferman \cite{Fefferman-76} is sketched by \v{C}ap \cite{Cap-02};
an analogous theory in the case of conformal geometry is detailed in a work of \v{C}ap--Gover \cite{Cap-Gover-03}.
Yet another, closely related approach via Fefferman's conformal circle bundle over CR manifolds
(\cite{Fefferman-76}, see also Lee \cite{Lee-86})
is also mentioned in \cite{Cap-02} and extensively used by \v{C}ap--Gover \cite{Cap-Gover-08}.
Herzlich \cite{Herzlich-09} gives a realization of the standard tractor bundle
as the bundle of 1-jets of CR-holomorphic sections of a certain complex line bundle,
which resembles the work of Bailey--Eastwood--Gover \cite{Bailey-Eastwood-Gover-94} in conformal geometry.

However, tractor bundles over non-integrable compatible almost CR structures have not been investigated well,
which are, as our observation reveals, also worthy of attention.
Note also that, in the terminology of the general theory of parabolic geometries,
integrable CR structures comprise the subclass of torsion-free geometries
of the class of $\abs{2}$-graded geometries corresponding to compatible almost CR structures.
The investigation given in this paper might serve as important working example in the general theory,
especially regarding the role played by the torsion of parabolic geometries.

In addition to this, we also believe that this article fixes a lack of detailed references
in the integrable case
on how the CR standard tractor connection of Gover--Graham \cite{Gover-Graham-05}
can be reconstructed based on the general theory of parabolic geometries.

To conclude the introduction, we briefly discuss possible relationships between CR tractor bundles and
other constructions for compatible almost CR structures.
For integrable CR structures, the standard tractor bundle can be recovered by
the ambient metric construction as we already mentioned;
by contrast, there is no plausible notion of the ambient metric for general compatible almost CR structures.
But here we want to recall the related notion of complete K\"ahler-Einstein spaces
in the integrable case, whose boundaries at infinity (``conformal infinity'') carry integrable CR structures and
the total spaces of whose canonical bundle, with zero section removed, carry the ambient metrics
(see \cite{Fefferman-76} or an article of Fefferman--Graham \cite{Fefferman-Graham-85}).
These complete K\"ahler-Einstein spaces generalize to ACHE spaces of Biquard \cite{Biquard-00},
whose conformal infinities are equipped with compatible almost CR structures.
Consequently, a natural hope may be that ACHE spaces can be used to recover the standard tractor bundles
associated with compatible almost CR structures.

Our observation in this paper sheds a doubt over this na\"ive idea.
ACHE spaces have some indirect relationship with the CR Killing operator via
the CR obstruction tensor introduced by the author \cite{Matsumoto-14}
as we recall in Section \ref{sec:CR-Killing-operator},
and so there should be some nice way to reconstruct the CR Killing operator using ACHE metrics.
But as we will see, the CR Killing operator is related to the modified adjoint tractor connection $\tilde{\nabla}$,
which has little to do with the standard tractor bundle.
Therefore, a more natural expectation might be as this:
there is a way to reconstruct the adjoint tractor bundle $\mathcal{A}M$ of
compatible almost CR structures $J$ from ACHE spaces,
and when $J$ is integrable, the realization of $\mathcal{A}M$ as a subbundle of $\End(\mathcal{V})$
via the standard tractor bundle $\mathcal{V}$
(see Section \ref{subsec:BGG-opertor-for-normal-adjoint-tractor-connection})
can be seen as its specialization.
The author wants to come back to such a holographic reconstruction of $\mathcal{A}M$ using ACHE spaces
in the future.
In fact, another article \cite{Matsumoto-21} by the author is also an attempt toward this idea.

\section{Preliminaries}
\label{sec:prelim}

We summarize basic matters in pseudohermitian geometry of compatible almost CR structures
by largely following \cite{Matsumoto-14}*{Section 3}, supplementing them with some new formulae
and minor notational modifications.

\subsection{Basic definitions}
\label{subsec:basics}

Throughout this article, we assume that we are given a contact manifold $(M,H)$ of dimension $2n+1$
that is cooriented,
i.e., associated with a fixed orientation of the line bundle $H^\perp$ of 1-forms annihilating $H$.
By a \emph{contact form} we mean a nowhere-vanishing section of the bundle $H^\perp$ that is positive
with respect to the given orientation;
hence any two contact forms are related as $\Hat{\theta}=e^u\theta$, where $u\in C^\infty(M)$.
Any choice of $\theta$ determines the direct sum decomposition
\begin{equation}
	\label{eq:Reeb-splitting}
	TM=H\oplus\mathbb{R}T,
\end{equation}
where $T$, the \emph{Reeb vector field}, is characterized by $\theta(T)=1$ and $T\intprod d\theta=0$.

Let $J$ be an almost CR structure on a contact manifold $(M,H)$,
by which we mean that $J$ is a complex structure of $H$.
The \emph{compatibility} of the almost CR structure $J$, which we always assume, means that
\begin{equation}
	\label{eq:Levi-form}
	h(X,Y)=d\theta(X,JY),\qquad X,\,Y\in H
\end{equation}
is a positive-definite Hermitian form on $H$, called the
\emph{Levi form}\footnote{We assume for brevity that the Levi form $h$ is positive definite, but
everything in this article can be immediately generalized to the case in which $h$ has arbitrary
(nondegenerate) signature as well.}.
The compatibility condition is irrelevant to the choice of $\theta$ because
$\Hat{\theta}=e^u\theta$ implies $d\Hat{\theta}(X,JY)=e^ud\theta(X,JY)$.
We sometimes call $(H,J)$, instead of $J$, a compatible almost CR structure.

\begin{rem}
	It is well known that the formal integrability condition of $J$
	implies that \eqref{eq:Levi-form} is a Hermitian form.
	Then the nondegeneracy of \eqref{eq:Levi-form} is automatic because of the contact condition,
	and the positive-definiteness is referred to as \emph{strict pseudoconvexity}.
	Therefore, the notion of compatible almost CR structures is a generalization of
	that of strictly pseudoconvex (integrable) CR structures.
	The point of this Section \ref{sec:prelim} is that
	pseudohermitian geometry does not change much in the broader class.
	Our ``compatible almost CR structures'' have been more often called
	``strictly pseudoconvex partially integrable almost CR structures''
	(e.g., in \v{C}ap--Schichl \cite{Cap-Schichl-00})
	in view of the fact that
	\eqref{eq:Levi-form} being a Hermitian form is equivalent to a weaker formal integrability.
	We also note that a compatible almost CR structure $(H,J)$ together with a fixed contact form $\theta$
	is equivalent to a \emph{contact metric structure} in the sense of Blair \cite{Blair-76}.
\end{rem}

A compatible almost CR structure $(H,J)$ naturally comes with the decomposition
\begin{equation}
	\label{eq:CR-splitting}
	\mathbb{C}H=H^{1,0}\oplus H^{0,1}
\end{equation}
of the complexification of $H$ into the $(\pm i)$-eigenspaces of $J$
(the compatibility of $J$ is actually irrelevant here).
In actual computations, we often need to take a local frame $\set{Z_\alpha}$ of $H^{1,0}$.
By doing so, at the same time, we obtain a local frame $\set{Z_{\conj{\alpha}}}$ of $H^{0,1}$,
which is the complex conjugate of $\set{Z_\alpha}$.

When $\theta$ is moreover fixed,
we have the \emph{admissible coframe} $\set{\theta^\alpha}$ associated with any fixed local frame $\set{Z_\alpha}$,
which is the set of complex 1-forms satisfying
$\theta^\alpha(Z_\beta)=\tensor{\delta}{_\beta^\alpha}$ and $\theta^\alpha(Z_{\conj{\beta}})=\theta^\alpha(T)=0$.
Its complex conjugate is denoted by $\set{\theta^{\conj{\alpha}}}$.
Then \eqref{eq:Levi-form} and $T\intprod d\theta=0$ imply that
\begin{equation}
	\label{eq:d-theta-and-Levi-form}
	d\theta=ih_{\alpha{\conj{\beta}}}\theta^\alpha\wedge\theta^{\conj{\beta}},
\end{equation}
where $h_{\alpha{\conj{\beta}}}=h(Z_\alpha,Z_{\conj{\beta}})$ is pointwisely a positive-definite Hermitian matrix.

A section of $H^{1,0}$ can be expressed locally by symbols with an index upstairs (e.g., $v^\alpha$)
in terms of any fixed local frame $\set{Z_\alpha}$.
For this reason, the bundle $H^{1,0}$ itself is denoted by $\mathcal{E}^\alpha$ in the sequel, where
$\alpha$ is an ``abstract index.''
In general, any tensor bundle (i.e., a tensor product of $H^{1,0}$, $H^{0,1}$, $(H^{1,0})^*$, $(H^{1,0})^*$
and its geometrically natural subbundle)
will be denoted by $\mathcal{E}$ associated with the same indices used for
expressing their sections locally, and if applicable, some additional symbols representing a particular type of
subbundles.
For example, the Levi form $h_{\alpha\conj{\beta}}$ is a section of
$\mathcal{E}_{\alpha\conj{\beta}}=(H^{1,0})^*\otimes(H^{1,0})^*$
(actually $h_{\alpha\conj{\beta}}$ has the Hermitian symmetry and so is a section of
the Hermitian symmetric part of $\mathcal{E}_{\alpha\conj{\beta}}$, which has no commonly accepted notation).
The Tanaka--Webster torsion tensor $A_{\alpha\beta}$ (defined in Section \ref{subsec:Tanaka-Webster-connection})
is a section of $\mathcal{E}_{\alpha\beta}=(H^{1,0})^*\otimes(H^{1,0})^*$,
and since $A_{\alpha\beta}=A_{\beta\alpha}$, it is further said that it is a section of
$\mathcal{E}_{(\alpha\beta)}$, where $(\cdots)$ denotes the symmetrization with respect to the enclosed indices.
Likewise, $[\cdots]$ denotes the skew-symmetrization.
Note also that the symmetry property $A_{\alpha\beta}=A_{\beta\alpha}$ can also be expressed as
$A_{\alpha\beta}=A_{(\alpha\beta)}$.

Furthermore, we observe the well-accepted custom in the field
that the space of (possibly local) sections, or the sheaf of germs of
local sections, of a tensor bundle is denoted by the same symbol assigned to the tensor bundle itself.
The distinction will be clearly made by context.

Next, we introduce the \emph{density bundles} $\mathcal{E}(w,w')$,
where $w$ and $w'$ are integers\footnote{The definition of $\mathcal{E}(w,w')$ can immediately be
generalized to $w$, $w'\in\mathbb{C}$ satisfying $w-w'\in\mathbb{Z}$, which we do not need.
See Gover--Graham \cite{Gover-Graham-05}*{p.~4}.}.
The \emph{canonical bundle} $\mathcal{K}$ over $(M,H,J)$ is the complex line bundle $\bigwedge^{n+1}(H^{0,1})^\perp$;
$\mathcal{K}$ is generated by $\theta\wedge\theta^1\wedge\dots\wedge\theta^n$
for any choice of $\theta$ and an admissible coframe $\set{\theta^\alpha}$.
Working locally if necessary, we fix an $(n+2)$-nd root of $\mathcal{K}$
and write its dual $\mathcal{E}(1,0)$.
Then we define
\begin{equation}
	\mathcal{E}(w,w')=(\mathcal{E}(1,0))^{\otimes w}\otimes(\conj{\mathcal{E}(1,0)})^{\otimes w'},\qquad
	w,\,w'\in\mathbb{Z}
\end{equation}
and its sections are called \emph{$(w,w')$-densities}.
Moreover, we write
$\tensor{\mathcal{E}}{_\alpha_{\conj{\beta}}^\sigma^{\conj{\tau}}}(w,w')
=\tensor{\mathcal{E}}{_\alpha_{\conj{\beta}}^\sigma^{\conj{\tau}}}\otimes\mathcal{E}(w,w')$, etc.,
and sections of such bundles are referred to as \emph{weighted tensors}.

Particularly important density bundles are those of the form $\mathcal{E}(w,w)$, for they are
independent of the choice of $\mathcal{E}(1,0)$, and as a consequence, globally defined.
In other words, although tensors weighted by general $\mathcal{E}(w,w')$ should be considered as an object
defined on $(M,H,J,\mathcal{E}(1,0))$,
those weighted by $\mathcal{E}(w,w)$ can be understood as an object on $(M,H,J)$.
Some readers may feel comfortable by regarding tensors weighted by $\mathcal{E}(w,w')$, $w\not=w'$, as
intermediate objects and those weighted by $\mathcal{E}(w,w)$ as ``actual'' objects.

Any choice of a contact form $\theta$ determines a trivialization of $\mathcal{E}(w,w)$ uniquely.
To see this,
let $\zeta=\theta\wedge\theta^1\wedge\dots\wedge\theta^n$ be a section of $\mathcal{K}=\mathcal{E}(-n-2,0)$,
where $\set{\theta^\alpha}$ is some unitary admissible coframe with respect to the Levi form $h$,
possibly only locally defined.
Then $\zeta$ is pointwisely well-defined up to phase
(and it is volume-normalized with respect to $\theta$ in the sense of Lee \cite{Lee-86}*{Section 3}).
If $\xi$ is a section of $\mathcal{E}(1,0)$ satisfying $\xi^{-n-2}=\zeta$, then
$\abs{\xi}^{2w}$ is a globally defined section of $\mathcal{E}(w,w)$ that depends only on $\theta$.
The mentioned trivialization of $\mathcal{E}(w,w)$ can be given by $\abs{\xi}^{2w}$.

\begin{rem}
	For any change $\Hat{\theta}=e^u\theta$ of contact forms,
	the $(w,w)$-density $\abs{\xi}^{2w}$ scales as $\abs{\Hat{\xi}}^{2w}=e^{-wu}\abs{\xi}^{2w}$.
	Therefore, if $f$ and $\Hat{f}$ are the trivializations of the same $(w,w)$-density
	with respect to $\theta$ and $\Hat{\theta}$, respectively, then $\Hat{f}=e^{wu}f$.
	An intuitive understanding of $\mathcal{E}(w,w)$ is that it is the bundle of scalar-valued quantities
	that rescales like $\Hat{f}=e^{wu}f$ for the change $\Hat{\theta}=e^u\theta$.
\end{rem}

We have $\abs{\Hat{\xi}}^2=e^{-u}\abs{\xi}^2$ in particular.
Consequently, the \emph{weighted contact form}
\begin{equation}
	\bm{\theta}=\theta\otimes\abs{\xi}^2
\end{equation}
is independent of $\theta$,
which also implies that there is a canonical identification $H^\perp\cong\mathcal{E}(-1,-1)$.
The \emph{weighted Levi form}
\begin{equation}
	\bm{h}_{\alpha{\conj{\beta}}}=h_{\alpha{\conj{\beta}}}\otimes\abs{\xi}^2
	\in\mathcal{E}_{\alpha{\conj{\beta}}}(1,1)
\end{equation}
and its inverse
\begin{equation}
	\bm{h}^{\alpha{\conj{\beta}}}=h^{\alpha{\conj{\beta}}}\otimes\abs{\xi}^{-2}
	\in\mathcal{E}^{\alpha{\conj{\beta}}}(-1,-1)
\end{equation}
are also independent of $\theta$.

\begin{notation}
	Whenever some fixed $\theta$ is taken, possibly implicitly,
	for any tensor weighted by $\mathcal{E}(w,w)$, we use the same symbol to express itself and
	its trivialization with respect to $\theta$ (in order to reduce the use of boldface letters).
	The distinction is made by context.
\end{notation}

Based on the above notational convention,
we could have used the symbol $h_{\alpha{\conj{\beta}}}$ for the weighted Levi form
instead of $\bm{h}_{\alpha{\conj{\beta}}}$
(and keep the unweighted one denoted by $h_{\alpha{\conj{\beta}}}$ as well),
and even could have used $\theta$ for the weighted contact form.
We try to take this way from now on as much as possible.

\subsection{The Nijenhuis tensor}
\label{subsec:Nijenhuis-tensor}

The non-integrability of a compatible almost CR structure $J$ is measured by
the \emph{Nijenhuis tensor} $N$, which is a section of
$\tensor{\mathcal{E}}{^{\conj{\gamma}}_[_\alpha_\beta_]}$,
defined by\footnote{We use a different convention
on the sign and the order of the indices compared to \cite{Matsumoto-14}*{p.~2137}.}
\begin{equation}
	[Z_\alpha,Z_\beta]=-\tensor{N}{^{\conj{\gamma}}_\alpha_\beta}Z_{\conj{\gamma}}\mod H^{1,0},
\end{equation}
or equivalently,
\begin{equation}
	\label{eq:definition-of-Nijenhuis-tensor}
	d\theta^{\conj{\gamma}}=\frac{1}{2}\tensor{N}{^{\conj{\gamma}}_\alpha_\beta}\theta^\alpha\wedge\theta^\beta
	\mod\theta^{\conj{\alpha}},\,\theta,\qquad
	\tensor{N}{^{\conj{\gamma}}_(_\alpha_\beta_)}=0.
\end{equation}
Moreover, $N$ can be understood as a real tensor by setting
\begin{equation}
	\tensor{N}{^\gamma_{\conj{\alpha}}_{\conj{\beta}}}=\conj{\tensor{N}{^{\conj{\gamma}}_\alpha_\beta}},
\end{equation}
the practice we repeatedly apply to various tensors in the sequel, sometimes without mentioning.

The index upstairs of $\tensor{N}{^{\conj{\gamma}}_\alpha_\beta}$ can be lowered
using the weighted Levi form as
\begin{equation}
	\label{eq:lowered-Nijenhuis-tensor}
	N_{\gamma\alpha\beta}=h_{\gamma{\conj{\sigma}}}\tensor{N}{^{\conj{\sigma}}_\alpha_\beta}.
\end{equation}
Therefore, $N_{\gamma\alpha\beta}$ is a section of $\mathcal{E}_{\gamma[\alpha\beta]}(1,1)$.
Given a fixed contact form $\theta$,
$h_{\gamma{\conj{\sigma}}}$ on the right-hand side can also be understood as the unweighted Levi form,
$N_{\gamma\alpha\beta}$ on the left as the trivialization of its weighted version,
and \eqref{eq:lowered-Nijenhuis-tensor}
is justified as an equality between sections of $\mathcal{E}_{\gamma[\alpha\beta]}$ as well.

It can be shown that the index-lowered Nijenhuis tensor $N_{\gamma\alpha\beta}$ has a symmetry,
in addition to the obvious $N_{\gamma(\alpha\beta)}=0$, that
\begin{equation}
	\label{eq:Nijenhuis-symmetry}
	N_{\alpha\beta\gamma}+N_{\beta\gamma\alpha}+N_{\gamma\alpha\beta}=0,
\end{equation}
by differentiating \eqref{eq:d-theta-and-Levi-form} and using \eqref{eq:definition-of-Nijenhuis-tensor}.

For convenience, we write
\begin{equation}
	N^\sym_{\alpha\beta\gamma}=N_{(\alpha\beta)\gamma},\qquad
	N^\skew_{\alpha\beta\gamma}=N_{[\alpha\beta]\gamma}.
\end{equation}
Note that \eqref{eq:Nijenhuis-symmetry} implies
\begin{equation}
	N^\skew_{\alpha\beta\gamma}
	=\frac{1}{2}(N_{\alpha\beta\gamma}-N_{\beta\alpha\gamma})
	=-\frac{1}{2}N_{\gamma\alpha\beta}.
\end{equation}
Moreover, if we define $\abs{N}^2=N_{\alpha\beta\gamma}N^{\alpha\beta\gamma}$
and $\abs{N^\sym}^2$, $\abs{N^\skew}^2$ in a similar manner, then we have
\begin{equation}
	\label{eq:decomposition-of-squared-norm-of-N}
	\abs{N^\skew}^2=\frac{1}{4}\abs{N}^2,\qquad
	\abs{N^\sym}^2=\abs{N}^2-\abs{N^\skew}^2=\frac{3}{4}\abs{N}^2.
\end{equation}

Furthermore, when $\theta$ is fixed, we write
\begin{equation}
	(\nabla^*N)^\sym_{\alpha\beta}=-\nabla^\gamma N_{(\alpha\beta)\gamma},\qquad
	(\nabla^*N)^\skew_{\alpha\beta}=-\nabla^\gamma N_{[\alpha\beta]\gamma}
\end{equation}
using the Tanaka--Webster connection $\nabla$ defined in the next subsection.

\subsection{The Tanaka--Webster connection}
\label{subsec:Tanaka-Webster-connection}

The \emph{Tanaka--Webster connection} of a manifold $M$ with
a compatible almost CR structure $(H,J)$ associated with a choice of a contact form $\theta$
is the connection of $TM$ characterized by the fact that it preserves the splitting \eqref{eq:Reeb-splitting},
makes $T$, $J$, $h$ parallel, and satisfies some torsion condition
(see \cite{Matsumoto-14}*{Proposition 3.1} for the detailed characterization).
Note that the connection also preserves the complex splitting \eqref{eq:CR-splitting} because $J$ is parallel.
Therefore, upon fixing a local frame $\set{Z_\alpha}$ of $H^{1,0}$,
the Tanaka--Webster connection $\nabla$ is described by the connection forms $\tensor{\omega}{_\alpha^\beta}$
and their complex conjugates $\tensor{\omega}{_{\conj{\alpha}}^{\conj{\beta}}}$.
Since $\nabla h=0$ (where $h$ is the unweighted Levi form), we have
\begin{equation}
	\label{eq:metric-compatibility-in-connection-form}
	\omega_{\alpha{\conj{\beta}}}+\omega_{{\conj{\beta}}\alpha}=dh_{\alpha{\conj{\beta}}}.
\end{equation}
The torsion condition for $\nabla$ implies that
\begin{equation}
	\label{eq:first-structure-equation}
	d\theta^\gamma
	=\theta^\alpha\wedge\tensor{\omega}{_\alpha^\gamma}
	+\tensor{A}{_{\conj{\alpha}}^\gamma}\theta\wedge\theta^{\conj{\alpha}}
	+\frac{1}{2}\tensor{N}{^\gamma_{\conj{\alpha}}_{\conj{\beta}}}\theta^{\conj{\alpha}}\wedge\theta^{\conj{\beta}}
\end{equation}
holds for some uniquely determined tensor
$\tensor{A}{_{\conj{\alpha}}^\gamma}\in\tensor{\mathcal{E}}{_{\conj{\alpha}}^\gamma}(-1,-1)$,
which is called the \emph{Tanaka--Webster torsion tensor}.
It is known that $A_{\alpha\beta}=A_{\beta\alpha}$.

The Tanaka--Webster connection induces a connection of the canonical bundle $\mathcal{K}$,
and hence that of $\mathcal{E}(w,w')$, all of which are denoted by $\nabla$.
Let $\zeta=\theta\wedge\theta^1\wedge\dots\wedge\theta^n$ be a section of $\mathcal{K}$,
where $\set{\theta^\alpha}$ is some unitary admissible coframe with respect to $\theta$.
Then we have $\nabla\zeta=-\tensor{\omega}{_\gamma^\gamma}\otimes\zeta$.
Consequently, if $\xi$ is a section of $\mathcal{E}(1,0)$ satisfying $\xi^{-n-2}=\zeta$, then
considering \eqref{eq:metric-compatibility-in-connection-form} we obtain
\begin{equation}
	\label{eq:TW-connection-on-densities}
	\nabla(\xi^w\otimes\smash{\conj{\xi}}^{w'})
	=\frac{w-w'}{n+2}\tensor{\omega}{_\gamma^\gamma}
	\otimes(\xi^w\otimes\smash{\conj{\xi}}^{w'}).
\end{equation}
In particular, $\abs{\xi}^{2w}=\xi^w\otimes\smash{\conj{\xi}}^w$ is a parallel section of $\mathcal{E}(w,w)$,
and hence the trivialization of weighted tensors using $\abs{\xi}^{2w}$ is compatible with
covariant differentiation.
As a result, the weighted contact form and the weighted Levi form are also parallel.

We express the curvature form
$\tensor{\Pi}{_\alpha^\beta}
=d\tensor{\omega}{_\alpha^\beta}-\tensor{\omega}{_\alpha^\gamma}\wedge\tensor{\omega}{_\gamma^\beta}$
as
\begin{equation}
	\label{eq:Tanaka-Webster-curvature-form}
	\tensor{\Pi}{_\alpha^\beta}
	=\tensor{R}{_\alpha^\beta_\sigma_{\conj{\tau}}}\theta^\sigma\wedge\theta^{\conj{\tau}}
	+\tensor{W}{_\alpha^\beta_\gamma}\theta^\gamma\wedge\theta
	+\tensor{W}{_\alpha^\beta_{\conj{\gamma}}}\theta^{\conj{\gamma}}\wedge\theta
	+\frac{1}{2}\tensor{V}{_\alpha^\beta_\sigma_\tau}\theta^\sigma\wedge\theta^\tau
	+\frac{1}{2}\tensor{V}{_\alpha^\beta_{\conj{\sigma}}_{\conj{\tau}}}\theta^{\conj{\sigma}}\wedge\theta^{\conj{\tau}},
\end{equation}
where $\tensor{V}{_\alpha^\beta_(_\sigma_\tau_)}=\tensor{V}{_\alpha^\beta_(_{\conj{\sigma}}_{\conj{\tau}}_)}=0$.
Among the components on the right-hand side,
$\tensor{R}{_\alpha^\beta_\sigma_{\conj{\tau}}}$ is called the \emph{Tanaka--Webster curvature tensor},
and it satisfies
\begin{equation}
	\label{eq:Tanaka-Webster-curvature-symmetry}
	\tensor{R}{_\alpha_{\conj{\beta}}_\sigma_{\conj{\tau}}}
	=\tensor{R}{_{\conj{\beta}}_\alpha_{\conj{\tau}}_\sigma},\qquad
	\tensor{R}{_\alpha^\beta_\sigma_{\conj{\tau}}}
	-\tensor{R}{_\sigma^\beta_\alpha_{\conj{\tau}}}
	=-\tensor{N}{^{\conj{\gamma}}_\alpha_\sigma}\tensor{N}{^\beta_{\conj{\tau}}_{\conj{\gamma}}}.
\end{equation}
The other components are given in terms of $N$ and $A$ as follows\footnote{Two vertical bars within
$(\dotsb)$ (resp.\ $[\dotsb]$) indicate that the indices between them are excluded from
the symmetrization (resp.\ the skew-symmetrization).}:
\begin{align}
	\tensor{W}{_\alpha^\beta_\gamma}
	&=\nabla^\beta A_{\alpha\gamma}+N_{\alpha\gamma\sigma}A^{\beta\sigma},
	&\tensor{W}{_\alpha^\beta_{\conj{\gamma}}}
	&=-\tensor{W}{^\beta_\alpha_{\conj{\gamma}}}
	=-\nabla_\alpha\tensor{A}{^\beta_{\conj{\gamma}}}
	-\tensor{N}{^\beta_{\conj{\gamma}}_{\conj{\sigma}}}\tensor{A}{_\alpha^{\conj{\sigma}}},\\
	\tensor{V}{_\alpha^\beta_\sigma_\tau}
	&=2i\tensor{\delta}{_[_\sigma_|^\beta}\tensor{A}{_\alpha_|_\tau_]}
	-\nabla^\beta N_{\alpha\sigma\tau},
	&\tensor{V}{_\alpha^\beta_{\conj{\sigma}}_{\conj{\tau}}}
	&=-\tensor{V}{^\beta_\alpha_{\conj{\sigma}}_{\conj{\tau}}}
	=2i\tensor{h}{_\alpha_[_{\conj{\sigma}}}\tensor{A}{^\beta_{\conj{\tau}}_]}
	+\nabla_\alpha\tensor{N}{^\beta_{\conj{\sigma}}_{\conj{\tau}}}.
\end{align}
Moreover, we define
\begin{equation}
	R_{\alpha{\conj{\beta}}}=\tensor{R}{_\gamma^\gamma_\alpha_{\conj{\beta}}},\qquad
	R'_{\alpha{\conj{\beta}}}=\tensor{R}{_\alpha_{\conj{\beta}}_\gamma^\gamma},\qquad
	R''_{\alpha{\conj{\beta}}}=\tensor{R}{_\alpha^\gamma_\gamma_{\conj{\beta}}}
\end{equation}
and
\begin{equation}
	R=\tensor{R}{_\gamma^\gamma},\qquad
	R'=\tensor{{R'}}{_\gamma^\gamma}(=R),\qquad
	R''=\tensor{{R''}}{_\gamma^\gamma}.
\end{equation}
Then it follows from \eqref{eq:Tanaka-Webster-curvature-symmetry} that
\begin{equation}
	\label{eq:Ricci-primed-identities}
	R'_{\alpha{\conj{\beta}}}
	=R_{\alpha{\conj{\beta}}}
	+2N_{(\alpha\lambda)\mu}\tensor{N}{^\mu^\lambda_{\conj{\beta}}},
	\qquad
	R''_{\alpha{\conj{\beta}}}
	=R_{\alpha{\conj{\beta}}}
	-N_{\lambda\mu\alpha}\tensor{N}{^\mu^\lambda_{\conj{\beta}}}
\end{equation}
and
\begin{equation}
	R''=R-N_{\lambda\mu\nu}N^{\mu\lambda\nu}
	=R-\abs{N^\sym}^2+\abs{N^\skew}^2=R-\frac{1}{2}\abs{N}^2
\end{equation}
by \eqref{eq:decomposition-of-squared-norm-of-N}.
Note also that $R_{\alpha{\conj{\beta}}}=R_{{\conj{\beta}}\alpha}$,
$R'_{\alpha{\conj{\beta}}}=R'_{{\conj{\beta}}\alpha}$ immediately follow from the definitions
and $R''_{\alpha{\conj{\beta}}}=R''_{{\conj{\beta}}\alpha}$ holds
because of the second equality in \eqref{eq:Ricci-primed-identities}.

As in Gover--Graham \cite{Gover-Graham-05}*{p.\ 6},
we define the \emph{CR Schouten tensor} $P_{\alpha{\conj{\beta}}}$ by
\begin{equation}
	P_{\alpha{\conj{\beta}}}
	=\frac{1}{n+2}\left(R_{\alpha{\conj{\beta}}}-\frac{1}{2(n+1)}Rh_{\alpha{\conj{\beta}}}\right),
\end{equation}
and $P'_{\alpha{\conj{\beta}}}$ (resp.\ $P''_{\alpha{\conj{\beta}}}$) will denote
the similarly defined quantity with $R_{\alpha{\conj{\beta}}}$ replaced with
$R'_{\alpha{\conj{\beta}}}$ (resp.\ $R''_{\alpha{\conj{\beta}}}$) and
$R$ replaced with $R'$ (resp.\ $R''$) (but note that $R'=R$).
Then \eqref{eq:Ricci-primed-identities} implies
\begin{align}
	P'_{\alpha{\conj{\beta}}}
	&=P_{\alpha{\conj{\beta}}}
	+\frac{2}{n+2}N_{(\alpha\lambda)\mu}\tensor{N}{^\mu^\lambda_{\conj{\beta}}},\\
	P''_{\alpha{\conj{\beta}}}
	&=P_{\alpha{\conj{\beta}}}
	-\frac{1}{n+2}N_{\lambda\mu\alpha}\tensor{N}{^\mu^\lambda_{\conj{\beta}}}
	+\frac{1}{4(n+1)(n+2)}\abs{N}^2h_{\alpha{\conj{\beta}}}.
\end{align}
We define $P=\tensor{P}{_\gamma^\gamma}$, $P'=\tensor{{P'}}{_\gamma^\gamma}$,
and $P''=\tensor{{P''}}{_\gamma^\gamma}$, which are related by
\begin{equation}
	\label{eq:relationship-of-various-P}
	P'=P,\qquad
	P''=P-\frac{1}{4(n+1)}\abs{N}^2.
\end{equation}

We will later need the following transformation law,
taken from \cite{Matsumoto-14}*{Proposition 3.6}, for changes of contact forms.

\begin{prop}
	\label{prop:TW-for-contact-form-change}
	Let $\theta$ and $\Hat{\theta}=e^u\theta$ be two contact forms, where $u\in C^\infty(M)$.
	Then, for any local frame $\set{Z_\alpha}$ and
	the dual admissible coframe $\set{\theta^\alpha}$ with respect to $\theta$,
	the Tanaka--Webster connection forms $\tensor{\Hat{\omega}}{_\alpha^\beta}$ for $\Hat{\theta}$
	with respect to $\set{Z_\alpha}$ are expressed in terms of those for $\theta$ by
	\begin{equation}
		\tensor{\Hat{\omega}}{_\alpha^\beta}
		=\tensor{\omega}{_\alpha^\beta}+(u_\alpha\theta^\beta-u^\beta\theta_\alpha)
		+\tensor{\delta}{_\alpha^\beta}u_\gamma\theta^\gamma
		+i(\tensor{u}{^\beta_\alpha}+u_\alpha u^\beta+\tensor{\delta}{_\alpha^\beta}u_\gamma u^\gamma)\theta,
	\end{equation}
	where the indices following $u$ denotes the Tanaka--Webster covariant differentiation with respect to $\theta$.
	Moreover,
	\begin{align}
		\Hat{A}_{\alpha\beta}
		&=A_{\alpha\beta}+iu_{(\alpha\beta)}-iu_\alpha u_\beta
		+iN_{(\alpha\beta)\gamma}u^\gamma,\\
		\Hat{P}_{\alpha{\conj{\beta}}}
		&=P_{\alpha{\conj{\beta}}}
		-\frac{1}{2}(u_{\alpha{\conj{\beta}}}+u_{{\conj{\beta}}\alpha})
		-\frac{1}{2}u_\gamma u^\gamma h_{\alpha{\conj{\beta}}},
	\end{align}
	where the unhatted (resp.\ hatted) quantities are associated with $\theta$ (resp.\ $\Hat{\theta}$).
\end{prop}

It follows from \eqref{eq:first-structure-equation},
\eqref{eq:TW-connection-on-densities}, and \eqref{eq:Tanaka-Webster-curvature-form} that
commutation of covariant derivatives on densities is given as follows,
generalizing Gover--Graham \cite{Gover-Graham-05}*{Proposition 2.2}.

\begin{prop}
	\label{prop:TW-commutation-on-densities}
	For $f\in\mathcal{E}(w,w')$,
	\begin{align}
		\nabla_\alpha\nabla_{\conj{\beta}}f-\nabla_{\conj{\beta}}\nabla_\alpha f
		&=-ih_{\alpha{\conj{\beta}}}\nabla_0f+\frac{w-w'}{n+2}R_{\alpha{\conj{\beta}}}f,\\
		\nabla_\alpha\nabla_\beta f-\nabla_\beta\nabla_\alpha f
		&=-\tensor{N}{^{\conj{\gamma}}_\alpha_\beta}\nabla_{\conj{\gamma}}f
		+\frac{w-w'}{n+2}\cdot 2\tensor{V}{_\gamma^\gamma_\alpha_\beta}f
		=-\tensor{N}{^{\conj{\gamma}}_\alpha_\beta}\nabla_{\conj{\gamma}}f
		-\frac{w-w'}{n+2}(\nabla^\gamma N_{\gamma\alpha\beta})f,\\
		\nabla_\alpha\nabla_0f-\nabla_0\nabla_\alpha f
		&=\tensor{A}{_\alpha^{\conj{\gamma}}}\nabla_{\conj{\gamma}}f
		+\frac{w-w'}{n+2}\tensor{W}{_\gamma^\gamma_\alpha}f
		=\tensor{A}{_\alpha^{\conj{\gamma}}}\nabla_{\conj{\gamma}}f
		+\frac{w-w'}{n+2}(\nabla^\gamma A_{\alpha\gamma}
		-N_{\lambda\mu\alpha}A^{\lambda\mu})f,
	\end{align}
	where $\nabla_0$ is the $\theta$-component of the covariant derivative.
\end{prop}

The following transformation law of the connection on densities
is given by the same formulae as \cite{Gover-Graham-05}*{Proposition 2.3}.

\begin{prop}
	\label{prop:TW-on-densities-for-contact-form-change}
	Let $\theta$ and $\Hat{\theta}=e^u\theta$ be two contact forms.
	Then, for $f\in\mathcal{E}(w,w')$,
	\begin{align}
		\Hat{\nabla}_\alpha f
		&=\nabla_\alpha f+wu_\alpha f,\\
		\Hat{\nabla}_{\conj{\alpha}}f
		&=\nabla_{\conj{\alpha}}f+w'u_{\conj{\alpha}}f,\\
		\Hat{\nabla}_0f
		&=\nabla_0f-iu^\gamma\nabla_\gamma f+iu^{\conj{\gamma}}\nabla_{\conj{\gamma}}f
		+\frac{1}{n+2}((w+w')u_0+iw\tensor{u}{^\gamma_\gamma}-iw'\tensor{u}{_\gamma^\gamma}-i(w-w')u_\gamma u^\gamma)f,
	\end{align}
	where the last equality should be understood as an identity in $\mathcal{E}(w-1,w'-1)$.
	If $\bm{T}$ and $\Hat{\bm{T}}$ denote the weighted Reeb vector field
	for $\theta$ and for $\Hat{\theta}$, respectively,
	then $\nabla_0f$ in the right-hand side (resp.\ $\Hat{\nabla}_0f$ in the left-hand side) means
	$\nabla_{\bm{T}}f$ (resp.\ $\Hat{\nabla}_{\Hat{\bm{T}}}f$).
\end{prop}

\begin{proof}
	We prove the case $(w,w')=(-n-2,0)$, i.e., the case of the canonical bundle $\mathcal{K}$;
	then the general case follows immediately.
	It suffices to consider $\nabla\zeta$ and $\Hat{\nabla}\zeta$ for some nowhere-vanishing
	local section $\zeta$ of $\mathcal{K}$.
	Moreover, without loss of generality, we can use a unitary local frame $\set{Z_\alpha}$ with respect to $\theta$
	to show the formulae.
	So we let $\zeta=\theta\wedge\theta^1\wedge\dots\wedge\theta^n$ by using the dual admissible coframe
	of such a unitary frame $\set{Z_\alpha}$,
	and let $\tensor{\omega}{_\alpha^\beta}$ and $\tensor{\Hat{\omega}}{_\alpha^\beta}$ be as in
	Proposition \ref{prop:TW-for-contact-form-change}. Then we already know that
	\begin{equation}
		\nabla\zeta=-\tensor{\omega}{_\gamma^\gamma}\otimes\zeta.
	\end{equation}
	On the other hand, we can take $\set{e^{-u/2}Z_\alpha}$ as a unitary frame with respect to $\Hat{\theta}$.
	The associated admissible coframe $\set{\tilde{\theta}^\alpha}$ with respect to $\Hat{\theta}$ is given by
	$\tilde{\theta}^\alpha=e^{u/2}\Hat{\theta}^\alpha\equiv e^{u/2}\theta^\alpha$ mod $\theta$,
	and hence $\Hat{\zeta}=\Hat{\theta}\wedge\tilde{\theta}^1\wedge\dots\wedge\tilde{\theta}^n$,
	which equals $e^{(1+n/2)u}\zeta$, satisfies
	\begin{equation}
		\Hat{\nabla}\Hat{\zeta}=-\tensor{\tilde{\omega}}{_\gamma^\gamma}\otimes\Hat{\zeta},
		\qquad\text{which implies}\qquad
		\Hat{\nabla}\zeta=-\tensor{\tilde{\omega}}{_\gamma^\gamma}\otimes\zeta
		-\left(1+\frac{n}{2}\right)du\otimes\zeta,
	\end{equation}
	where $\tensor{\tilde{\omega}}{_\alpha^\beta}$ is the connection forms of $\Hat{\nabla}$
	with respect to $\set{e^{-u/2}Z_\alpha}$.
	Since $\tensor{\tilde{\omega}}{_\alpha^\beta}$ and $\tensor{\Hat{\omega}}{_\alpha^\beta}$ are related by
	$\tensor{\tilde{\omega}}{_\alpha^\beta}=\tensor{\Hat{\omega}}{_\alpha^\beta}-(1/2)\tensor{\delta}{_\alpha^\beta}du$,
	we conclude that
	\begin{equation}
		\Hat{\nabla}\zeta=-\tensor{\Hat{\omega}}{_\gamma^\gamma}\otimes\zeta-du\otimes\zeta.
	\end{equation}
	Then it follows from Proposition \ref{prop:TW-for-contact-form-change} that
	\begin{equation}
		\Hat{\nabla}\zeta
		=\nabla\zeta-((n+2)u_\gamma\theta^\gamma
		+(u_0+i\tensor{u}{^\gamma_\gamma}+i(n+1)u_\gamma u^\gamma)\theta)\otimes\zeta,
	\end{equation}
	which implies the formulae to be shown in view of the fact that
	$\Hat{\bm{T}}=\bm{T}-iu^\gamma Z_\gamma+iu^{\conj{\gamma}}Z_{\conj{\gamma}}$.
\end{proof}

\section{The CR Killing operator}
\label{sec:CR-Killing-operator}

The formula \eqref{eq:CR-Killing-operator} of the CR Killing operator $D$ now makes sense
thanks to various definitions introduced in Section \ref{sec:prelim}.
It can be checked by using
Propositions \ref{prop:TW-for-contact-form-change} and \ref{prop:TW-on-densities-for-contact-form-change}
that the right-hand side of \eqref{eq:CR-Killing-operator} is independent of the choice of a contact form $\theta$.
Moreover we want to note that,
in view of Proposition \ref{prop:TW-commutation-on-densities}, $D$ can also be expressed as
\begin{equation}
	\label{eq:CR-Killing-operator-reexpressed}
	Df=
	(i\nabla_{\alpha}\nabla_{\beta}f
	-A_{\alpha\beta}f
	-iN_{\beta\alpha\gamma}\nabla^\gamma f,
	-i\nabla_{\conj{\alpha}}\nabla_{\conj{\beta}}f
	-A_{\conj{\alpha}\conj{\beta}}f
	+iN_{\conj{\beta}\conj{\alpha}\conj{\gamma}}\nabla^{\conj{\gamma}}f).
\end{equation}

For integrable CR structures (i.e., if $N=0$), the operator $D$ reduces to
\begin{equation}
	Df=
	(i\nabla_{(\alpha}\nabla_{\beta)}f-A_{\alpha\beta}f,
	-i\nabla_{(\conj{\alpha}}\nabla_{\conj{\beta})}f
	-A_{\conj{\alpha}\conj{\beta}}f).
\end{equation}
This operator essentially appears in the literature.
The ``new'' CR differential complex of Akahori--Garfield--Lee \cite{Akahori-Garfield-Lee-02}
contains the mapping
$f\mapsto \nabla_{(\conj{\alpha}}\nabla_{\conj{\beta})}f-iA_{\conj{\alpha}\conj{\beta}}f$
(acting on complex-valued functions) as the first operator, and it is known that it describes
the infinitesimal action of Kuranishi's ``wiggle,''
i.e., internal moves of embedded CR manifolds within a complex manifold.
The same operator is revisited in Hirachi--Marugame--Matsumoto
\cite{Hirachi-Marugame-Matsumoto-17} from the viewpoint of Fefferman's ambient metric construction.

In this section, we first give the derivation of the CR Killing operator $D$
as what describes trivial infinitesimal changes of compatible almost CR structures caused by
contact Hamiltonian vector fields.
Then we discuss the fact that $D$ also appears in the context of asymptotic expansion of
ACHE metrics (asymptotically complex hyperbolic Einstein metrics) and the CR obstruction tensor,
whose basic theory is developed in the author's previous papers \cites{Matsumoto-14,Matsumoto-16}.

\subsection{Trivial infinitesimal deformations}

Let $f\in\Re\mathcal{E}(1,1)$.
Given a fixed contact form $\theta$, $f$ is trivialized in the manner discussed in Section \ref{subsec:basics}
and identified with a smooth real-valued function.
We define the associated \emph{contact Hamiltonian vector field} $X_f$ on $M$ by
\begin{equation}
	\begin{cases}
		\theta(X_f)=f,\\
		d\theta(X_f,Y)=-df(Y),\qquad Y\in H.
	\end{cases}
\end{equation}
The vector field $X_f$ is irrelevant to the choice of $\theta$.
Indeed, if $\Hat{\theta}=e^u\theta$,
then since the corresponding trivialization of the density is given by $\Hat{f}=e^uf$,
we have $\Hat{\theta}(X_f)=e^uf=\Hat{f}$ and
\begin{equation}
	d\Hat{\theta}(X_f,Y)
	=e^u(d\theta+du\wedge\theta)(X_f,Y)
	=-e^u(df(Y)+f\,du(Y))
	=-d\Hat{f}(Y)
	\quad\text{for $Y\in H$}.
\end{equation}
In terms of the Reeb vector field $T$ associated with $\theta$, one can explicitly write
\begin{equation}
	\label{eq:contact-Hamiltonian-vector-field-in-terms-of-Reeb}
	X_f=fT+i(\nabla^\alpha f)Z_\alpha-i(\nabla^{\conj{\alpha}}f)Z_{\conj{\alpha}},
\end{equation}
where the indices $\alpha$ and $\conj{\alpha}$ are raised by the unweighted Levi form.

By Cartan's formula, it follows that the Lie derivative $\mathcal{L}_{X_f}$ satisfies
\begin{equation}
	\label{eq:Lie-derivative-of-theta-by-contact-Hamiltonian-flow}
	\mathcal{L}_{X_f}\theta=d(\theta(X_f))+d\theta(X_f,\cdot)=(Tf)\theta,
\end{equation}
which implies that $X_f$ is a contact vector field.
Moreover, the restriction of $d\theta$ to $H$ rescales conformally, with the same conformal factor
as the one for $\theta$, as we can see by
\begin{equation}
	\label{eq:Lie-derivative-of-d-theta-by-contact-Hamiltonian-flow}
	\mathcal{L}_{X_f}d\theta
	=d(d\theta(X_f,\cdot))=-d(df-(Tf)\theta)\equiv(Tf)d\theta\mod\theta.
\end{equation}

These computations also imply that the compatibility of $J$ is preserved by the flow $\Fl_t$ generated by $X_f$.
Indeed, \eqref{eq:Lie-derivative-of-theta-by-contact-Hamiltonian-flow} and
\eqref{eq:Lie-derivative-of-d-theta-by-contact-Hamiltonian-flow} imply that
$\Fl_t^*(d\theta)\equiv d\theta_t$ mod $\theta$, where $\theta_t=\Fl_t^*\theta$.
Consequently, if we pull back the symmetric form
$d\theta(\mathord{\cdot},J\mathord{\cdot})$ on $H$ by the flow, then we get
$d\theta_t(\mathord{\cdot},J_t\mathord{\cdot})$ on $H$,
where $J_t=(\Fl_t^{-1})_*\circ J\circ(\Fl_t)_*$.
Therefore, the latter bilinear form is symmetric and positive definite, which
implies that $J_t$ satisfies the compatibility condition.
Thus we get a family $J_t$ of compatible almost CR structures on the same underlying contact manifold $(M,H)$.

In general, if $\tilde{J}$ is a compatible almost CR structure on $(M,H)$ sufficiently close to $J$ pointwisely,
then we can define the \emph{deformation tensor} $\varphi$ as the bundle homomorphism $H^{1,0}\to H^{0,1}$ such that
\begin{equation}
	\tilde{T}^{1,0}M=\bigsqcup_{p\in M}\set{Z+\varphi_p(Z)|Z\in H^{1,0}_p}.
\end{equation}
The homomorphism $\varphi$ can be expressed as $\tensor{\varphi}{_\alpha^{\conj{\beta}}}$ in index notation.
We set $\tensor{\varphi}{_{\conj{\alpha}}^\beta}=\conj{\tensor{\varphi}{_\alpha^{\conj{\beta}}}}$
so that $\varphi$ is understood to be a real tensor.

We apply this general definition to our $J_t=(\Fl_t^{-1})_*\circ J\circ(\Fl_t)_*$
and let $\varphi_t$ be the deformation tensor of $J_t$. We set
\begin{equation}
	\psi=\left.\frac{d\varphi_t}{dt}\right|_{t=0},
\end{equation}
i.e., $\psi$ is the ``derivative'' of our family of almost CR structures $J_t$
taken in terms of the deformation tensor.
If we simply regard $J_t$ itself as a tensor and differentiate it, then the derivative at $t=0$,
which is nothing but the Lie derivative of $J$, can be computed as follows.
For $Z\in H^{1,0}$,
\begin{equation}
	J_tZ=J_t((Z+t\psi(Z))-t\psi(Z))=i(Z+t\psi(Z))+it\psi(Z)+O(t^2)=iZ+2it\psi(Z)+O(t^2),
\end{equation}
and hence
\begin{equation}
	\label{eq:deformation-tensor-and-Lie-derivative-of-J}
	\mathcal{L}_{X_f}J
	=\lim_{t\to 0}\frac{J_t-J}{t}
	=-2J\circ\psi.
\end{equation}

It is immediate from the definition that $\psi$ can be seen as an element of
$\Re(\tensor{\mathcal{E}}{_\alpha_\beta}(1,1)\oplus\tensor{\mathcal{E}}{_{\conj{\alpha}}_{\conj{\beta}}}(1,1))$
by lowering the upper index by the weighted Levi form.
In addition, the compatibility of $J_t$ and \eqref{eq:deformation-tensor-and-Lie-derivative-of-J} imply that
$d\theta(\mathord{\cdot},J\circ\psi(\mathord{\cdot}))$ is symmetric on $H$,
by which we can conclude that
\begin{equation}
	\psi\in\Re(\mathcal{E}_{(\alpha\beta)}(1,1)\oplus\mathcal{E}_{(\conj{\alpha}\conj{\beta})}(1,1)).
\end{equation}
The following proposition claims that
the assignment $f\mapsto\psi$ actually equals the operator $D$ defined by \eqref{eq:CR-Killing-operator}.

\begin{prop}
	\label{prop:CR-Killing-operator-as-trivial-deformations}
	The infinitesimal deformation tensor $\psi$
	associated with the contact Hamiltonian vector field $X_f$ is given by $\psi=Df$.
\end{prop}

\begin{proof}
	Let $\set{Z_\alpha}$ be a local frame of $H^{1,0}$.
	Note first that
	\begin{equation}
		(\mathcal{L}_{X_f}J)(Z_\alpha)
		=[X_f,iZ_\alpha]-J[X_f,Z_\alpha]
		=2i[X_f,Z_\alpha]_{0,1},
	\end{equation}
	where the subscript ``0,1'' denotes the projection from $\mathbb{C}H=H^{1,0}\oplus H^{0,1}$
	onto the second summand.
	On the other hand, \eqref{eq:deformation-tensor-and-Lie-derivative-of-J} implies
	$(\mathcal{L}_{X_f}J)(Z_\alpha)=2i\psi(Z_\alpha)$,
	by which we conclude that $\psi(Z_\alpha)=[X_f,Z_\alpha]_{0,1}$, or equivalently,
	\begin{equation}
		\tensor{\psi}{_\alpha^{\conj{\beta}}}=\theta^{\conj{\beta}}([X_f,Z_\alpha]).
	\end{equation}
	The right-hand side can be rewritten as
	\begin{equation}
		\theta^{\conj{\beta}}([X_f,Z_\alpha])
		=-d\theta^{\conj{\beta}}(X_f,Z_\alpha)-Z_\alpha(\theta^{\conj{\beta}}(X_f))
		=-d\theta^{\conj{\beta}}(X_f,Z_\alpha)+iZ_\alpha\nabla^{\conj{\beta}}f,
	\end{equation}
	where the second equality follows by \eqref{eq:contact-Hamiltonian-vector-field-in-terms-of-Reeb}.
	Furthermore, it follows from \eqref{eq:first-structure-equation} and
	\eqref{eq:contact-Hamiltonian-vector-field-in-terms-of-Reeb} that
	\begin{equation}
		\begin{split}
			d\theta^{\conj{\beta}}(X_f,Z_\alpha)
			&=\left(\theta^{\conj{\gamma}}\wedge\tensor{\omega}{_{\conj{\gamma}}^{\conj{\beta}}}
				-\tensor{A}{_\gamma^{\conj{\beta}}}\theta^\gamma\wedge\theta
				+\frac{1}{2}\tensor{N}{^{\conj{\beta}}_\gamma_\sigma}\theta^\gamma\wedge\theta^\sigma\right)
				(X_f,Z_\alpha) \\
			&=\theta^{\conj{\gamma}}(X_f)\tensor{\omega}{_{\conj{\gamma}}^{\conj{\beta}}}(Z_\alpha)
			+\tensor{A}{_\alpha^{\conj{\beta}}}\theta(X_f)
			+\tensor{N}{^{\conj{\beta}}_\gamma_\alpha}\theta^\gamma(X_f) \\
			&=-i\tensor{\omega}{_{\conj{\gamma}}^{\conj{\beta}}}(Z_\alpha)\nabla^{\conj{\gamma}}f
			+\tensor{A}{_\alpha^{\conj{\beta}}}f
			+i\tensor{N}{^{\conj{\beta}}_\gamma_\alpha}\nabla^\gamma f,
		\end{split}
	\end{equation}
	which implies
	$\tensor{\psi}{_\alpha^{\conj{\beta}}}
	=i\nabla_\alpha\nabla^{\conj{\beta}}f-\tensor{A}{_\alpha^{\conj{\beta}}}f
	-i\tensor{N}{^{\conj{\beta}}_\alpha_\gamma}\nabla^\gamma f$.
	This means $\psi=Df$ because of \eqref{eq:CR-Killing-operator-reexpressed}.
\end{proof}

\subsection{As an operator whose adjoint annihilates the CR obstruction tensor}

The CR obstruction tensor $\mathcal{O}_{\alpha\beta}\in\mathcal{E}_{(\alpha\beta)}(-n,-n)$
of a contact manifold $(M,H)$ with a compatible almost CR structure $J$
is introduced in \cite{Matsumoto-14} when $\dim M=2n+1\ge 5$.
Recall that $\mathcal{O}_{\alpha\beta}$ vanishes when $J$ is integrable.
In this paper, we prefer regarding $\mathcal{O}$ as a section of
$\Re(\mathcal{E}_{(\alpha\beta)}(-n,-n)\oplus\mathcal{E}_{({\conj{\alpha}}{\conj{\beta}})}(-n,-n))$
by setting
$\mathcal{O}=(\mathcal{O}_{\alpha\beta},\mathcal{O}_{\conj{\alpha}\conj{\beta}})$,
where $\mathcal{O}_{\conj{\alpha}\conj{\beta}}=\conj{\mathcal{O}_{\alpha\beta}}$.
Then it is known \cite{Matsumoto-14}*{Theorem 1.2 (3)} that
\begin{equation}
	\label{eq:obstruction-tensor-annihilation}
	D^*\mathcal{O}=0,
\end{equation}
where
\begin{equation}
	D^*\colon
	\Re(\mathcal{E}_{(\alpha\beta)}(-n,-n)
	\oplus\mathcal{E}_{(\conj{\alpha}\conj{\beta})}(-n,-n))
	\to\Re\mathcal{E}(-n-2,-n-2)
\end{equation}
is the formal adjoint of the CR Killing operator \eqref{eq:CR-Killing-operator}.

The CR obstruction tensor is an analog of the Fefferman--Graham conformal obstruction tensor
\cites{Fefferman-Graham-85,Fefferman-Graham-12}, which we write $\mathcal{O}_{ij}$,
defined for conformal manifolds of even dimension $\ge 4$.
To define $\mathcal{O}_{ij}$, one considers the asymptotic Dirichlet problem for the Einstein equation
for AH metrics (asymptotically hyperbolic metrics) with Dirichlet data given by a conformal manifold $(M,[g])$;
then $\mathcal{O}_{ij}$ is the obstruction to the existence of formal power series solutions to this problem.
In view of the fact that any AH Einstein metric admits a so-called polyhomogeneous expansion
\cites{Chrusciel-Delay-Lee-Skinner-05,Biquard-Herzlich-14} at boundary at infinity in an appropriate gauge,
one can also say that $\mathcal{O}_{ij}$ is the first logarithmic term coefficient
of the expansion of an AH Einstein metric $g_+$ with prescribed conformal infinity $(M,[g])$,
assuming that such an metric exists
(note that $\mathcal{O}_{ij}$ itself can always be defined for any conformal manifold, no matter $g_+$ exists or not).

In the standard notation in conformal geometry,
$\mathcal{O}_{ij}$ is a symmetric section of the weighted tensor bundle $\mathcal{E}_{ij}[-n+2]$,
and the trace of $\mathcal{O}_{ij}$ always vanishes;
so we may write $\mathcal{O}_{ij}\in\mathcal{E}_{(ij)_0}[-n+2]$.
In 4 dimensions, $\mathcal{O}_{ij}$ is also known as the Bach tensor $B_{ij}$, which is explicitly given by
\begin{equation}
	B_{ij}=\nabla^k C_{ijk}+P^{kl}W_{ikjl},
\end{equation}
where $P$, $C$, and $W$ are the Schouten, the Cotton, and the Weyl tensor, respectively,
of any fixed representative metric $g$ of the conformal class $[g]$
(see \cite{Fefferman-Graham-12} for details).
It is easy to see that $\mathcal{O}_{ij}$ can also be expressed in general dimensions
by a universal formula given in terms of the curvature tensor and its covariant derivatives,
although deriving the concrete formula is quite a hard task.

Likewise, the CR obstruction tensor $\mathcal{O}_{\alpha\beta}$ of $(M,H,J)$ is
the obstruction to the existence of formal power series solutions to the asymptotic Dirichlet problem
for the Einstein equation for ACH metrics with conformal infinity $(M,H,J)$.
If a contact form $\theta$ is fixed, then
$\mathcal{O}_{\alpha\beta}$ can be universally expressed in terms of the Tanaka--Webster local invariants,
although the concrete formula is currently lacking in any dimensions $\ge 5$.

\begin{rem}
	Originally, the definition of the conformal obstruction tensor $\mathcal{O}_{ij}$ is
	given in \cite{Fefferman-Graham-85} using the ambient metrics,
	whose formal expansion theory is basically parallel with that of AH Einstein metrics.
	On the other hand, the ambient metric approach in CR geometry is only partly available.
	For integrable CR structures the notion of ambient metrics exists
	(and it is actually the origin of the notion; see Fefferman \cite{Fefferman-76}),
	but the ambient metric is not yet known for general compatible almost CR structures,
	and the CR obstruction tensor $\mathcal{O}_{\alpha\beta}$
	can be currently only defined via ACH Einstein metrics.
\end{rem}

Equality \eqref{eq:obstruction-tensor-annihilation} follows from the (contracted) second Bianchi identity
satisfied by the Ricci tensor of ACH Einstein metrics.
More precisely, to show \eqref{eq:obstruction-tensor-annihilation},
we use an ACH metric $g_+$ with the prescribed conformal infinity
having a power series asymptotic expansion
such that its Ricci tensor $\Ric(g_+)$ is asymptotic to $\lambda g_+$, where $\lambda$ is a negative constant,
to a certain sufficiently high order.
Then the first non-vanishing coefficient of the power series expansion of $\Ric(g_+)-\lambda g_+$
contains the CR obstruction tensor $\mathcal{O}$,
and the second Bianchi identity implies some equality satisfied by $\mathcal{O}$,
which is \eqref{eq:obstruction-tensor-annihilation}, actually.

If we employ the same procedure in conformal geometry, the equality that we obtain is the
well-known (e.g., \cite{Fefferman-Graham-12}*{Theorem 3.8 (2)})
divergence-freeness of the conformal obstruction tensor
\begin{equation}
	\label{eq:divergence-freeness-in-conformal-geometry}
	\nabla^j\mathcal{O}_{ij}=0.
\end{equation}
The divergence operator here can be naturally understood as the formal adjoint of the conformal Killing operator
\begin{equation}
	K\colon\mathcal{E}_i[2]\to\mathcal{E}_{(ij)_0}[2],\qquad
	v_i\mapsto\nabla_{(i}v_{j)_0}.
\end{equation}
The parallelism of $K$ and $D$ is the reason why we suggest the name ``the CR Killing operator''
for the operator $D$.

Taking advantage of this opportunity, we want to present that
there is another way to show \eqref{eq:divergence-freeness-in-conformal-geometry} in conformal geometry
or \eqref{eq:obstruction-tensor-annihilation} in CR geometry,
which is based on the variational formula of the total integral of the $Q$-curvature
\cites{Graham-Hirachi-05,Matsumoto-16}.
In conformal geometry, we argue as follows.
Recall that the conformal Killing operator $K$ describes the trivial infinitesimal deformation of
conformal structure induced by vector fields in $\mathcal{E}^i\cong\mathcal{E}_i[2]$.
Consequently, it follows from the variational formula that
\begin{equation}
	0=\int_M\braket{Kv,\mathcal{O}}=\int_M\braket{v,K^*\mathcal{O}}
	\qquad\text{for any $v=v_i\in\mathcal{E}_i[2]$},
\end{equation}
which implies that $K^*\mathcal{O}=0$.
The same argument works in CR geometry as well, because the CR Killing operator $D$ describes
trivial infinitesimal deformations of compatible almost CR structure as we saw in
Proposition \ref{prop:CR-Killing-operator-as-trivial-deformations}.

\section{Parabolic geometries, tractor connections, BGG operators}
\label{sec:general-theory-on-parabolic-geometry}

In this section, we summarize some general notions and results from the theory of parabolic geometries
that we use in the subsequent developments.
For more background and details, see the extensive reference of \v{C}ap--Slov\'ak \cite{Cap-Slovak-09}
and, for various specific aspects,
Yamaguchi \cite{Yamaguchi-93}, Sharpe \cite{Sharpe-97}, Calderbank--Diemer \cite{Calderbank-Diemer-01},
for example.

\subsection{Cartan geometries}

Let $G$ be a Lie group and $H$ a closed subgroup.
Their Lie algebras are denoted by $\mathfrak{g}$ and $\mathfrak{h}$, respectively.
Then a \emph{Cartan geometry} of type $(G,H)$ is a smooth manifold $M$ of the same dimension as
$G/H$ equipped with a principal $H$-bundle $\mathcal{G}\to M$
and an $H$-equivariant $\mathfrak{g}$-valued 1-form $\omega$ on $\mathcal{G}$ satisfying the following conditions:
\begin{enumerate}[(i)]
	\item $\omega(\zeta_X)=X$ for each $X\in\mathfrak{h}$, where $\zeta_X$ is the fundamental vector field;
	\item $\omega_u\colon T_u\mathcal{G}\to\mathfrak{g}$ is a linear isomorphism for all $u\in\mathcal{G}$.
\end{enumerate}
The 1-form $\omega$ is called a \emph{Cartan connection}.
Note that $\omega_u$ induces an isomorphism $T_{\pi(u)}M\cong\mathfrak{g}/\mathfrak{h}$ for each $u\in\mathcal{G}$,
where $\pi$ is the projection $\mathcal{G}\to M$, and thus
$TM$ is identified with the associated vector bundle $\mathcal{G}\times_H\mathfrak{g}/\mathfrak{h}$.

The \emph{curvature form} of a Cartan geometry $(\mathcal{G},\omega)$ is
the $\mathfrak{g}$-valued 2-form $K$ on $\mathcal{G}$ defined by
\begin{equation}
	K=d\omega+\frac{1}{2}[\omega\wedge\omega].
\end{equation}
An equivalent notion is
the \emph{curvature function} $\kappa\colon\mathcal{G}\to\bigwedge^2\mathfrak{g}^*\otimes\mathfrak{g}$,
which is defined by
\begin{equation}
	\kappa(u)(X,Y)=K_u(\omega_u^{-1}(X),\omega_u^{-1}(Y)),
	\qquad X,\,Y\in\mathfrak{g}.
\end{equation}
It follows from the definition that $K$ is a horizontal $H$-equivariant 2-form, and hence
$\kappa$ can also be seen as a function with values in
$\bigwedge^2(\mathfrak{g}/\mathfrak{h})^*\otimes\mathfrak{g}$.

The \emph{torsion form} of a Cartan geometry $(\mathcal{G},\omega)$ is defined to be
$\pi_\mathfrak{g}\circ K$, where $\pi_\mathfrak{g}\colon\mathfrak{g}\to\mathfrak{g}/\mathfrak{h}$ is
the natural projection.
The geometry $(\mathcal{G},\omega)$ is called \emph{torsion-free} when the torsion form vanishes,
or equivalently, the curvature form $K$ takes values in $\mathfrak{h}$.

\subsection{Parabolic geometries}

Now let $G$ be a (real or complex) semisimple Lie group, and suppose that its Lie algebra $\mathfrak{g}$ is given
a $\abs{k}$-grading
\begin{equation}
	\label{eq:k-grading}
	\mathfrak{g}=\mathfrak{g}_{-k}\oplus\mathfrak{g}_{-k+1}\oplus\dots\oplus\mathfrak{g}_{k-1}\oplus\mathfrak{g}_k,
\end{equation}
i.e., a direct sum decomposition of $\mathfrak{g}$ such that
$[\mathfrak{g}_i,\mathfrak{g}_j]\subset\mathfrak{g}_{i+j}$ and
$\mathfrak{g}_-=\mathfrak{g}_{-k}\oplus\dots\oplus\mathfrak{g}_{-1}$ is generated by $\mathfrak{g}_{-1}$.
It is known that the subalgebra $\mathfrak{g}_0$ is reductive,
and there exists a unique element $E$ in the center of $\mathfrak{g}_0$ such that $\ad(E)=i\id$ on $\mathfrak{g}_i$
for all $i$, which is called the \emph{grading element} (see \cite{Yamaguchi-93}*{Section 3}).

Associated with the above $\abs{k}$-grading is the filtration
\begin{equation}
	\label{eq:k-grading-filtration}
	\mathfrak{g}=\mathfrak{g}^{-k}\supset\mathfrak{g}^{-k+1}\supset\dots\supset\mathfrak{g}^{k-1}
	\supset\mathfrak{g}^k\supset\set{0},
\end{equation}
where $\mathfrak{g}^i=\bigoplus_{i\le j}\mathfrak{g}_j$.
In particular, $\mathfrak{g}^0$ is also denoted by $\mathfrak{p}$.
A closed subgroup $P$ of $G$ is called a \emph{parabolic subgroup} admitted by the $\abs{k}$-grading
\eqref{eq:k-grading} if its Lie algebra equals $\mathfrak{p}$ and
the adjoint action of any $p\in P$ preserves each $\mathfrak{g}^i$.

\begin{dfn}
	A \emph{parabolic geometry} is a Cartan geometry of type $(G,P)$,
	where $G$ is a semisimple Lie group and $P$ is its parabolic subgroup.
\end{dfn}

The \emph{Levi subgroup} $G_0\subset P$ is defined by
\begin{equation}
	G_0=\set{p\in P|\text{$\Ad(p)(\mathfrak{g}_i)\subset\mathfrak{g}_i$ for all $i$}}.
\end{equation}
In view of the fact that the elements $X\in\mathfrak{g}_0$ are characterized by the property
$\ad(X)(\mathfrak{g}_i)\subset\mathfrak{g}_i$ for all $i$ \cite{Cap-Slovak-09}*{Lemma 3.1.3},
it follows that the Lie algebra of $G_0$ equals $\mathfrak{g}_0$.

Among the $P$-modules $\mathfrak{g}^j$ in the filtration \eqref{eq:k-grading-filtration},
not only $\mathfrak{p}=\mathfrak{g}^0$ but also
$\mathfrak{p}_+=\mathfrak{g}^1$ is particularly important for us, because
there is an isomorphism $\mathfrak{p}_+\cong(\mathfrak{g}/\mathfrak{p})^*$ of $P$-modules
induced by the Killing form of $\mathfrak{g}$.
On the other hand, $\mathfrak{g}/\mathfrak{p}$ is isomorphic to $\mathfrak{g}_-$ as a $G_0$-module,
but $\mathfrak{g}_-$ does not carry a natural $P$-module structure.
It is known that the mapping $G_0\times\mathfrak{p}_+\to P$ defined by $(g_0,Z)\mapsto g_0e^Z$
is a diffeomorphism \cite{Cap-Slovak-09}*{Theorem 3.1.3},
and consequently, if we define $P_+=\exp\mathfrak{p}_+$,
then $P$ is the semidirect product $G_0\ltimes P_+$.

Recall that the curvature function $\kappa$ of a parabolic geometry $(\mathcal{G},\omega)$ takes values in
$\bigwedge^2(\mathfrak{g}/\mathfrak{p})^*\otimes\mathfrak{g}$.
As a $P$-module, $\bigwedge^2(\mathfrak{g}/\mathfrak{p})^*\otimes\mathfrak{g}$
is isomorphic to $\bigwedge^2\mathfrak{p}_+\otimes\mathfrak{g}$,
and this is the chain space $C_2(\mathfrak{p}_+,\mathfrak{g})$ in the Lie algebra homology complex
for $\mathfrak{p}_+$ with values in $\mathfrak{g}$.
Let $\mathbb{W}$ be a $\mathfrak{g}$-module in general
and consider the chain space $C_k(\mathfrak{p}_+,\mathbb{W})=\bigwedge^k\mathfrak{p}_+\otimes\mathbb{W}$.
Then the boundary operator, which is traditionally denoted by $\partial^*$ and called the
\emph{Kostant codifferential}, acting on $C_k(\mathfrak{p}_+,\mathbb{W})$ is defined by
the following formula of the action on decomposable elements, where
$Z_1$, ..., $Z_k\in\mathfrak{p}_+$, $w\in\mathbb{W}$, and $\check{\phantom{a}}$ indicates the omission:
\begin{equation}
	\label{eq:Kostant-codifferential}
	\begin{multlined}
		\partial^*(Z_1\wedge\dots\wedge Z_k\otimes w)
		=\sum_i (-1)^{i+1}Z_1\wedge\dots\wedge\check{Z}_i\wedge\dots\wedge Z_k\otimes(Z_i\cdot w)\\
		+\sum_{i<j} (-1)^{i+j}
		[Z_i,Z_j]\wedge Z_1\wedge\dots\wedge\check{Z}_i\wedge\dots\wedge\check{Z}_j\wedge\dots\wedge Z_k\otimes w.
	\end{multlined}
\end{equation}
Also, by identifying $C_k$ with $\bigwedge^k\mathfrak{g}_-^*\otimes\mathbb{W}$ as a $G_0$-module,
we define $\partial\colon C_k\to C_{k+1}$ by
\begin{equation}
	\label{eq:Lie-algebra-cohomology-coboundary-operator}
	\begin{multlined}
		\partial\varphi(X_1,\dots,X_{k+1})
		=\sum_i (-1)^{i+1}X_i\cdot\varphi(X_1,\dots,\check{X}_i,\dots,X_{k+1})\\
		+\sum_{i<j} (-1)^{i+j}
		\varphi([X_i,X_j],X_1,\dots,\check{X}_i,\dots,\check{X}_j,\dots,X_{k+1}),
	\end{multlined}
\end{equation}
where $X_0$, ..., $X_k\in\mathfrak{g}_-$.
The Lie algebra homology $H_k(\mathfrak{p}_+,\mathbb{W})$ and cohomology $H^k(\mathfrak{g}_-,\mathbb{W})$
are defined by
\begin{equation}
	H_k(\mathfrak{p}_+,\mathbb{W})
	:=\frac{\ker(\partial^*\colon C_k\to C_{k-1})}{\im(\partial^*\colon C_{k+1}\to C_k)}\qquad\text{and}\qquad
	H^k(\mathfrak{g}_-,\mathbb{W})
	:=\frac{\ker(\partial\colon C_k\to C_{k+1})}{\im(\partial\colon C_{k-1}\to C_k)},
\end{equation}
where we understand that $C_{-1}=0$ as usual.

\begin{dfn}
	A parabolic geometry $(\mathcal{G},\omega)$ is said to be \emph{normal} if its curvature function satisfies
	$\partial^*\kappa=0$.
\end{dfn}

When $\mathbb{W}=\mathfrak{g}$,
the chain space $C_k(\mathfrak{p}_+,\mathfrak{g})$ carries a natural $G_0$-invariant grading,
with respect to which $C_k(\mathfrak{p}_+,\mathfrak{g})_l$ is the set of elements in
$C_k(\mathfrak{p}_+,\mathfrak{g})=\bigwedge^k\mathfrak{p}_+\otimes\mathfrak{g}$
that are expressible as a sum of elements in
$\mathfrak{g}_{i_1}\wedge\dots\wedge\mathfrak{g}_{i_k}\otimes\mathfrak{g}_j$ with $i_1+\dots+i_k+j=l$;
elements of $C_k(\mathfrak{p}_+,\mathfrak{g})_l$ are said to be of \emph{homogeneity $l$}.
The associated $P$-invariant filtration is defined by setting
$C_k(\mathfrak{p}_+,\mathfrak{g})^l=\bigoplus_{l\le m}C_k(\mathfrak{p}_+,\mathfrak{g})_m$.

Similarly, any $\mathfrak{g}$-module $\mathbb{W}$ admits a $\mathfrak{g}_0$-invariant eigendecomposition
with respect to the action of the grading element $E$ corresponding to real eigenvalues,
and hence the set of homogeneity $l$ elements and
the $\mathfrak{p}$-invariant filtration of $C_k(\mathfrak{p}_+,\mathbb{W})$ are naturally defined.
These are further inherited by the homology $H_k(\mathfrak{p}_+,\mathbb{W})$ because
\eqref{eq:Kostant-codifferential} shows that $\partial^*$ preserves the homogeneity degree.
Also, the notion of homogeneity $l$ elements of $H^k(\mathfrak{g}_-,\mathbb{W})$ makes sense
because \eqref{eq:Lie-algebra-cohomology-coboundary-operator} implies that $\partial$ preserve
the homogeneity degree.

\subsection{Tractor bundles and connections}
\label{subsec:tractor-bundles-and-connections}

Let $(\mathcal{G},\omega)$ be a Cartan geometry of type $(G,H)$.
As $\mathcal{G}\to M$ is a principal $H$-bundle, any $H$-module gives rise to the associated vector bundle
over the space $M$.
If $\mathbb{W}$ is moreover a $G$-module,
then the associated vector bundle $\mathcal{W}=\mathcal{G}\times_H\mathbb{W}$, which is called
the \emph{tractor bundle}, naturally carries a linear connection.

We describe this in a slightly generalized situation.
Suppose that $\mathbb{W}$ is a \emph{$(\mathfrak{g},H)$-module}, i.e.,
that $\mathbb{W}$ is a $\mathfrak{g}$-module
together with an action of $H$ whose induced action of $\mathfrak{h}$ equals
the restriction of the action of $\mathfrak{g}$.
Then we can still define the associated tractor bundle by $\mathcal{W}=\mathcal{G}\times_H\mathbb{W}$.
Recall that sections of $\mathcal{W}$ are identified with
elements of $\Gamma(\mathcal{G},\mathbb{W})^H$, i.e., $\mathbb{W}$-valued functions on $\mathcal{G}$
that are $H$-equivariant.
Then a natural linear connection of $\mathcal{W}$ can be defined as follows.

\begin{dfn}
	(1) For any $H$-module $\mathbb{W}$, the \emph{fundamental derivative} on
	$\mathcal{W}=\mathcal{G}\times_H\mathbb{W}$ is defined by
	$D_\xi s=\omega^{-1}(\xi)s$ for any section $s$ of $\mathcal{W}$ and $\xi\in\mathfrak{g}$.

	(2) For any $(\mathfrak{g},H)$-module $\mathbb{W}$, we define
	$\nabla_\xi s=D_\xi s+\xi\cdot s$ for any section $s$ of
	the associated tractor bundle $\mathcal{W}=\mathcal{G}\times_H\mathbb{W}$ and $\xi\in\mathfrak{g}$.
	Then $\nabla s$ gives an element of
	$\Gamma(\mathcal{G},(\mathfrak{g}/\mathfrak{h})^*\otimes\mathbb{W})^H$, and thus $\nabla$ can be
	understood as a linear connection of $\mathcal{W}$, which is called the \emph{tractor connection}.
\end{dfn}

Now consider a parabolic geometry $(\mathcal{G},\omega)$ of type $(G,P)$,
with associated $\abs{k}$-grading \eqref{eq:k-grading} of $\mathfrak{g}$.
Then any $(\mathfrak{g},P)$-module $\mathbb{W}$ carries, as discussed in the previous subsection,
a natural $\mathfrak{g}_0$-invariant eigendecomposition with respect to the grading element $E$,
which is actually $G_0$-invariant because $E$ is fixed by the adjoint action of any element of $G_0$.
As a consequence, a $P$-invariant filtration of $C_k(\mathfrak{p}_+,\mathbb{W})$ is induced,
and the bundle $\Omega^k\otimes\mathcal{W}$ of $k$-forms with values in the tractor bundle $\mathcal{W}$
inherits a filtration.

Then it follows that $\nabla$ is \emph{of homogeneity $0$} in the sense that
sections of $\mathcal{W}^l$ are mapped by $\nabla$ to sections of $(\Omega^1\otimes\mathcal{W})^l$.
Passing to the associated graded bundles,
the induced mapping $\gr\nabla\colon\Gamma(\gr_l(\mathcal{W}))\to\Gamma(\gr_l(\Omega^1\otimes\mathcal{W}))$
is in fact algebraic, and equals $\gr\partial$.
This is also the case for differential forms of higher degree
if we assume that $(\mathcal{G},\omega)$ is a regular parabolic geometry (see the next subsection):
the covariant exterior derivative $d^\nabla$ maps sections of $(\Omega^k\otimes\mathcal{W})^l$
to sections of $(\Omega^{k+1}\otimes\mathcal{W})^l$, and $\gr d^\nabla=\gr\partial$.

A case of particular interest is when $\mathbb{W}=\mathfrak{g}$.
In this case, the associated tractor bundle, which we write $\mathcal{A}M$,
is called the \emph{adjoint tractor bundle} and plays an important role.
Note that there is a canonical projection $\Pi\colon\mathcal{A}M\to TM$ induced by
the quotient map $\mathfrak{g}\to\mathfrak{g}/\mathfrak{p}$.

The curvature function $\kappa$ of a parabolic geometry $(\mathcal{G},\omega)$, which takes values in
$\bigwedge^2(\mathfrak{g}/\mathfrak{p})^*\otimes\mathfrak{g}$, is clearly $P$-equivariant,
so it can be interpreted as a 2-form on $M$ with values in $\mathcal{A}M$.
We can use this fact to introduce another linear connection $\tilde{\nabla}$ of $\mathcal{A}M$ by
\begin{equation}
	\label{eq:modified-tractor-connection}
	\tilde{\nabla}s=\nabla s+\iota_{\Pi(s)}\kappa,
\end{equation}
where $s$ is a section of $\mathcal{A}M$ and $\Pi(s)$ denotes, by abuse of notation, the vector field $\Pi\circ s$.

\begin{dfn}
	We call $\tilde{\nabla}$ defined by \eqref{eq:modified-tractor-connection}
	the \emph{modified adjoint tractor connection}.
\end{dfn}

\v{C}ap \cite{Cap-08} introduced this connection in relation with deformations of parabolic geometries.
In fact, trivial infinitesimal deformations of a parabolic geometry $(\mathcal{G},\omega)$
(or of any Cartan geometry) are given in the form $\tilde{\nabla}s$,
as discussed in \cite{Cap-08}*{Sections 3.1 and 3.2} and briefly summarized as follows.
A section $s$ of the adjoint tractor bundle $\mathcal{A}M$
corresponds to a $P$-invariant vector field $\omega^{-1}(s)$ on $\mathcal{G}$, or in other words,
an infinitesimal principal bundle automorphism of $\mathcal{G}$.
The induced infinitesimal change $\mathcal{L}_{\omega^{-1}(s)}\omega$ of the Cartan connection
is a horizontal $P$-equivariant $\mathfrak{g}$-valued 1-form on $\mathcal{G}$ and hence
is an element of $\Omega^1(M,\mathcal{A}M)$, and actually it is given by $\tilde{\nabla}s$.
In addition, it is also known by \cite{Cap-08} that the
covariant exterior differentiation $d^{\tilde{\nabla}}\colon\Omega^1(M,\mathcal{A}M)\to\Omega^2(M,\mathcal{A}M)$
assigns to an infinitesimal modification of the Cartan connection
the induced infinitesimal change of the curvature.

\subsection{Infinitesimal flag structures and regularity}
\label{subsec:infinitesimal-flag-structures}

For a parabolic geometry $(\mathcal{G},\omega)$ of type $(G,P)$, with Levi subgroup $G_0$ of $P$,
we define the principal $G_0$-bundle $\mathcal{G}_0\to M$ by $\mathcal{G}_0=\mathcal{G}/P_+$,
which we call the \emph{graded frame bundle} of $(\mathcal{G},\omega)$.

Note that, since $\omega_u$ defines a linear isomorphism $T_u\mathcal{G}\to\mathfrak{g}$ for each $u\in\mathcal{G}$,
the filtration \eqref{eq:k-grading-filtration} of $\mathfrak{g}$ induces the filtration
\begin{equation}
	\label{eq:filtration-of-tangent-bundle-of-Cartan-geometry}
	T\mathcal{G}=T^{-k}\mathcal{G}\supset T^{-k+1}\mathcal{G}\supset\dots\supset T^0\mathcal{G}
	=V\mathcal{G},
\end{equation}
where $T^i_u\mathcal{G}$ is the preimage of $\mathfrak{g}^i$ by $\omega_u$ and
$V\mathcal{G}$ is the vertical bundle for the projection $\mathcal{G}\to M$.
Then \eqref{eq:filtration-of-tangent-bundle-of-Cartan-geometry} naturally induces the filtration
\begin{equation}
	\label{eq:filtration-of-frame-bundle}
	T\mathcal{G}_0=T^{-k}\mathcal{G}_0\supset T^{-k+1}\mathcal{G}_0\supset\dots\supset T^0\mathcal{G}_0
	=V\mathcal{G}_0
\end{equation}
of $T\mathcal{G}_0$, and furthermore, the one
\begin{equation}
	\label{eq:filtration-of-tangent-bundle}
	TM=T^{-k}M\supset T^{-k+1}M\supset\dots\supset T^0M=0
\end{equation}
of the tangent bundle $TM$ of the base manifold $M$.

One can show that, for $-k\le i\le -1$, the $\mathfrak{g}_i$-component $\omega_i$ of the Cartan connection $\omega$
restricted to $T^i\mathcal{G}$
descends to a $G_0$-equivariant section of $(T^i\mathcal{G}_0)^*\otimes\mathfrak{g}_i$
over $\mathcal{G}_0$ \cite{Cap-Slovak-09}*{Proposition 3.1.5}, which we write $\theta_i$.
Thus $\mathcal{G}_0$ comes with a collection $(\theta_{-k},\dots,\theta_{-1})$ of partially defined
$\mathfrak{g}_i$-valued 1-forms with each $\theta_i$ satisfying $\ker\theta_i=T^{i+1}\mathcal{G}_0$.
The triple $((T^iM)_{i=-k}^{-1},\mathcal{G}_0,(\theta_{-k},\dots,\theta_{-1}))$ is
called the \emph{induced infinitesimal flag structure} of a parabolic geometry $(\mathcal{G},\omega)$.
Abstractly, infinitesimal flag structures are defined as follows.

\begin{dfn}[\cite{Cap-Slovak-09}*{Definition 3.1.6}]
	\label{dfn:infinitesimal-flag-structure}
	An \emph{infinitesimal flag structure} of type $(G,P)$ on a smooth manifold $M$
	is a triple $((T^iM)_{i=-k}^{-1},\mathcal{G}_0,(\theta_{-k},\dots,\theta_{-1}))$, where:
	\begin{enumerate}[(i)]
		\item
			\label{item:infinitesimal-flag-structure-filtration-of-tangent-bundle}
			$(T^iM)_{i=-k}^{-1}$ denotes a filtration
			$TM=T^{-k}M\supset T^{-k+1}M\supset\dots\supset T^{-1}M$
			satisfying $\rank T^iM=\dim(\mathfrak{g}^i/\mathfrak{p})$;
		\item
			$\mathcal{G}_0$ is a principal $G_0$-bundle over $M$;
		\item
			$\theta_i$ is a $G_0$-equivariant section of $(T^i\mathcal{G}_0)^*\otimes\mathfrak{g}_i$,
			where $T^i\mathcal{G}_0$ is the preimage of $T^iM$, such that $\ker\theta_i=T^{i+1}\mathcal{G}_0$.
	\end{enumerate}
\end{dfn}

There is a reinterpretation of infinitesimal flag structures in terms of
a reduction of the structure group of the full graded frame bundle $\mathcal{F}_{\gr}M$
of the graded tangent bundle $\gr(TM)$.
Note that the adjoint action of $G_0$ preserves the grading \eqref{eq:k-grading} and thus
defines a group homomorphism $\Ad\colon G_0\to\GL_{\gr}(\mathfrak{g}_-)$, with respect to which
we reduce the structure group\footnote{The homomorphism $\Ad\colon G_0\to\GL_{\gr}(\mathfrak{g}_-)$
need not be injective.}.

\begin{prop}[\cite{Cap-Slovak-09}*{Proposition 3.1.6}]
	\label{prop:infinitesimal-flag-structure-as-structure-group-reduction}
	An infinitesimal flag structure of type $(G,P)$ on a manifold $M$ is equivalent to the pair of
	a filtration $(T^iM)_{i=-k}^{-1}$ of $TM$
	as in Definition \ref{dfn:infinitesimal-flag-structure}
	\ref{item:infinitesimal-flag-structure-filtration-of-tangent-bundle}
	and a reduction of the structure group of $\mathcal{F}_{\gr}M$
	to $G_0$ with respect to the homomorphism $\Ad\colon G_0\to\GL_{\gr}(\mathfrak{g}_-)$.
\end{prop}

The notion of infinitesimal flag structures should be supplemented by an important notion of regularity.
An infinitesimal flag structure $((T^iM)_{i=-k}^{-1},\mathcal{G}_0,(\theta_{-k},\dots,\theta_{-1}))$
is said to be \emph{regular} if the filtration $(T^iM)_{i=-k}^{-1}$ is compatible with the Lie bracket
and the induced bundle map
\begin{equation}
	\gr_i(TM)\times\gr_j(TM)\to\gr_{i+j}(TM)
\end{equation}
coincides with the map induced by the Lie algebra bracket $\mathfrak{g}_i\times\mathfrak{g}_j\to\mathfrak{g}_{i+j}$
(see \cite{Cap-Slovak-09}*{Section 3.1.7} for details).
A parabolic geometry is called \emph{regular} if it induces a regular infinitesimal flag structure.

The following fact is crucial when we consider the BGG construction for
the modified adjoint tractor connection in the next subsection.

\begin{prop}[\cite{Cap-Slovak-09}*{Corollary 3.1.8 (2)}]
	\label{prop:characterization-of-regularity-by-curvature}
	A parabolic geometry $(\mathcal{G},\omega)$ is regular if and only if
	its curvature function $\kappa$ is of homogeneity $1$, i.e.,
	seen as the $\mathcal{A}M$-valued 2-form, $\kappa$ is a section of $(\Omega^2\otimes\mathcal{A}M)^1$.
\end{prop}

Giving an infinitesimal flag structure is a universal way of describing 
certain kind of geometry of the base space $M$.
For example, if $G=\PSU(n+1,1)$ and $P$ is defined to be the stabilizer of a null complex line
in the complex Minkowski space $\mathbb{C}^{n+1,1}$, then
infinitesimal flag structures of type $(G,P)$ on a $(2n+1)$-dimensional smooth manifold $M$
are in one-to-one correspondence with compatible almost CR structures on $M$, as discussed in detail in
Section \ref{subsec:frame-bundles-in-CR-geometry}.

The following theorem, which is basically due to works of Tanaka in 1960--70s,
provides a generalization in terms of parabolic geometries
of the construction of normal Cartan connection associated with compatible almost CR structures
carried out by Tanaka \cite{Tanaka-62} and Chern--Moser \cite{Chern-Moser-74} (in the integrable case).

\begin{thm}[cf.\ \v{C}ap--Slov\'ak \cite{Cap-Slovak-09}*{Sections 3.1.13--14}]
	\label{thm:existence-of-normal-cartan-connection}
	Let $((T^iM)_{i=-k}^{-1},\mathcal{G}_0,(\theta_{-k},\dots,\theta_{-1}))$ be
	a regular infinitesimal flag structure of type $(G,P)$,
	and suppose that $H^1(\mathfrak{g}_-,\mathfrak{g})$ is concentrated in non-positive homogeneity degrees
	and $\mathfrak{g}_0$ carries no simple ideals of $\mathfrak{g}$.
	Then, there exists a normal regular parabolic geometry $(\mathcal{G},\omega)$ that induces
	$((T^iM)_{i=-k}^{-1},\mathcal{G}_0,(\theta_{-k},\dots,\theta_{-1}))$, which is unique
	up to the action of principal $P$-bundle isomorphisms inducing the identity on
	the underlying infinitesimal structure.
\end{thm}

The fact that CR geometry satisfies the assumption of the above theorem is well-known and
checked later in Section \ref{subsec:homology-with-values-in-adjoint}.
In fact, it suffices to compute the Lie algebra homology $H_1(\mathfrak{p}_+,\mathfrak{g})$ to
verify the cohomological assumption, since it is known that $H^1(\mathfrak{g}_-,\mathfrak{g})$ is
isomorphic to $H_1(\mathfrak{p}_+,\mathfrak{g})$ as a $G_0$-module.

\subsection{BGG operators}

As an application of the theory of parabolic geometries,
we can construct certain invariant linear differential operators in a systematic way.
An implementation of such techniques is given by the ``curved'' BGG construction
of \v{C}ap--Slov\'ak--Sou\v{c}ek \cite{Cap-Slovak-Soucek-01},
which is then simplified by Calderbank--Diemer \cite{Calderbank-Diemer-01}.
It starts with the tractor connection $\nabla$
of any tractor bundle $\mathcal{W}=\mathcal{G}\times_P\mathbb{W}$ of a parabolic geometry $(\mathcal{G},\omega)$
and ends up with a sequence of ``BGG operators'' $D_k$, $k\ge 0$,
which forms a complex when $(\mathcal{G},\omega)$ is locally flat (i.e., locally homogeneous).
In fact, the same construction can also be triggered for some modified connections of $\mathcal{W}$,
including the modified adjoint tractor connection $\tilde{\nabla}$
discussed in Section \ref{subsec:tractor-bundles-and-connections},
as pointed out by \v{C}ap \cite{Cap-08} and sketched below.

Let $(\mathcal{G},\omega)$ be a regular parabolic geometry of type $(G,P)$,
$\mathcal{W}$ the tractor bundle associated with a $(\mathfrak{g},P)$-module $\mathbb{W}$,
and let $\nabla$ be the tractor connection (unmodified, for the time being).
Then the following twisted de Rham sequence can be considered,
where $d^\nabla$ is the covariant exterior differentiation:
\begin{equation}
	\label{eq:twisted-deRham-sequence}
	\begin{tikzcd}
		\mathcal{W}\ar[r,"\nabla"] &
		\Omega^1\otimes\mathcal{W}\ar[r,"d^\nabla"] &
		\Omega^2\otimes\mathcal{W}\ar[r,"d^\nabla"] & \cdots.
	\end{tikzcd}
\end{equation}
The BGG operators are introduced to extend the sequence \eqref{eq:twisted-deRham-sequence}
into a certain commutative diagram.
Recall that the Lie algebra homology $H_k(\mathfrak{p}_+,\mathbb{W})$ is defined by
\begin{equation}
	H_k(\mathfrak{p}_+,\mathbb{W})
	=\frac{\ker(\partial^*\colon C_k\to C_{k-1})}{\im(\partial^*\colon C_{k+1}\to C_k)},
\end{equation}
where $C_k=C_k(\mathfrak{p}_+,\mathbb{W})=\bigwedge^k\mathfrak{p}_+\otimes\mathbb{W}$.
Note that, in addition to the chain spaces $C_k$, we have another two families of $P$-modules,
namely $Z_k=\ker\partial^*\subset C_k$ and $H_k$.
The vector bundles associated with $C_k$ are $\Omega^k\otimes\mathcal{W}$,
and we set $\mathcal{Z}_k:=\mathcal{G}\times_P Z_k$ and $\mathcal{H}_k:=\mathcal{G}\times_P H_k$.
Then, the diagram that extends \eqref{eq:twisted-deRham-sequence} will be
\begin{equation}
	\begin{tikzcd}
		\mathcal{W}\ar[r,"\nabla"] &
		\Omega^1\otimes\mathcal{W}\ar[r,"d^\nabla"] &
		\Omega^2\otimes\mathcal{W}\ar[r,"d^\nabla"] & \cdots \\
		\mathcal{Z}_0\ar[u,equal]\ar[d,"\proj"] &
		\mathcal{Z}_1\ar[u,hookrightarrow]\ar[d,"\proj"] &
		\mathcal{Z}_2\ar[u,hookrightarrow]\ar[d,"\proj"] \\
		\mathcal{H}_0\ar[r,"D_0"] &
		\mathcal{H}_1\ar[r,"D_1"] &
		\mathcal{H}_2\ar[r,"D_2"] & \cdots.
	\end{tikzcd}
\end{equation}

The definition of the \emph{BGG operators} $D_k$ is given by
\begin{equation}
	\label{eq:definition-of-BGG-operators}
	D_k=\proj\circ d^\nabla\circ L_k
\end{equation}
where $L_k\colon\mathcal{H}_k\to\mathcal{Z}_k$ is some differential splitting operator.
In order that \eqref{eq:definition-of-BGG-operators} may make sense, it is necessary that the image of
$d^\nabla\circ L_k$ is contained in $\mathcal{Z}_{k+1}$.
General constructions of such $L_k$ are given by \cite{Cap-Slovak-Soucek-01} and \cite{Calderbank-Diemer-01}.
However, we do not need to use those constructions of $L_k$ because it is known that
the requirement we just mentioned also gives the characterization of $L_k$,
as pointed out by \v{C}ap \cite{Cap-05}.
Hammerl--Somberg--Sou\v{c}ek--\v{S}ilhan \cite{Hammerl-Somberg-Soucek-Silhan-12}*{Theorem 3.1}
summarizes this as follows.

\begin{prop}
	\label{prop:construction-of-BGG-splitting-operator}
	Let $E_k$ be a differential operator from
	$\Omega^k\otimes\mathcal{W}$ to $\Omega^{k+1}\otimes\mathcal{W}$
	of homogeneity $0$
	with the property that the associated graded map coincides with $\gr\partial$.
	Then for every $\sigma\in\mathcal{H}_k$,
	there exists a unique element $s\in\mathcal{Z}_k$ with $\sigma=[s]$ and $E_ks\in\mathcal{Z}_{k+1}$.
	Moreover, the mapping $L_k\colon\sigma\mapsto s$ is given by a differential operator.
\end{prop}

The definition of the BGG operators $D_k$ is completed by applying
Proposition \ref{prop:construction-of-BGG-splitting-operator} to the operator $E_k=d^\nabla$.

It is sometimes the case that
the tractor connection $\nabla$ happens to be a prolongation of the associated first BGG operator $D_0$,
which basically means that the system of differential equations $\nabla s=0$, where $s\in\mathcal{W}$,
can be recovered by
extending the system $D_0f=0$, where $f\in\mathcal{H}_0$, by adjoining its trivial consequences.
In this case, it follows that $\ker D_0$ is in one-to-one correspondence with $\ker\nabla$,
which often provides some geometric insight toward the space under study,
because in many settings $D_0$ admits a down-to-earth geometric interpretation.
Such examples and non-examples are overviewed by
Hammerl--Somberg--Sou\v{c}ek--\v{S}ilhan \cite{Hammerl-Somberg-Soucek-Silhan-12}*{Section 5}.

The same relationship can also happen for other differential operators
$E_0\colon\mathcal{W}\to\Omega^1\otimes\mathcal{W}$
satisfying the assumption of Proposition \ref{prop:construction-of-BGG-splitting-operator}
and its associated first BGG operator.
In particular, the following fact regarding the first BGG operator $D_0^{\tilde{\nabla}}$
associated with the modified adjoint tractor connection $\tilde{\nabla}$,
which is well-defined by Proposition \ref{prop:characterization-of-regularity-by-curvature}
and which we call the \emph{modified first adjoint BGG operator},
is proved in
\cite{Hammerl-Somberg-Soucek-Silhan-12}*{Section 5.4} (or Hammerl \cite{Hammerl-09-Thesis}*{Section 4.3}).

\begin{thm}
	\label{thm:modified-adjoint-tractor-connection-prolongates}
	For any normal regular parabolic geometry $(\mathcal{G},\omega)$,
	the modified adjoint tractor connection $\tilde{\nabla}$ is
	always a prolongation of the associated first BGG operator $D_0^{\tilde{\nabla}}$.
\end{thm}

Now let us turn to CR geometry, and recall from Section \ref{sec:CR-Killing-operator}
that the CR Killing operator
$D\colon\Re\mathcal{E}(1,1)\to\Re(\mathcal{E}_{(\alpha\beta)}(1,1)\oplus
\mathcal{E}_{({\conj{\alpha}}{\conj{\beta}})}(1,1))$ is
an operator describing trivial deformations of compatible almost CR structures
induced by contact Hamiltonian vector fields.
In particular, $\ker D$ is the space of infinitesimal symmetries.
On the other hand, Theorem \ref{thm:modified-adjoint-tractor-connection-prolongates} implies
that the modified first adjoint BGG operator $D_0^{\tilde{\nabla}}$ associated with the
CR normal regular Cartan geometry $(\mathcal{G},\omega)$ shares the same property that
$\ker D_0^{\tilde{\nabla}}$ is the space of infinitesimal symmetries---although the precise meaning of symmetries is
different, namely, given in terms of infinitesimal principal bundle automorphisms of $\mathcal{G}$.
But in any case, this observation strongly suggests that $D_0^{\tilde{\nabla}}$ is equal to $D$
(perhaps up to some nonzero constant factor),
which is indeed the case as we shall see in Section \ref{sec:computation-of-CR-BGG-operators}
(without any factor, thanks to our normalization used in \eqref{eq:CR-Killing-operator}).

Another point of interest is whether $D_0^{\tilde{\nabla}}$ is actually different from
the original (i.e., unmodified) first adjoint BGG operator $D_0^\nabla$.
The following theorem regarding this matter is due to \v{C}ap \cite{Cap-08}*{3.5 Theorem}.

\begin{thm}
	\label{thm:first-adjoint-BGG-operators-agree}
	If $(\mathcal{G},\omega)$ is a torsion-free normal regular parabolic geometry
	and the homology $H_1(\mathfrak{p}_+,\mathfrak{g})$ is concentrated in non-positive homogeneity degrees,
	then $D_0^{\tilde{\nabla}}=D_0^\nabla$.
\end{thm}

However, in CR geometry, the normal regular parabolic geometry $(\mathcal{G},\omega)$
associated with a compatible almost CR manifold $(M,H,J)$ by Theorem \ref{thm:existence-of-normal-cartan-connection}
is torsion-free only if $(M,H,J)$ is integrable,
as we shall see in Section \ref{sec:construction-of-normal-tractor-connections}.
Therefore, in the non-integrable case,
Theorem \ref{thm:first-adjoint-BGG-operators-agree} does not guarantee that $D_0^{\tilde{\nabla}}=D_0^\nabla$.
Our conclusion in Section \ref{sec:computation-of-CR-BGG-operators} is going to be
that they are actually not the same.

\section{Lie algebra homology relevant to CR geometry}
\label{sec:Lie-algebra-homology-for-CR-geometry}

Hereafter $\mathfrak{g}$ denotes $\mathfrak{su}(n+1,1)$,
the special unitary algebra of indefinite signature $(n+1,1)$,
and we consider the parabolic subalgebra $\mathfrak{p}$ associated with
a certain $\abs{2}$-grading of $\mathfrak{g}$ (see \eqref{eq:2-grading}),
which corresponds to geometry of compatible almost CR structures.
In order to specialize the general theory outlined in the previous section to this case,
in this section we express the Lie algebra homology $H_k(\mathfrak{p}_+,\mathbb{W})$ and
the relevant Kostant codifferential $\partial^*$ in explicit terms when $\mathbb{W}$ is
the standard representation $\mathbb{V}=\mathbb{C}^{n+1,1}$ or the adjoint representation $\mathfrak{g}$,
and for $k=0$, $1$.

It suffices to treat the both representations simply as $\mathfrak{g}$-modules in this section,
but as group representations we later need to consider $\mathbb{V}$ as a $G^\sharp$-module,
while $\mathfrak{g}$ can be treated as a $G$-module,
where $G^\sharp=\SU(n+1,1)$ and $G=\PSU(n+1,1)$.
In this article, parabolic geometries or infinitesimal flag structures are said to be
\emph{corresponding to geometry of compatible almost CR structures} or simply \emph{of CR type}
when they are of type $(G,P)$ or $(G^\sharp,P^\sharp)$, where
$P$ and $P^\sharp$ are the parabolic subgroups defined in Section \ref{subsec:frame-bundles-in-CR-geometry}.

\subsection{Preliminaries}

We regard $\mathfrak{g}=\mathfrak{su}(n+1,1)$ as the Lie algebra of the special unitary group
associated with the indefinite Hermitian inner product of $\mathbb{C}^{n+2}$ given by
\begin{equation}
	\label{eq:indefinite-Hermitian-inner-product}
	\begin{pmatrix}
		& & 1 \\
		& I & \\
		1 & &
	\end{pmatrix},
\end{equation}
where $I$ is the $n\times n$ identity matrix. It is easily seen that
\begin{equation}
	\mathfrak{g}=\Set{
		\begin{pmatrix}
			a & Z & iz \\
			X & A & -Z^* \\
			ix & -X^* & -\conj{a}
		\end{pmatrix}
	|
		\Centerstack{%
			{$x$, $z\in\mathbb{R}$, $X\in\mathbb{C}^n$, $Z\in(\mathbb{C}^n)^*$,}
			{$A\in\mathfrak{u}(n)$, $a\in\mathbb{C}$, $a+\tr A-\conj{a}=0$}%
		}
	},
\end{equation}
where $\mathbb{C}^n$ and $(\mathbb{C}^n)^*$ denote the set of column vectors and that of row vectors, respectively,
and $*$ on the upper right of $X^*$ and $Z^*$ indicates the conjugate transpose.
With index notation, we can also express elements of $\mathfrak{g}$ as
\begin{equation}
	\label{eq:general-elements-of-g}
	\begin{pmatrix}
		a & Z_\beta & iz \\
		X^\alpha & \tensor{A}{_\beta^\alpha} & -Z^\alpha \\
		ix & -X_\beta & -\conj{a}
	\end{pmatrix}
\end{equation}
by setting
\begin{equation}
	X^{\conj{\alpha}}=\conj{X^\alpha},\qquad
	Z_{\conj{\beta}}=\conj{Z_\beta},\qquad
	X_\beta=h_{\beta{\conj{\alpha}}}X^{\conj{\alpha}},\qquad
	Z^\alpha=h^{\alpha{\conj{\beta}}}Z_{\conj{\beta}},
\end{equation}
where
\begin{equation}
	\label{eq:inner-product-for-Lie-algebra-q}
	h_{\alpha{\conj{\beta}}}=
	h^{\alpha{\conj{\beta}}}=
	\begin{cases}
		1, & \alpha=\beta, \\
		0, & \text{otherwise}.
	\end{cases}
\end{equation}

We also use the complexification $\tilde{\mathfrak{g}}$ of $\mathfrak{g}$.
Note that $\tilde{\mathfrak{g}}$ equals $\mathfrak{sl}(n+2,\mathbb{C})$, the set of complex trace-free matrices,
with complex conjugation relative to the real form $\mathfrak{g}$.
That is, if we express a general element of $\tilde{\mathfrak{g}}$ as
\begin{equation}
	\label{eq:general-elements-of-g-complexified}
	\begin{pmatrix}
		a & Z_\beta & iz \\
		X^\alpha & \tensor{A}{_\beta^\alpha} & -W^\alpha \\
		ix & -Y_\beta & -b
	\end{pmatrix},
\end{equation}
where $x$, $z\in\mathbb{C}$, $X$, $W\in\mathbb{C}^n$, $Y$, $Z\in(\mathbb{C}^n)^*$,
$A\in\mathfrak{gl}(n,\mathbb{C})$, $a$, $b\in\mathbb{C}$ with $a+\tr A-b=0$, then
\begin{equation}
	\overline{
		\begin{pmatrix}
			a & Z_\beta & iz \\
			X^\alpha & \tensor{A}{_\beta^\alpha} & -W^\alpha \\
			ix & -Y_\beta & -b
		\end{pmatrix}
	}=
	\begin{pmatrix}
		\conj{b} & W_\beta & i\conj{z} \\
		Y^\alpha & -\tensor{A}{^\alpha_\beta} & -Z^\alpha \\
		i\conj{x} & -X_\beta & -{\conj{a}}
	\end{pmatrix},
\end{equation}
where
\begin{equation}
	X^{\conj{\alpha}}=\conj{X^\alpha},\qquad
	Y_{\conj{\beta}}=\conj{Y_\beta},\qquad
	Z_{\conj{\beta}}=\conj{Z_\beta},\qquad
	W^{\conj{\alpha}}=\conj{W^\alpha},\qquad
	\tensor{A}{_{\conj{\beta}}^{\conj{\alpha}}}=\conj{\tensor{A}{_\beta^\alpha}}
\end{equation}
and indices are lowered/raised by \eqref{eq:inner-product-for-Lie-algebra-q}.

We endow the Lie algebra $\mathfrak{g}$ with the $\abs{2}$-grading
\begin{equation}
	\label{eq:2-grading}
	\mathfrak{g}=\mathfrak{g}_{-2}\oplus\mathfrak{g}_{-1}\oplus\mathfrak{g}_0\oplus\mathfrak{g}_1\oplus\mathfrak{g}_2
\end{equation}
given by the grading element
\begin{equation}
	\label{eq:grading-element}
	E=
	\begin{pmatrix}
		1 & & \\
		& O & \\
		& & -1
	\end{pmatrix}.
\end{equation}
The subscript $i$ of $\mathfrak{g}_i$, which is the eigenvalue of $\ad E$ on $\mathfrak{g}_i$,
equals the distance from the diagonal of the (possibly) non-zero
entries of its elements in the block form \eqref{eq:general-elements-of-g}. That is,
\begin{gather}
	\begin{pmatrix}
		\phantom{M} & \phantom{M} & \phantom{M} \\
		& & \\
		ix & &
	\end{pmatrix},\quad
	\begin{pmatrix}
		\phantom{M} & \phantom{M} & \phantom{M} \\
		X^\alpha & & \\
		& -X_\beta &
	\end{pmatrix},\quad
	\begin{pmatrix}
		a & & \phantom{M} \\
		  & \tensor{A}{_\beta^\alpha} & \\
		\phantom{M} & & -\conj{a}
	\end{pmatrix},\\
	\begin{pmatrix}
		& Z_\beta & \\
		& & -Z^\alpha \\
		\phantom{M} & \phantom{M} & \phantom{M}
	\end{pmatrix},\quad
	\begin{pmatrix}
		& & iz \\
		& & \\
		\phantom{M} & \phantom{M} & \phantom{M}
	\end{pmatrix}
\end{gather}
are the elements of $\mathfrak{g}_{-2}$, $\mathfrak{g}_{-1}$, $\mathfrak{g}_0$, $\mathfrak{g}_1$, and $\mathfrak{g}_2$,
respectively.

Following the general notation used in Section \ref{sec:general-theory-on-parabolic-geometry}, we write
\begin{equation}
	\mathfrak{g}_-=\mathfrak{g}_{-2}\oplus\mathfrak{g}_{-1},\qquad
	\mathfrak{p}=\mathfrak{g}_0\oplus\mathfrak{g}_1\oplus\mathfrak{g}_2,\qquad
	\mathfrak{p}_+=\mathfrak{g}_1\oplus\mathfrak{g}_2.
\end{equation}
The complexification of the $\abs{2}$-grading is expressed as
$\tilde{\mathfrak{g}}
=\tilde{\mathfrak{g}}_{-2}\oplus\tilde{\mathfrak{g}}_{-1}
\oplus\tilde{\mathfrak{g}}_0\oplus\tilde{\mathfrak{g}}_1\oplus\tilde{\mathfrak{g}}_2$,
and $\tilde{\mathfrak{g}}_-$, $\tilde{\mathfrak{p}}$, $\tilde{\mathfrak{p}}_+$ denote the complexification of
their untilded counterparts.

We introduce the basis $\xi_0$, $\xi_1$, $\dotsc$, $\xi_n$, $\xi_{\conj{1}}$, $\dotsc$, $\xi_{\conj{n}}$
of $\tilde{\mathfrak{g}}_-$ by
\begin{equation}
	\label{eq:basis-negative}
	\xi_0=
	\begin{pmatrix}
		\phantom{M} & & \\
		& \phantom{M} & \\
		i & & \phantom{M}
	\end{pmatrix},\qquad
	\xi_\sigma=
	\begin{pmatrix}
		\phantom{M} & & \\
		\tensor{\delta}{_\sigma^\alpha} & \phantom{M} & \\
		& 0 & \phantom{M}
	\end{pmatrix},\qquad
	\xi_{\conj{\sigma}}=
	\begin{pmatrix}
		\phantom{M} & & \\
		0 & \phantom{M} & \\
		& -h_{\beta{\conj{\sigma}}} & \phantom{M}
	\end{pmatrix},
\end{equation}
and the basis $\zeta^0$, $\zeta^1$, $\dotsc$, $\zeta^n$, $\zeta^{\conj{1}}$, $\dotsc$, $\zeta^{\conj{n}}$
of $\tilde{\mathfrak{p}}_+$ by
\begin{equation}
	\label{eq:basis-positive}
	\zeta^0=
	\begin{pmatrix}
		\phantom{M} & & i \\
		& \phantom{M} & \\
		& & \phantom{M}
	\end{pmatrix},\qquad
	\zeta^\sigma=
	\begin{pmatrix}
		\phantom{M} & \tensor{\delta}{_\beta^\sigma} & \\
		& \phantom{M} & 0 \\
		& & \phantom{M}
	\end{pmatrix},\qquad
	\zeta^{\conj{\sigma}}=
	\begin{pmatrix}
		\phantom{M} & 0 & \\
		& \phantom{M} & -h^{\alpha{\conj{\sigma}}} \\
		& & \phantom{M}
	\end{pmatrix}.
\end{equation}
Note that $\xi_0$ and $\zeta^0$ are real elements, which belong to $\mathfrak{g}_{-2}$ and $\mathfrak{g}_2$,
respectively.
The vectors $\xi_\sigma$, $\xi_{\conj{\sigma}}\in\tilde{\mathfrak{g}}_{-1}$ are the complex conjugates of each other,
and so are $\zeta^\sigma$, $\zeta^{\conj{\sigma}}\in\tilde{\mathfrak{g}}_1$.

\begin{lem}
	\label{lem:g0-contains-no-simple-ideals-of-g}
	$\mathfrak{g}_0$ contains no nontrivial ideals of $\mathfrak{g}$.
\end{lem}

\begin{proof}
	Suppose that $I$ is an ideal of $\mathfrak{g}$ contained in $\mathfrak{g}_0$.
	Then, since \eqref{eq:2-grading} is a grading, any element
	\begin{equation}
		F=
		\begin{pmatrix}
			a & & \phantom{M} \\
			  & \tensor{A}{_\beta^\alpha} & \\
			\phantom{M} & & -\conj{a}
		\end{pmatrix}
		\in I
	\end{equation}
	must satisfy $[F,\xi_0]=[F,\xi_\sigma]=0$.
	Then $\Re a=0$ and $\tensor{A}{_\beta^\alpha}-a\tensor{\delta}{_\beta^\alpha}=0$ follow,
	and consequently, $F$ is of the form
	\begin{equation}
		F=
		\begin{pmatrix}
			ic & & \phantom{M} \\
			  & ic\tensor{\delta}{_\beta^\alpha} & \\
			\phantom{M} & & ic
		\end{pmatrix},
		\qquad c\in\mathbb{R}.
	\end{equation}
	However, since $F\in\mathfrak{g}_0$ we have $c=0$, and hence $I=0$.
\end{proof}

Since \eqref{eq:2-grading} is a grading,
the Killing form $B$ of $\tilde{\mathfrak{g}}$ vanishes on
$\tilde{\mathfrak{g}}_i\times\tilde{\mathfrak{g}}_j$ unless $i+j=0$.
It can be easily seen that, on $\tilde{\mathfrak{g}}_{-2}\times\tilde{\mathfrak{g}}_2$ and
on $\tilde{\mathfrak{g}}_{-1}\times\tilde{\mathfrak{g}}_1$, we have
\begin{equation}
	B(\xi_0,\zeta^0)=-2(n+2)
\end{equation}
and
\begin{equation}
	B(\xi_\sigma,\zeta^\tau)=2(n+2)\tensor{\delta}{_\sigma^\tau},\qquad
	B(\xi_{\conj{\sigma}},\zeta^{\conj{\tau}})=2(n+2)\tensor{\delta}{_{\conj{\sigma}}^{\conj{\tau}}},\qquad
	B(\xi_\sigma,\zeta^{\conj{\tau}})=B(\xi_{\conj{\sigma}},\zeta^\tau)=0.
\end{equation}
Consequently, the dual basis of $\tilde{\mathfrak{p}}_+$ to
the basis $\xi_0$, $\xi_\sigma$, $\xi_{\conj{\sigma}}$ of $\tilde{\mathfrak{g}}_-$
with respect to $B$ is given by
\begin{equation}
	\label{eq:dual-basis-positive}
	\xi^*_0=-\frac{1}{2(n+2)}\zeta^0,\qquad
	\xi^*_\sigma=\frac{1}{2(n+2)}\zeta^\sigma,\qquad
	\xi^*_{\conj{\sigma}}=\frac{1}{2(n+2)}\zeta^{\conj{\sigma}}.
\end{equation}

The dual basis \eqref{eq:dual-basis-positive} can be used in computations of the action of the Kostant codifferential
\eqref{eq:Kostant-codifferential} on $C_k(\mathfrak{p}_+,\mathbb{W})$,
identified with $\bigwedge^k(\mathfrak{g}/\mathfrak{p})^*\otimes\mathbb{W}$ by the Killing form $B$, as follows.
For $k=1$, we have for $\psi\in(\mathfrak{g}/\mathfrak{p})^*\otimes\mathbb{W}$
\begin{equation}
	\label{eq:Kostant-codifferential-in-duality-1}
	\partial^*\psi=\sum_A\xi_A^*\cdot\psi(\xi_A),
\end{equation}
where $A$ runs through $\set{0,1,\dots,n,\conj{1},\dots,\conj{n}}$.
Indeed, \eqref{eq:Kostant-codifferential-in-duality-1} follows from
$\partial^*(\xi_A^*\otimes w)=\xi_A^*\cdot w$, which is correct by \eqref{eq:Kostant-codifferential}.
Similarly, for $\phi\in\bigwedge^2(\mathfrak{g}/\mathfrak{p})^*\otimes\mathbb{W}$ one has
\begin{equation}
	\label{eq:Kostant-codifferential-in-duality-2}
	\partial^*\phi(X)=2\sum_A[\xi_A^*,\phi(X,\xi_A)]-\sum_A\phi([\xi_A^*,X],\xi_A),\qquad
	X\in\mathfrak{g}/\mathfrak{p}.
\end{equation}
This formula appears in \cite{Cap-Slovak-09}*{Lemma 3.1.11}.

\subsection{Filtration of the standard and the adjoint representations}

Corresponding to the block form \eqref{eq:general-elements-of-g} of elements of $\mathfrak{g}$,
those of the standard representation $\mathbb{V}=\mathbb{C}^{n+2}$ of $\mathfrak{g}$ can be expressed as
\begin{equation}
	\label{eq:block-form-of-standard-representation}
	\begin{pmatrix} s \\ t^\alpha \\ u \end{pmatrix},
\end{equation}
which corresponds to the eigendecomposition
\begin{equation}
	\mathbb{V}=\mathbb{V}_1\oplus\mathbb{V}_0\oplus\mathbb{V}_{-1}
\end{equation}
induced by the grading element \eqref{eq:grading-element}.
The associated $\mathfrak{p}$-invariant filtration is
\begin{equation}
	\mathbb{V}=\mathbb{V}^{-1}\supset\mathbb{V}^0\supset\mathbb{V}^1,
\end{equation}
where
\begin{equation}
	\mathbb{V}^0=\mathbb{V}_0\oplus\mathbb{V}_1=\Set{\begin{pmatrix} s \\ t^\alpha \\ 0 \end{pmatrix}},\qquad
	\mathbb{V}^1=\mathbb{V}_1=\Set{\begin{pmatrix} s \\ 0 \\ 0 \end{pmatrix}}.
\end{equation}
Note that $\mathbb{V}$ carries the $\mathfrak{g}$-invariant indefinite Hermitian inner product
\begin{equation}
	\label{eq:standard-representation-pairing}
	\Braket{
		\begin{pmatrix}
			s \\ t^\alpha \\ u
		\end{pmatrix}
		,
		\begin{pmatrix}
			s' \\ t'^\alpha \\ u'
		\end{pmatrix}
	}
	=s\conj{u}'+\sum_{\alpha=1}^n t^\alpha\conj{t'^\alpha}+u\conj{s}'.
\end{equation}

The $\mathfrak{p}$-module $\mathbb{V}^1=\mathbb{V}_1$ is denoted by $E(-1,0)$,
in view of the fact that its associated tractor bundle can be identified
with the density bundle $\mathcal{E}(-1,0)$,
as we see in Section \ref{sec:construction-of-normal-tractor-connections}.
We set
\begin{equation}
	E(1,0)=E(-1,0)^*,\qquad
	E(0,1)=\conj{E(1,0)}
\end{equation}
and
\begin{equation}
	E(w,w')=E(1,0)^{\otimes w}\otimes E(0,1)^{\otimes w'}.
\end{equation}
The inner product \eqref{eq:standard-representation-pairing}
induces a duality between $\gr_{-1}(\mathbb{V})=\mathbb{V}/\mathbb{V}^0$ and $\smash{\conj{\mathbb{V}}}^1$.
Consequently,
\begin{equation}
	\label{eq:representation-E01-as-bottom-component}
	\gr_{-1}(\mathbb{V})\cong E(0,1),
\end{equation}
which means that the bottom component $u$ of \eqref{eq:block-form-of-standard-representation},
which is regarded as representing an equivalence class in $\mathbb{V}/\mathbb{V}^0$ in this context,
can be regarded as an element of $E(0,1)$.

Next, we define the $\mathfrak{p}$-module $E^\alpha$ to be $\gr_0(\mathbb{V})\otimes E(1,0)$. Then
\begin{equation}
	\gr_0(\mathbb{V})\cong E^\alpha(-1,0),
\end{equation}
where we write $E^\alpha(w,w')=E^\alpha\otimes E(w,w')$.
This means that the second component $t^\alpha$ of
\begin{equation}
	\label{eq:representation-Ea-10-as-middle-component}
	\begin{pmatrix} s \\ t^\alpha \\ 0 \end{pmatrix},
\end{equation}
representing an equivalence class in $\mathbb{V}^0/\mathbb{V}^1$,
can be regarded as an element of $E^\alpha(-1,0)$.
The inner product \eqref{eq:standard-representation-pairing}
naturally induces a positive-definite Hermitian inner product $\gr_0(\mathbb{V})\times\gr_0(\mathbb{V})\to\mathbb{C}$,
and hence a canonical complex bilinear mapping
\begin{equation}
	\label{eq:gr-minus-1-pairing}
	E^\alpha\times E^{\conj{\beta}}\to E(1,1),
\end{equation}
which we write
\begin{equation}
	\label{eq:Levi-form-at-representation-level}
	h_{\alpha{\conj{\beta}}}\in E_{\alpha{\conj{\beta}}}(1,1).
\end{equation}
Indeed, this pairing will induce the weighted Levi form of compatible almost CR structures.

The complexified adjoint representation $\tilde{\mathfrak{g}}=\mathfrak{sl}(n+2)$ is
isomorphic to the trace-free part of $\End(\mathbb{V})=\mathbb{V}^*\otimes\mathbb{V}$,
and the associated graded module of $\End(\mathbb{V})$ is decomposed as
\begin{equation}
	\begin{pmatrix}
		E & E_\beta & E(-1,-1) \\
		E^\alpha & \tensor{E}{_\beta^\alpha} & E^\alpha(-1,-1) \\
		E(1,1) & E_\beta(1,1) & E
	\end{pmatrix},
\end{equation}
where $E$ is the trivial representation.
Consequently, we have the isomorphisms of $\mathfrak{p}$-modules
\begin{align}
	\gr_2(\tilde{\mathfrak{g}})
	&=\tilde{\mathfrak{g}}^2 \cong E(-1,-1), \\
	\gr_1(\tilde{\mathfrak{g}})
	&=\tilde{\mathfrak{g}}^1/\tilde{\mathfrak{g}}^2 \cong E_\beta\oplus E^\alpha(-1,-1)
	\cong E_\beta\oplus E_{\conj{\alpha}}, \\
	\gr_0(\tilde{\mathfrak{g}})
	&=\tilde{\mathfrak{g}}^0/\tilde{\mathfrak{g}}^1 \cong E\oplus E\oplus \tf\tensor{E}{_\beta^\alpha}, \\
	\gr_{-1}(\tilde{\mathfrak{g}})
	&=\tilde{\mathfrak{g}}^{-1}/\tilde{\mathfrak{g}}^0 \cong E^\alpha\oplus E_\beta(1,1)
	\cong E^\alpha\oplus E^{\conj{\beta}},\\
	\gr_{-2}(\tilde{\mathfrak{g}})&=\tilde{\mathfrak{g}}/\tilde{\mathfrak{g}}^{-1} \cong E(1,1).
\end{align}
In particular, if the vectors in \eqref{eq:basis-negative} are understood as representing
equivalence classes in $\tilde{\mathfrak{g}}/\tilde{\mathfrak{p}}$,
then the first (resp.\ the second) summand of
$\gr_{-1}(\tilde{\mathfrak{g}})=\tilde{\mathfrak{g}}^{-1}/\tilde{\mathfrak{p}}=E^\alpha\oplus E^{\conj{\beta}}$
is spanned by $\xi_\sigma$ (resp.\ by $\xi_{\conj{\sigma}}$).

\subsection{Homology of $\mathfrak{p}_+$ with values in the standard representation}
\label{subsec:homology-with-values-in-standard}

Based on the block form \eqref{eq:block-form-of-standard-representation}
for the standard representation $\mathbb{V}$, elements of $C_k(\mathfrak{p}_+,\mathbb{V})$,
which are identified with elements of $\bigwedge^k(\mathfrak{g}/\mathfrak{p})^*\otimes\mathbb{V}$, will be
denoted as
\begin{equation}
	\begin{pmatrix}
		\tensor{s}{_{A_1}_{\cdots}_{A_k}} \\
		\tensor{t}{^\alpha_{A_1}_{\cdots}_{A_k}} \\
		\tensor{u}{_{A_1}_{\cdots}_{A_k}}
	\end{pmatrix}
\end{equation}
using the basis \eqref{eq:basis-negative},
where each $A_i$ is an index running through $\set{0,1,\dots,n,\conj{1},\dots,\conj{n}}$.
Sometimes the indices $A_1\cdots A_k$ in this notation will be suppressed,
in which case we will use care so that any confusion does not occur.

Then, for an element of $C_1(\mathfrak{p}_+,\mathbb{V})$,
\eqref{eq:Kostant-codifferential-in-duality-1} implies that the action of $\partial^*$ is given by
\begin{equation}
	\partial^*
	\begin{pmatrix}
		s \\ t^\alpha \\ u
	\end{pmatrix}
	=
	\frac{1}{2(n+2)}
	\begin{pmatrix}
		-iu_0+\tensor{t}{^\gamma_\gamma} \\
		-u^\alpha \\
		0
	\end{pmatrix}.
\end{equation}
For an element of $C_2(\mathfrak{p}_+,\mathbb{V})$, we have from \eqref{eq:Kostant-codifferential-in-duality-2} that
\begin{align}
	\left(
		\partial^*
		\begin{pmatrix}
			s \\ t^\alpha \\ u
		\end{pmatrix}
	\right)(\xi_0)
	&=
	\frac{1}{2(n+2)}
	\begin{pmatrix}
		\tensor{t}{^\gamma_0_\gamma}+i\tensor{s}{_\gamma^\gamma} \\
		-\tensor{u}{_0^\alpha}+i\tensor{t}{^\alpha_\gamma^\gamma} \\
		i\tensor{u}{_\gamma^\gamma}
	\end{pmatrix}, \\
	\left(
		\partial^*
		\begin{pmatrix}
			s \\ t^\alpha \\ u
		\end{pmatrix}
	\right)(\xi_\sigma)
	&=
	\frac{1}{2(n+2)}
	\begin{pmatrix}
		i\tensor{u}{_0_\sigma}+\tensor{t}{^\gamma_\sigma_\gamma} \\
		-\tensor{u}{_\sigma^\alpha} \\
		0
	\end{pmatrix}, \\
	\left(
		\partial^*
		\begin{pmatrix}
			s \\ t^\alpha \\ u
		\end{pmatrix}
	\right)(\xi_{\conj{\sigma}})
	&=
	\frac{1}{2(n+2)}
	\begin{pmatrix}
		i\tensor{u}{_0_{\conj{\sigma}}}+\tensor{t}{^\gamma_{\conj{\sigma}}_\gamma} \\
		-\tensor{u}{_{\conj{\sigma}}^\alpha} \\
		0
	\end{pmatrix}.
\end{align}

These computations can be used to identify the homology groups
$H_0(\mathfrak{p}_+,\mathbb{V})$ and $H_1(\mathfrak{p}_+,\mathbb{V})$ as follows.
The zeroth homology is given by
\begin{equation}
	H_0(\mathfrak{p}_+,\mathbb{V})
	=\mathbb{V}/\im\partial^*_1
	=\mathbb{V}/\mathbb{V}^0
	\cong E(0,1),
\end{equation}
where the last identification is nothing but \eqref{eq:representation-E01-as-bottom-component}.
As mentioned after \eqref{eq:representation-E01-as-bottom-component}, we can regard
the bottom component $u$ of the vector \eqref{eq:block-form-of-standard-representation}
in the standard representation $\mathbb{V}$ as an element of $E(0,1)$ by identifying $u$ with
the equivalence class of the vector \eqref{eq:block-form-of-standard-representation}.
Using this interpretation, we can express
the projection $\mathbb{V}\to H_0(\mathfrak{p}_+,\mathbb{V})\cong E(0,1)$ as,
in fact tautologically,
\begin{equation}
	\begin{pmatrix}
		s \\ t^\alpha \\ u
	\end{pmatrix}
	\mapsto u.
\end{equation}
The first homology is
\begin{equation}
	\begin{split}
		H_1(\mathfrak{p}_+,\mathbb{V})
		&=\frac{\ker\partial^*_1}{\im\partial^*_2}
		=\frac{\Set{
			\begin{pmatrix}
				\tensor{s}{_A} \\
				\tensor{t}{^\alpha_A} \\
				\tensor{u}{_A}
			\end{pmatrix} |
			u_{\conj{\sigma}}=0,\, u_0=-i\tensor{t}{^\gamma_\gamma}
		}}{\Set{
			\begin{pmatrix}
				\tensor{s}{_A} \\
				\tensor{t}{^\alpha_A} \\
				\tensor{u}{_A}
			\end{pmatrix} |
			u_\sigma=u_{\conj{\sigma}}=0,\,u_0=-i\tensor{t}{^\gamma_\gamma},\,
			t_{\conj{\sigma}\conj{\tau}}=t_{[\conj{\sigma}\conj{\tau}]}
		}} \\
		&\cong\set{(u_\sigma)}
		\oplus
		\set{(\tensor{t}{^\alpha_{\conj{\beta}}})
		|t_{\conj{\alpha}\conj{\beta}}=t_{\conj{\beta}\conj{\alpha}}}
		\cong E_\sigma(0,1)\oplus E_{(\conj{\alpha}\conj{\beta})}(0,1),
	\end{split}
\end{equation}
where the projections from $\ker\partial_1^*$ onto
$E_\sigma(0,1)$ and onto $E_{(\conj{\alpha}\conj{\beta})}(0,1)$ are given by
\begin{equation}
	\begin{pmatrix}
		\tensor{s}{_A} \\
		\tensor{t}{^\alpha_A} \\
		\tensor{u}{_A}
	\end{pmatrix}
	\mapsto (u_\sigma)
	\qquad\text{and}\qquad
	\begin{pmatrix}
		\tensor{s}{_A} \\
		\tensor{t}{^\alpha_A} \\
		\tensor{u}{_A}
	\end{pmatrix}
	\mapsto (t_{(\conj{\alpha}\conj{\beta})}).
\end{equation}
To understand the latter, i.e.,
to see that the component $(t_{(\conj{\alpha}\conj{\beta})})$ of any element of $\ker\partial^*_1$
can be regarded as an element of $E_{(\conj{\alpha}\conj{\beta})}(0,1)$,
one must note that $u_{\conj{\sigma}}=0$.
By this, and by recalling \eqref{eq:representation-Ea-10-as-middle-component},
we can regard the middle component of the $A=\conj{\sigma}$ part
\begin{equation}
	\begin{pmatrix}
		\tensor{s}{_{\conj{\sigma}}} \\
		\tensor{t}{^\alpha_{\conj{\sigma}}} \\
		0
	\end{pmatrix}
\end{equation}
of any element of $\ker\partial^*_1$ as giving an element of $\tensor{E}{^\alpha_{\conj{\sigma}}}(-1,0)$.
Then, by lowering an index using \eqref{eq:Levi-form-at-representation-level} and taking the symmetric part,
we arrive at the conclusion that $(t_{(\conj{\alpha}\conj{\beta})})$ is interpretable as
an element of $E_{(\conj{\alpha}\conj{\beta})}(0,1)$.

\subsection{Homology of $\mathfrak{p}_+$ with values in the adjoint representation}
\label{subsec:homology-with-values-in-adjoint}

Following the manner used in the previous subsection,
elements of $C_k(\mathfrak{p}_+,\mathfrak{g})$,
which are identified with elements of $\bigwedge^k(\mathfrak{g}/\mathfrak{p})^*\otimes\mathfrak{g}$, will be
expressed as
\begin{equation}
	\label{eq:block-form-of-adjoint-valued-forms}
	\begin{pmatrix}
		\tensor{a}{_{A_1}_{\cdots}_{A_k}} & \tensor{Z}{_\beta_{A_1}_{\cdots}_{A_k}} & i\tensor{z}{_{A_1}_{\cdots}_{A_k}} \\
		\tensor{X}{^\alpha_{A_1}_{\cdots}_{A_k}} & \tensor{A}{_\beta^\alpha_{A_1}_{\cdots}_{A_k}} & -\tensor{Z}{^\alpha_{A_1}_{\cdots}_{A_k}} \\
		i\tensor{x}{_{A_1}_{\cdots}_{A_k}} & -\tensor{X}{_\beta_{A_1}_{\cdots}_{A_k}} & -\tensor{\conj{a}}{_{A_1}_{\cdots}_{A_k}}
	\end{pmatrix}
\end{equation}
based on \eqref{eq:general-elements-of-g}.
As before, the indices $A_1\cdots A_k$ in this notation can be suppressed.

Then, for an element $\psi\in(\mathfrak{g}/\mathfrak{p})^*\otimes\mathfrak{g}$
expressed as \eqref{eq:block-form-of-adjoint-valued-forms},
it follows from \eqref{eq:Kostant-codifferential-in-duality-1} that
\begin{equation}
	\partial^*\psi
	=\frac{1}{2(n+2)}
	\begin{pmatrix}
		x_0+\tensor{X}{^\gamma_\gamma}
		& i\tensor{X}{_\beta_0}+\tensor{A}{_\beta^\gamma_\gamma}-a_\beta
		& i(a_0+\conj{a}_0)+\tensor{Z}{_\gamma^\gamma}-\tensor{Z}{^\gamma_\gamma} \\
		-ix^\alpha
		& -\tensor{X}{^\alpha_\beta}+\tensor{X}{_\beta^\alpha}
		& i\tensor{X}{^\alpha_0}+\smash{\conj{a}}^\alpha+\tensor{A}{_\gamma^\alpha^\gamma} \\
		0 & -ix_\beta & -x_0-\tensor{X}{_\gamma^\gamma} 
	\end{pmatrix}.
\end{equation}
Likewise, for $\phi\in\bigwedge^2(\mathfrak{g}/\mathfrak{p})^*\otimes\mathfrak{g}$,
\eqref{eq:Kostant-codifferential-in-duality-2} implies that
\begin{align}
	(\partial^*\phi)(\xi_\sigma)
	&=\frac{1}{2(n+2)}
	\begin{pmatrix}
		\tensor{X}{^\gamma_\sigma_\gamma}-\tensor{x}{_0_\sigma} & \tensor{A}{_\beta^\gamma_\sigma_\gamma}-i\tensor{X}{_\beta_0_\sigma}-\tensor{a}{_\sigma_\beta} & \tensor{Z}{_\gamma_\sigma^\gamma}-\tensor{Z}{^\gamma_\sigma_\gamma}-i\tensor{a}{_0_\sigma}-i\tensor{\conj{a}}{_0_\sigma} \\
		-i\tensor{x}{_\sigma^\alpha} & -\tensor{X}{^\alpha_\sigma_\beta}+\tensor{X}{_\beta_\sigma^\alpha} & \tensor{A}{_\gamma^\alpha_\sigma^\gamma}-i\tensor{X}{^\alpha_0_\sigma}+\tensor{\conj{a}}{_\sigma^\alpha} \\
		0 & -i\tensor{x}{_\sigma_\beta} & -\tensor{X}{_\gamma_\sigma^\gamma}+\tensor{x}{_0_\sigma}
	\end{pmatrix}, \\
	(\partial^*\phi)(\xi_{\conj{\sigma}})
	&=\frac{1}{2(n+2)}
	\begin{pmatrix}
		\tensor{X}{^\gamma_{\conj{\sigma}}_\gamma}-\tensor{x}{_0_{\conj{\sigma}}} & \tensor{A}{_\beta^\gamma_{\conj{\sigma}}_\gamma}-i\tensor{X}{_\beta_0_{\conj{\sigma}}}-\tensor{a}{_{\conj{\sigma}}_\beta} & \tensor{Z}{_\gamma_{\conj{\sigma}}^\gamma}-\tensor{Z}{^\gamma_{\conj{\sigma}}_\gamma}-i\tensor{a}{_0_{\conj{\sigma}}}-i\tensor{\conj{a}}{_0_{\conj{\sigma}}} \\
		-i\tensor{x}{_{\conj{\sigma}}^\alpha} & -\tensor{X}{^\alpha_{\conj{\sigma}}_\beta}+\tensor{X}{_\beta_{\conj{\sigma}}^\alpha} & \tensor{A}{_\gamma^\alpha_{\conj{\sigma}}^\gamma}-i\tensor{X}{^\alpha_0_{\conj{\sigma}}}+\tensor{\conj{a}}{_{\conj{\sigma}}^\alpha} \\
		0 & -i\tensor{x}{_{\conj{\sigma}}_\beta} & -\tensor{X}{_\gamma_{\conj{\sigma}}^\gamma}+\tensor{x}{_0_{\conj{\sigma}}}
	\end{pmatrix}, \\
	(\partial^*\phi)(\xi_0)
	&=\frac{1}{2(n+2)}
	\begin{pmatrix}
		\tensor{X}{^\gamma_0_\gamma}+i\tensor{a}{_\gamma^\gamma} & i\tensor{Z}{_\beta_\gamma^\gamma}+\tensor{A}{_\beta^\gamma_0_\gamma}-\tensor{a}{_0_\beta} & -\tensor{z}{_\gamma^\gamma}+\tensor{Z}{_\gamma_0^\gamma}-\tensor{Z}{^\gamma_0_\gamma} \\
		i(\tensor{X}{^\alpha_\gamma^\gamma}-\tensor{x}{_0^\alpha}) & i\tensor{A}{_\beta^\alpha_\gamma^\gamma}-\tensor{X}{^\alpha_0_\beta}+\tensor{X}{_\beta_0^\alpha} & -i\tensor{Z}{^\alpha_\gamma^\gamma}+\tensor{A}{_\gamma^\alpha_0^\gamma}+\tensor{\conj{a}}{_0^\alpha} \\
		-\tensor{x}{_\gamma^\gamma} & -i(\tensor{X}{_\beta_\gamma^\gamma}+\tensor{x}{_0_\beta}) & -\tensor{X}{_\gamma_0^\gamma}-i\tensor{\conj{a}}{_\gamma^\gamma}
	\end{pmatrix}.
\end{align}

It follows that the zeroth homology is
\begin{equation}
	H_0(\mathfrak{p}_+,\mathfrak{g})=\mathfrak{g}/\im\partial_1^*=\mathfrak{g}/\mathfrak{g}^{-1}
	\cong\Re E(1,1),
\end{equation}
where the projection $\mathfrak{g}\to\Re E(1,1)$ is given by
\begin{equation}
	\begin{pmatrix}
		a & Z_\beta & iz \\
		X^\alpha & \tensor{A}{_\beta^\alpha} & -Z^\alpha \\
		ix & -X_\beta & -\conj{a}
	\end{pmatrix}
	\mapsto x.
\end{equation}
The first homology can be computed as
\begin{equation}
	\begin{split}
		H_1(\mathfrak{p}_+,\mathfrak{g})
		&=\frac{\ker\partial^*_1}{\im\partial^*_2}
		=\frac{\Set{
			\begin{pmatrix}
				\tensor{a}{_A} & \tensor{Z}{_\beta_A} & i\tensor{z}{_A} \\
				\tensor{X}{^\alpha_A} & \tensor{A}{_\beta^\alpha_A} & -\tensor{Z}{^\alpha_A} \\
				i\tensor{x}{_A} & -\tensor{X}{_\beta_A} & -\tensor{\conj{a}}{_A}
			\end{pmatrix} |
			\Centerstack{{$x_\sigma=x_{\conj{\sigma}}=0$, $x_0=-\tensor{X}{^\gamma_\gamma}$,}
				{$\tensor{X}{^\alpha_0}=i\tensor{A}{_\gamma^\alpha^\gamma}+i\tensor{\conj{a}}{^\alpha}$,}
				{$a_0+\conj{a}_0=i\tensor{Z}{_\gamma^\gamma}-i\tensor{Z}{^\gamma_\gamma}$}}
		}}{\Set{
			\begin{pmatrix}
				\tensor{a}{_A} & \tensor{Z}{_\beta_A} & i\tensor{z}{_A} \\
				\tensor{X}{^\alpha_A} & \tensor{A}{_\beta^\alpha_A} & -\tensor{Z}{^\alpha_A} \\
				i\tensor{x}{_A} & -\tensor{X}{_\beta_A} & -\tensor{\conj{a}}{_A}
			\end{pmatrix} |
			\Centerstack{{$x_\sigma=x_{\conj{\sigma}}=0$, $x_0=-\tensor{X}{^\gamma_\gamma}$,}
				{$\tensor{X}{_\sigma_\tau}=\tensor{X}{_[_\sigma_\tau_]},$}
				{$\tensor{X}{^\alpha_0}=i\tensor{A}{_\gamma^\alpha^\gamma}+i\tensor{\conj{a}}{^\alpha}$,}
				{$a_0+\conj{a}_0=i\tensor{Z}{_\gamma^\gamma}-i\tensor{Z}{^\gamma_\gamma}$}}
		}} \\
		&\cong\Re
		(\set{(\tensor{X}{^\alpha_{\conj{\beta}}})
		|X_{\conj{\alpha}\conj{\beta}}=X_{\conj{\beta}\conj{\alpha}}}
		\oplus
		\set{(\tensor{X}{_\alpha_\beta})
		|X_{\alpha\beta}=X_{\beta\alpha}}) \\
		&\cong \Re(E_{(\alpha\beta)}(1,1)\oplus E_{(\conj{\alpha}\conj{\beta})}(1,1)),
	\end{split}
\end{equation}
where we define the projection
$\ker\partial^*_1\to\Re(E_{(\alpha\beta)}(1,1)\oplus E_{(\conj{\alpha}\conj{\beta})}(1,1))$ by
\begin{equation}
	\begin{pmatrix}
		\tensor{a}{_A} & \tensor{Z}{_\beta_A} & i\tensor{z}{_A} \\
		\tensor{X}{^\alpha_A} & \tensor{A}{_\beta^\alpha_A} & -\tensor{Z}{^\alpha_A} \\
		i\tensor{x}{_A} & -\tensor{X}{_\beta_A} & -\tensor{\conj{a}}{_A}
	\end{pmatrix}
	\mapsto
	(X_{(\alpha\beta)},X_{(\conj{\alpha}\conj{\beta})}).
\end{equation}

Note that, in particular, $H_1(\mathfrak{p}_+,\mathfrak{g})$ is concentrated in
(the equivalence classes of elements of) $\mathfrak{g}_1\otimes\mathfrak{g}_{-1}$,
and hence in homogeneity $0$.
Together with Lemma \ref{lem:g0-contains-no-simple-ideals-of-g},
this fact guarantees that Theorem \ref{thm:existence-of-normal-cartan-connection} is applicable
to regular infinitesimal flag structures of CR type.

\section{Weyl structures and the CR normal Weyl forms}
\label{sec:construction-of-normal-tractor-connections}

In this section and the next one, we will, for normal regular parabolic geometries of CR type,
derive concrete formulae of the tractor connections
associated with the standard representation $\mathbb{V}$ and the adjoint representation $\mathfrak{g}$.

We are going to do so with the help of the general notion of \emph{Weyl structures} of
parabolic geometries $(\mathcal{G},\omega)$,
which are by definition $P_+$-equivariant sections $\sigma\colon\mathcal{G}_0\to\mathcal{G}$
where $\mathcal{G}_0=\mathcal{G}/P_+$
(see \v{C}ap--Slov\'ak \cite{Cap-Slovak-09}*{Chapter 5}).
A choice of a Weyl structure $\sigma$ gives rise to the \emph{induced Weyl form}
$\sigma^*\omega\in\Omega^1(\mathcal{G}_0,\mathfrak{g})$,
and by studying (normal) Weyl forms directly, one can effectively avoid
using the original (normal) parabolic geometry in the description of tractor bundles,
as we recall in the general setting in Section \ref{subsec:tractor-calculus-in-terms-of-Weyl-forms}
following \cite{Cap-Slovak-09}.
In Sections \ref{subsec:frame-bundles-in-CR-geometry}--\ref{subsec:normal-Weyl-form-continued},
we specialize in geometries of CR type and determine the normal Weyl form.

\subsection{Tractor calculus in terms of Weyl forms}
\label{subsec:tractor-calculus-in-terms-of-Weyl-forms}

Let $G$ be an arbitrary semisimple Lie group and $P$ its parabolic subgroup.
Recall the notion of infinitesimal flag structures of type $(G,P)$ introduced
in Section \ref{subsec:infinitesimal-flag-structures}.
Abstractly, Weyl forms of infinitesimal flag structures are defined as follows
(see \cite{Cap-Slovak-09}*{Definition 5.2.1}).

\begin{dfn}
	\label{dfn:Weyl-form}
	A \emph{Weyl form} of an infinitesimal flag structure
	$((T^iM)_{i=-k}^{-1},\mathcal{G}_0,(\theta_{-k},\dots,\theta_{-1}))$ of type $(G,P)$ is
	a $G_0$-invariant 1-form $\tau\in\Omega^1(\mathcal{G}_0,\mathfrak{g})$ such that
	\begin{enumerate}[(i)]
		\item
			$\tau(\zeta_A)=A$ for $A\in\mathfrak{g}_0$, where $\zeta_A$ is the fundamental vector field;
		\item
			For $-k\le i\le -1$, $\tau|_{T^i\mathcal{G}_0}$ takes values in $\mathfrak{g}^i$ and
			equals $\theta_i$ modulo $\mathfrak{g}^{i+1}$.
	\end{enumerate}
\end{dfn}

Given a Weyl structure $\sigma$ of a parabolic geometry $(\mathcal{G},\omega)$,
one can use the induced infinitesimal flag structure
$((T^iM)_{i=-k}^{-1},\mathcal{G}_0,(\theta_{-k},\dots,\theta_{-1}))$ and
the induced Weyl form $\tau=\sigma^*\omega$ to describe tractor bundles.
First, if $\mathbb{W}$ is a $(\mathfrak{g},P)$-module, then the associated bundle is
\begin{equation}
	\mathcal{W}=\mathcal{G}\times_P\mathbb{W}\cong\mathcal{G}_0\times_{G_0}\mathbb{W},
\end{equation}
the isomorphism being given by identifying $\assocbracket{g_0,w}^{G_0}$ with $\assocbracket{\sigma(g_0),w}^P$
(where $\assocbracket{\dotsb}$ denotes the equivalence class of $(\dotsb)$).
In order to describe the tractor connection $\nabla$ using this $\sigma$-dependent expression
$\mathcal{W}\cong\mathcal{G}_0\times_{G_0}\mathbb{W}$,
we need to introduce three objects related with the Weyl structure $\sigma$.
Let
\begin{equation}
	\sigma^*\omega=(\sigma^*\omega_-,\sigma^*\omega_0,\sigma^*\omega_+)
	\in
	\Omega^1(\mathcal{G}_0,\mathfrak{g}_-)\oplus
	\Omega^1(\mathcal{G}_0,\mathfrak{g}_0)\oplus
	\Omega^1(\mathcal{G}_0,\mathfrak{p}_+)
\end{equation}
be the decomposition of the induced Weyl form $\sigma^*\omega$
with respect to $\mathfrak{g}=\mathfrak{g}_-\oplus\mathfrak{g}_0\oplus\mathfrak{p}_+$.
The middle part $\gamma^\sigma=\sigma^*\omega_0$ is
a principal connection of $\mathcal{G}_0$, which is called the \emph{Weyl connection}.
The negative part $\sigma^*\omega_-$ is the \emph{soldering form},
which defines an isomorphism $T_{p(u)}M\cong\mathfrak{g}_-$ for each $u\in\mathcal{G}_0$,
where $p\colon\mathcal{G}_0\to M$ is the projection.
The positive part $\mathsf{P}^\sigma=\sigma^*\omega_+$ is called the \emph{Rho tensor}.
Then it is known that $\nabla$ can be expressed as follows.

\begin{prop}[\cite{Cap-Slovak-09}*{Proposition 5.1.10}]
	\label{prop:tractor-connection-in-terms-of-Weyl-form}
	Let $\sigma$ be a Weyl structure of a parabolic geometry $(\mathcal{G},\omega)$.
	Then the tractor connection of
	$\mathcal{W}=\mathcal{G}\times_P\mathbb{W}\cong\mathcal{G}_0\times_{G_0}\mathbb{W}$ is given by
	\begin{equation}
		\nabla_\xi s
		=\nabla^\sigma_\xi s+\mathsf{P}^\sigma(\xi)\cdot s+\xi\cdot s
	\end{equation}
	for a vector field $\xi$,
	where $\nabla^\sigma$ denotes the covariant differentiation with respect to the Weyl connection $\gamma^\sigma$
	and, in the last term on the right-hand side,
	$\xi$ is identified with a $\mathfrak{g}_-$-valued function by the soldering form $\sigma^*\omega_-$.
\end{prop}

It is crucially important in this approach to know that, for normal regular parabolic geometries,
we can characterize the Weyl form $\tau=\sigma^*\omega$
induced by some Weyl structure $\sigma$ directly at the level of the infinitesimal flag structure.
For an arbitrary Weyl form $\tau$ of the induced regular infinitesimal flag structure
(or of an infinitesimal flag structure in general),
its \emph{curvature form} $K\in\Omega^2(\mathcal{G}_0,\mathfrak{g})$ is defined by
\begin{equation}
	K=d\tau+\frac{1}{2}[\tau\wedge\tau],
\end{equation}
and, since $TM\cong\mathcal{G}_0\times_{G_0}\mathfrak{g}_-$, the curvature $K$ is equivalently expressed by
the \emph{curvature function} $\kappa\colon\mathcal{G}_0\to\bigwedge^2\mathfrak{g}_-\otimes\mathfrak{g}$.
Then $\tau$ is said to be \emph{normal} if $\partial^*\kappa=0$ is satisfied.
Under the assumption that $H^1(\mathfrak{g}_-,\mathfrak{g})$ is concentrated in non-positive
homogeneity degrees (cf.\ Theorem \ref{thm:existence-of-normal-cartan-connection}),
the normality is in fact a necessary and sufficient condition for $\tau$ being
induced by some Weyl structure \cite{Cap-Slovak-09}*{Theorem 5.2.2}.
Therefore, in order to compute the normal tractor connections explicitly,
it just suffices to investigate normal Weyl forms of the induced infinitesimal flag structure.

\subsection{Frame bundles in CR geometry and the Tanaka--Webster connection}
\label{subsec:frame-bundles-in-CR-geometry}

We next discuss how the geometry of a compatible almost CR structure $(H,J)$ on $M$ is encoded
in terms of regular infinitesimal flag structures, or equivalently
(see Proposition \ref{prop:infinitesimal-flag-structure-as-structure-group-reduction}),
in terms of reductions of the structure group of the full graded frame bundle $\mathcal{F}_{\gr}$
associated with the filtration $TM=T^{-2}M\supset T^{-1}M=H$.
We also need to discuss how the geometry of $(H,J,\mathcal{E}(1,0))$,
where $\mathcal{E}(1,0)$ is an $(n+2)$-nd root of the canonical bundle $\mathcal{K}$
(cf.\ Section \ref{subsec:basics}), can be encoded similarly.
To some extent, this subsection is an elaboration of \cite{Cap-Slovak-09}*{Example 3.1.7}.

Let $G^\sharp=\SU(n+1,1)$ be the special unitary group with respect to the indefinite Hermitian inner product
\eqref{eq:indefinite-Hermitian-inner-product}
and $G=\PSU(n+1,1)=G^\sharp/Z(G^\sharp)$ its quotient by its center $Z(G^\sharp)\cong\mathbb{Z}_{n+2}$
(the set of constant matrices of determinant $1$ in $G^\sharp$).
We define the parabolic subgroup $P$ of $G$ as the stabilizer of the null complex line
$\braket{\transpose{\smash{(1\; 0\; \cdots\; 0\; 0)}}}$ in $\mathbb{C}^{n+1,1}$,
and $P^\sharp$ is defined to be the pullback of $P$ by the projection $G^\sharp\to G$.
Then the Levi subgroups are
\begin{equation}
	G_0^\sharp=
	\Set{
		\begin{pmatrix}
			c & & \\
			& U & \\
			& & 1/\conj{c}
		\end{pmatrix}
		|
		\text{$c\in\mathbb{C}^\times$, $U\in U(n)$, and $\dfrac{c}{\,\conj{c}\,}\det U=1$}
		}
\end{equation}
and $G_0=G_0^\sharp/\mathbb{Z}_{n+2}$.

First, given a compatible almost CR structure $(H,J)$ on a smooth manifold $M$, we define
\begin{equation}
	\label{eq:G0-first-expression}
	\mathcal{G}_0
	=\bigsqcup_{x\in M}
	\Set{\varphi\colon\mathfrak{g}_-\to\gr(T_xM)|
		\Centerstack{%
			\text{$\varphi$ is a grading respecting linear isomorphism preserving}
			\text{the complex structures $j$ on $\mathfrak{g}_{-1}$ and $J$ on $H_x$, and}
			\text{the brackets $\mathfrak{g}_{-1}\times\mathfrak{g}_{-1}\to\mathfrak{g}_{-2}$
				and $H_x\times H_x\to T_xM/H_x$}%
		}
	},
\end{equation}
where $H_x\times H_x\to T_xM/H_x$ is the Levi bracket.
Clearly, $\mathcal{G}_0$ can be equivalently described as
\begin{equation}
	\label{eq:G0-second-expression}
	\mathcal{G}_0
	=\bigsqcup_{x\in M}
	\Set{(Z_1,\dots,Z_n,T)|
		\Centerstack{%
			\text{$Z_1$, $\dots$, $Z_n\in H^{1,0}_x$, $T\in T_xM/H_x$, and}
			\text{$[Z_\alpha,Z_{\conj{\beta}}]=-i\delta_{\alpha{\conj{\beta}}}T$ in $T_xM/H_x$}%
		}
	},
\end{equation}
or even as
\begin{equation}
	\label{eq:G0-third-expression}
	\mathcal{G}_0
	=\bigsqcup_{x\in M}
	\Set{(Z_1,\dots,Z_n)|
		\Centerstack{%
			\text{$Z_1$, $\dots$, $Z_n\in H^{1,0}_x$ is a conformal unitary frame}
			\text{with respect to $i[\mathord{\cdot},\mathord{\cdot}]\colon H^{1,0}_x\times H^{0,1}_x
			\to T_xM/H_x$}%
		}
	}.
\end{equation}
The adjoint action of the group $G_0$ on $\mathfrak{g}_-$
induces its right action on $\mathcal{G}_0$, given explicitly
using the second expression \eqref{eq:G0-second-expression} by
\begin{equation}
	\label{eq:G0-action-on-frame-bundle}
	(Z_1,\dots,Z_n,T)\cdot
	\left[
	\begin{pmatrix}
		c & & \\
		  & U & \\
		  & & 1/\conj{c}
	\end{pmatrix}
	\right]
	=(c^{-1}(Z_1,\dots,Z_n)U,\abs{c}^{-2}T).
\end{equation}
Thus the bundle $\mathcal{G}_0$ is a reduction of the structure group of $\mathcal{F}_{\gr}$ to $G_0$,
and the agreement of the brackets $\mathfrak{g}_{-1}\times\mathfrak{g}_{-1}\to\mathfrak{g}_{-2}$
and $H_x\times H_x\to T_xM/H_x$
means that the corresponding infinitesimal flag structure is regular.
Conversely, any regular infinitesimal flag structure corresponding to a structure group reduction
of $\mathcal{F}_{\gr}$ to $G_0$
is induced this way by some compatible almost CR structure.

The third expression \eqref{eq:G0-third-expression} of $\mathcal{G}_0$ shows that
the conformal unitary group $\CU(n)\subset\GL(n,\mathbb{C})$ is naturally acting on
$\mathcal{G}_0$ from the right.
This action and the one defined by \eqref{eq:G0-action-on-frame-bundle} are
compatible with the isomorphism $G_0\cong\CU(n)$ given by
\begin{equation}
	\label{eq:correspondence-of-G0-and-CUn}
	\left[
		\begin{pmatrix}
			c & & \\ & U & \\ & & 1/\conj{c}
		\end{pmatrix}
	\right]
	\mapsto c^{-1}U.
\end{equation}

Next, suppose that we are moreover given a complex line bundle $\mathcal{E}(1,0)$
with $\mathcal{E}(1,0)^{-n-2}=\mathcal{K}$.
Then the $(n+2)$-sheeted covering $\mathcal{G}_0^\sharp$ of $\mathcal{G}_0$ is defined by
\begin{equation}
	\mathcal{G}_0^\sharp
	=\bigsqcup_{x\in M}
	\Set{(Z_1,\dots,Z_n,T,\eta)|
		\Centerstack{%
			\text{$Z_1$, $\dots$, $Z_n\in H^{1,0}_x$, $T\in T_xM/H_x$,}
			\text{$[Z_\alpha,Z_{\conj{\beta}}]=-i\delta_{\alpha\conj{\beta}}T$ in $T_xM/H_x$, and}
			\text{$\eta\in\mathcal{E}(1,0)_x$, $\eta^{n+2}=T\wedge Z_1\wedge\dots\wedge Z_n\in\mathcal{K}_x^*$}%
		}
	},
\end{equation}
on which $G_0^\sharp$ acts by
\begin{equation}
	(Z_1,\dots,Z_n,T,\eta)\cdot
	\begin{pmatrix}
		c & & \\
		  & U & \\
		  & & 1/\conj{c}
	\end{pmatrix}
	=(c^{-1}(Z_1,\dots,Z_n)U,\abs{c}^{-2}T,c^{-1}\eta).
\end{equation}
The principal bundle structures of $\mathcal{G}_0$ and $\mathcal{G}_0^\sharp$ are
compatible with the projections $G_0^\sharp\to G_0$ and $\mathcal{G}_0^\sharp\to\mathcal{G}_0$,
and thus $\mathcal{G}_0^\sharp$ is a reduction of the structure group of $\mathcal{F}_{\gr}$ to $G_0^\sharp$.
Giving such a reduction is equivalent to giving a structure $(H,J,\mathcal{E}(1,0))$.

The Tanaka--Webster connection associated with any fixed contact form $\theta$ can be thought of
as a principal connection of $\mathcal{G}_0$ or $\mathcal{G}_0^\sharp$.
To describe this, take any local unitary frame $\set{Z_\alpha}$ of $H^{1,0}$ with respect to $\theta$,
defined over an open set $U$ of $M$.
Let $\omega=\tensor{\omega}{_\alpha^\beta}$ be the Tanaka--Webster connection form with respect $\set{Z_\alpha}$,
which takes values in $\mathfrak{u}(n)$.
Then, in terms of the local trivialization $\mathcal{G}_0|_U\cong U\times\CU(n)$ given by the frame $\set{Z_\alpha}$,
we define the $\mathfrak{cu}(n)$-valued principal connection form $\gamma$ of $\mathcal{G}_0$ by
\begin{equation}
	\gamma=\omega+g^{-1}dg,
\end{equation}
where $g\in\CU(n)$ denotes the fiber coordinate of $\mathcal{G}_0$ with respect to the above local trivialization.
Since $G_0\cong\CU(n)$ by \eqref{eq:correspondence-of-G0-and-CUn}, $\gamma$ can also be regarded as
a $\mathfrak{g}_0$-valued 1-form,
whose pullback $\underline{\gamma}$
by the local section $\bm{Z}\colon U\to\mathcal{G}_0|_U$ given by $\set{Z_\alpha}$ is
\begin{equation}
	\label{eq:Tanaka-Webster-connection-as-principal-g0-connection}
	\underline{\gamma}=\bm{Z}^*\gamma=
	\begin{pmatrix}
		-\frac{1}{n+2}\tensor{\omega}{_\gamma^\gamma} & & \\
		& \tensor{\omega}{_\alpha^\beta}-\frac{1}{n+2}\tensor{\omega}{_\gamma^\gamma}\tensor{\delta}{_\alpha^\beta} & \\
		& & -\frac{1}{n+2}\tensor{\omega}{_\gamma^\gamma}
	\end{pmatrix}.
\end{equation}
The pullback of $\gamma$ by the covering $\mathcal{G}_0^\sharp\to\mathcal{G}_0$,
which is a principal $G_0^\sharp$-connection form on $\mathcal{G}_0^\sharp$, is denoted by $\gamma^\sharp$.

\subsection{Exact Weyl structures and determination of the induced normal Weyl form}
\label{subsec:exact-Weyl-structures-and-normal-Weyl-form}

There is a special class of Weyl structures called \emph{exact Weyl structures}
in the theory of parabolic geometries (see \cite{Cap-Slovak-09}*{Section 5.1.7}).
Exact Weyl structures are in one-to-one correspondence with flat connections of any fixed ``bundle of scales''
induced by its global trivializations.
For parabolic geometries of CR type (or more generally, for parabolic contact structures),
the bundle $H^\perp$ of annihilators of $H$, which is associated with $\mathcal{G}_0$
through the homomorphism
\begin{equation}
	\label{eq:group-homomorphism-for-CR-bundle-of-scales}
	G_0\to\mathbb{R}^*,\qquad
	\left[
		\begin{pmatrix}
			c & & \\ & U & \\ & & 1/\conj{c}
		\end{pmatrix}
	\right]
	\mapsto\abs{c}^2
\end{equation}
or with $\mathcal{G}_0^\sharp$ through the composition $G_0^\sharp\to G_0\to\mathbb{R}^*$ of
\eqref{eq:group-homomorphism-for-CR-bundle-of-scales} and the quotient mapping,
is a bundle of scales \cite{Cap-Slovak-09}*{Section 5.2.11}.
Consequently, any contact form $\theta$ determines an exact Weyl structure
$\sigma_\theta\colon\mathcal{G}_0\to\mathcal{G}$,
or $\sigma_\theta^\sharp\colon\mathcal{G}_0^\sharp\to\mathcal{G}^\sharp$,
which is characterized by the fact that the associated Weyl connection
on $\mathcal{G}_0$ or $\mathcal{G}_0^\sharp$ induces
a linear connection of $H^\perp$ with respect to which $\theta$ is parallel.

We are going to determine the Weyl form induced by an exact Weyl structure
from a given normal regular parabolic geometry of CR type.
In fact, we discuss only the case of type $(G,P)$ in detail;
then the Weyl form $\tau_\theta^\sharp$ induced by $\sigma_\theta^\sharp$
from a normal regular parabolic geometry of type $(G^\sharp,P^\sharp)$ can simply be obtained by
pulling back the Weyl form $\tau_\theta$ induced from the corresponding normal regular parabolic geometry
of type $(G,P)$ by the standard covering mapping $\mathcal{G}_0^\sharp\to\mathcal{G}_0$.

So, let $(\mathcal{G},\omega)$ be a regular parabolic geometry of type $(G,P)$
and let $\sigma_\theta\colon\mathcal{G}_0\to\mathcal{G}$ be an exact Weyl structure as above.
Take an arbitrary local unitary frame $\set{Z_\alpha}$ of $H^{1,0}$ with respect to $\theta$
and let $\set{\theta^\alpha}$ be the dual admissible coframe.
The induced Weyl form $\sigma_\theta^*\omega$ is denoted by $\tau_\theta$, or simply by $\tau$ in the sequel.
Then, since the induced Weyl form $\tau$ satisfies the conditions in Definition \ref{dfn:Weyl-form},
its pullback $\underline{\tau}=\bm{Z}^*\tau$
by the local section $\bm{Z}\colon U\to\mathcal{G}_0|_U$ given by $\set{Z_\alpha}$ must be of the form
\begin{equation}
	\label{eq:weyl-form-compatibility-with-CR-infinitesimal-flag-structure}
	\underline{\tau}=
	\begin{pmatrix}
		* & * & * \\
		\theta^\alpha\bmod\theta & * & * \\
		i\theta & -\theta_\beta\bmod\theta & *
	\end{pmatrix},
\end{equation}
where each $*$ denotes an unspecified component and
$\theta_\beta=\delta_{\beta\conj{\gamma}}\theta^{\conj{\gamma}}$.

If we furthermore assume the normality of $(\mathcal{G},\omega)$ and hence that of $\tau$,
we can determine all the components of $\underline{\tau}$, which means that $\tau$ is determined
completely.
We make the following observation as the first step toward this.
Note that some general theory regarding this step is developed in \cite{Cap-Slovak-09}*{Section 5.2.11},
to which we do not appeal here.

\begin{lem}
	\label{lem:Weyl-form-and-Tanaka-Webster-connection}
	Let $(\mathcal{G},\omega)$ be a normal parabolic geometry of type $(G,P)$,
	and let $\tau$ be the Weyl form
	induced by the exact Weyl structure $\sigma_\theta\colon\mathcal{G}_0\to\mathcal{G}$
	associated with a contact form $\theta$.
	Then its pullback $\underline{\tau}$ by the local section $\bm{Z}\colon U\to\mathcal{G}_0|_U$ given by
	a unitary frame $\set{Z_\alpha}$ must be of the form
	\begin{equation}
		\label{eq:normal-Weyl-form-up-to-homogeneity-one}
		\underline{\tau}=
		\begin{pmatrix}
			-\frac{1}{n+2}\tensor{\omega}{_\gamma^\gamma}\bmod\theta & * & * \\
			\theta^\alpha & \tensor{\omega}{_\beta^\alpha}-\frac{1}{n+2}\tensor{\omega}{_\gamma^\gamma}\tensor{\delta}{_\beta^\alpha}\bmod\theta & * \\
			i\theta & -\theta_\beta & -\frac{1}{n+2}\tensor{\omega}{_\gamma^\gamma}\bmod\theta
		\end{pmatrix},
	\end{equation}
	where $\set{\theta^\alpha}$ be the dual admissible coframe
	and $\tensor{\omega}{_\beta^\alpha}$ is the Tanaka--Webster connection forms with respect to $\set{Z_\alpha}$.
\end{lem}

We remark that the notion of homogeneity is convenient to prove the above lemma
and to determine the remaining components in the next subsection.
We say that a $p$-form $\tau$ with values in $\mathfrak{g}$ on an open set of $\mathcal{G}_0$ is
of \emph{homogeneity} $\ge l$ if $\xi_1\in T^{i_1}\mathcal{G}_0$, $\dotsc$, $\xi_p\in T^{i_p}\mathcal{G}_0$ implies
$\tau(\xi_1,\dots,\xi_p)\in\mathfrak{g}^{i_1+\dots+i_p+l}$.
For example, Definition \ref{dfn:Weyl-form} implies that any Weyl form is of homogeneity $\ge 0$.
Similar language is used for differential forms with values in $\mathfrak{g}$ on open sets of $M$,
using the filtration $T^{-2}M\supset T^{-1}M$, as well.
Thus, if we set
\begin{equation}
	\label{eq:weyl-form-modulo-terms-of-homogeneity-at-least-1}
	\eta^{(0)}=
	\begin{pmatrix}
		& & \\
		\theta^\alpha & & \\
		i\theta & -\theta_\beta & \phantom{M}
	\end{pmatrix},
\end{equation}
then \eqref{eq:weyl-form-compatibility-with-CR-infinitesimal-flag-structure} implies that
$\underline{\tau}$ equals $\eta^{(0)}$ modulo terms of homogeneity $\ge 1$.

\begin{proof}[Proof of Lemma \ref{lem:Weyl-form-and-Tanaka-Webster-connection}]
	Let
	\begin{equation}
		\tilde{\eta}^{(0)}=
		\begin{pmatrix}
			-\frac{1}{n+2}\tensor{\omega}{_\gamma^\gamma} & & \\
			\theta^\alpha & \tensor{\omega}{_\beta^\alpha}-\frac{1}{n+2}\tensor{\omega}{_\gamma^\gamma}\tensor{\delta}{_\beta^\alpha} & \\
			i\theta & -\theta_\beta & -\frac{1}{n+2}\tensor{\omega}{_\gamma^\gamma}
		\end{pmatrix}.
	\end{equation}
	Then, \eqref{eq:weyl-form-compatibility-with-CR-infinitesimal-flag-structure} implies that
	the pullback of the Weyl form $\tau$ to the base should be expressed as
	\begin{equation}
		\underline{\tau}=
		\tilde{\eta}^{(0)}+
		\underbrace{\begin{pmatrix}
			\weylcomp{a}_\sigma\theta^\sigma+\weylcomp{a}_{\conj{\sigma}}\theta^{\conj{\sigma}}\bmod\theta & * & * \\
			\tensor{\weylcomp{X}}{^\alpha_0}\theta & \tensor{\weylcomp{A}}{_\beta^\alpha_\sigma}\theta^\sigma+\tensor{\weylcomp{A}}{_\beta^\alpha_{\conj{\sigma}}}\theta^{\conj{\sigma}}\bmod\theta & * \\
			& -\tensor{\weylcomp{X}}{_\beta_0}\theta & -\conj{\weylcomp{a}}_\sigma\theta^\sigma-\conj{\weylcomp{a}}_{\conj{\sigma}}\theta^{\conj{\sigma}}\bmod\theta
		\end{pmatrix}}_{\text{homogeneity $\ge 1$}},
	\end{equation}
	where the correction term (the second term) on the right-hand side is
	$\mathfrak{g}$-valued and transforms tensorially
	with respect to the indices $\alpha$, $\beta$
	(and, of course, with respect to $\sigma$ and $\conj{\sigma}$ as well)
	for changes of unitary frames $\set{Z_\alpha}$.
	Note that, since the correction term is $\mathfrak{g}$-valued, in particular we have
	\begin{equation}
		\conj{\weylcomp{a}_\sigma}=\conj{\weylcomp{a}}_{\conj{\sigma}},\qquad
		\conj{\weylcomp{a}_{\conj{\sigma}}}=\conj{\weylcomp{a}}_\sigma,\qquad
		\weylcomp{a}_\sigma+\tensor{\weylcomp{A}}{_\gamma^\gamma_\sigma}-\conj{\weylcomp{a}}_\sigma=0.
	\end{equation}
	Moreover, since the contact form $\theta$ should be parallel with respect to the linear connection of $H^\perp$
	induced by the Weyl connection $\gamma^{\sigma_\theta}$,
	\eqref{eq:group-homomorphism-for-CR-bundle-of-scales} implies that
	\begin{equation}
		\label{eq:weyl-structure-parallelism-of-contact-form}
		\weylcomp{a}_\sigma+\conj{\weylcomp{a}}_\sigma=0
	\end{equation}
	is necessary. It follows, therefore, that
	\begin{equation}
		\label{eq:weyl-structure-parallelism-of-contact-form-consequence}
		\tensor{\weylcomp{A}}{_\gamma^\gamma_\sigma}+2\weylcomp{a}_\sigma=0.
	\end{equation}

	The pullback of the curvature form $K$ of $\tau$ to the base is given by
	\begin{equation}
		\begin{multlined}
			\underline{K}\equiv
			d\tilde{\eta}^{(0)}+\frac{1}{2}[\tilde{\eta}^{(0)}\wedge\tilde{\eta}^{(0)}] \\
			+d
				\begin{pmatrix}
					\weylcomp{a}_\sigma\theta^\sigma+\weylcomp{a}_{\conj{\sigma}}\theta^{\conj{\sigma}}\bmod\theta & * & * \\
					\tensor{\weylcomp{X}}{^\alpha_0}\theta & \tensor{\weylcomp{A}}{_\beta^\alpha_\sigma}\theta^\sigma+\tensor{\weylcomp{A}}{_\beta^\alpha_{\conj{\sigma}}}\theta^{\conj{\sigma}}\bmod\theta & * \\
					& -\tensor{\weylcomp{X}}{_\beta_0}\theta & -\conj{\weylcomp{a}}_\sigma\theta^\sigma-\conj{\weylcomp{a}}_{\conj{\sigma}}\theta^{\conj{\sigma}}\bmod\theta
				\end{pmatrix} \\
			+\left[\eta^{(0)}\wedge
				\begin{pmatrix}
					\weylcomp{a}_\sigma\theta^\sigma+\weylcomp{a}_{\conj{\sigma}}\theta^{\conj{\sigma}}\bmod\theta & * & * \\
					\tensor{\weylcomp{X}}{^\alpha_0}\theta & \tensor{\weylcomp{A}}{_\beta^\alpha_\sigma}\theta^\sigma+\tensor{\weylcomp{A}}{_\beta^\alpha_{\conj{\sigma}}}\theta^{\conj{\sigma}}\bmod\theta & * \\
					& -\tensor{\weylcomp{X}}{_\beta_0}\theta & -\conj{\weylcomp{a}}_\sigma\theta^\sigma-\conj{\weylcomp{a}}_{\conj{\sigma}}\theta^{\conj{\sigma}}\bmod\theta
				\end{pmatrix}
			\right]
		\end{multlined}
	\end{equation}
	modulo terms of homogeneity $\ge 2$,
	where $\eta^{(0)}$ is defined by \eqref{eq:weyl-form-modulo-terms-of-homogeneity-at-least-1},
	because the bracketed wedge product of two homogeneity $\ge 1$ terms is of homogeneity $\ge 2$.
	By omitting more homogeneity $\ge 2$ terms, the above formula can be simplified further to
	\begin{equation}
		\label{eq:curvature-modulo-homogeneity-2}
		\begin{multlined}
			\underline{K}\equiv
			d\tilde{\eta}^{(0)}+\frac{1}{2}[\tilde{\eta}^{(0)}\wedge\tilde{\eta}^{(0)}]
			+\begin{pmatrix}
				& & \\
				\tensor{\weylcomp{X}}{^\alpha_0} & & \\
				& -\tensor{\weylcomp{X}}{_\beta_0} & \phantom{M}
			\end{pmatrix}d\theta \\
			+\left[\eta^{(0)}\wedge
				\begin{pmatrix}
					\weylcomp{a}_\sigma\theta^\sigma+\weylcomp{a}_{\conj{\sigma}}\theta^{\conj{\sigma}} & & \\
					\tensor{\weylcomp{X}}{^\alpha_0}\theta & \tensor{\weylcomp{A}}{_\beta^\alpha_\sigma}\theta^\sigma+\tensor{\weylcomp{A}}{_\beta^\alpha_{\conj{\sigma}}}\theta^{\conj{\sigma}} & \\
					& -\tensor{\weylcomp{X}}{_\beta_0}\theta & -\conj{\weylcomp{a}}_\sigma\theta^\sigma-\conj{\weylcomp{a}}_{\conj{\sigma}}\theta^{\conj{\sigma}}
				\end{pmatrix}
			\right].
		\end{multlined}
	\end{equation}

	The normality of $\tau$ means that its curvature function $\kappa$ satisfies $\partial^*\kappa=0$,
	which is equivalent to
	\begin{equation}
		\label{eq:normality-of-Weyl-form-pulled-back}
		\partial^*\underline{\kappa}=0.
	\end{equation}
	We now determine the correction term using
	\eqref{eq:normality-of-Weyl-form-pulled-back} and
	\eqref{eq:weyl-structure-parallelism-of-contact-form-consequence}.
	First, we have
	\begin{equation}
		d\tilde{\eta}^{(0)}+\frac{1}{2}[\tilde{\eta}^{(0)}\wedge\tilde{\eta}^{(0)}]
		=
		\begin{pmatrix}
			* & * & * \\
			\frac{1}{2}\tensor{N}{^\alpha_{\conj{\sigma}}_{\conj{\tau}}}\theta^{\conj{\sigma}}\wedge\theta^{\conj{\tau}} \bmod\theta & * & * \\
			& -\frac{1}{2}\tensor{N}{_\beta_\sigma_\tau}\theta^\sigma\wedge\theta^\tau \bmod\theta & *
		\end{pmatrix}.
	\end{equation}
	Consequently, \eqref{eq:curvature-modulo-homogeneity-2} implies that
	$\underline{K}$ is of homogeneity $\ge 1$, and if we write
	\begin{equation}
		\underline{K}=
		\begin{pmatrix}
			* & * & * \\
			\bm{X}^\alpha & * & * \\
			i\bm{x} & -\bm{X}_\beta & *
		\end{pmatrix},
	\end{equation}
	then
	\begin{align}
		\tensor{\bm{x}}{_0_\sigma}
		&=i\tensor{\weylcomp{X}}{_\sigma_0},\\
		\tensor{\bm{X}}{^\alpha_\sigma_\tau}
		&=\tensor{\weylcomp{A}}{_\tau^\alpha_\sigma}
		-\tensor{\weylcomp{A}}{_\sigma^\alpha_\tau}
		+\weylcomp{a}_\tau\tensor{\delta}{_\sigma^\alpha}
		-\weylcomp{a}_\sigma\tensor{\delta}{_\tau^\alpha},\\
		\tensor{\bm{X}}{^\alpha_\sigma_{\conj{\tau}}}
		&=i\tensor{h}{_\sigma_{\conj{\tau}}}\tensor{\weylcomp{X}}{^\alpha_0}
		-\tensor{\weylcomp{A}}{_\sigma^\alpha_{\conj{\tau}}}
		+\weylcomp{a}_{\conj{\tau}}\tensor{\delta}{_\sigma^\alpha},\\
		\tensor{\bm{X}}{^\alpha_{\conj{\sigma}}_{\conj{\tau}}}
		&=\tensor{N}{^\alpha_{\conj{\sigma}}_{\conj{\tau}}}.
	\end{align}

	The normality \eqref{eq:normality-of-Weyl-form-pulled-back} implies,
	by the computation in Section \ref{subsec:homology-with-values-in-adjoint},
	\begin{subequations}
	\begin{gather}
		\label{eq:weyl-form-homogeneity-one-equation-1}
		\tensor{\bm{X}}{^\gamma_\sigma_\gamma}-\tensor{\bm{x}}{_0_\sigma}=0,\\
		\label{eq:weyl-form-homogeneity-one-equation-2}
		\tensor{\bm{X}}{^\alpha_\sigma_\beta}-\tensor{\bm{X}}{_\beta_\sigma^\alpha}=0,\\
		\label{eq:weyl-form-homogeneity-one-equation-3}
		\tensor{\bm{X}}{_\gamma_\sigma^\gamma}-\tensor{\bm{x}}{_0_\sigma}=0,\\
		\label{eq:weyl-form-homogeneity-one-equation-4}
		\tensor{\bm{X}}{_\beta_\gamma^\gamma}+\tensor{\bm{x}}{_0_\beta}=0.
	\end{gather}
	\end{subequations}
	Equations \eqref{eq:weyl-form-homogeneity-one-equation-1}, \eqref{eq:weyl-form-homogeneity-one-equation-3},
	\eqref{eq:weyl-form-homogeneity-one-equation-4} imply
	\begin{gather}
		-i\tensor{\weylcomp{X}}{_\sigma_0}
		+\tensor{\weylcomp{A}}{_\gamma^\gamma_\sigma}
		-\tensor{\weylcomp{A}}{_\sigma^\gamma_\gamma}
		-(n-1)\tensor{\weylcomp{a}}{_\sigma}=0,\\
		-2i\tensor{\weylcomp{X}}{_\sigma_0}
		+\tensor{\weylcomp{A}}{_\gamma^\gamma_\sigma}
		+n\tensor{\weylcomp{a}}{_\sigma}=0,\\
		i(n+1)\tensor{\weylcomp{X}}{_\sigma_0}
		-\tensor{\weylcomp{A}}{_\sigma^\gamma_\gamma}
		-\tensor{\weylcomp{a}}{_\sigma}=0,
	\end{gather}
	from which, together with \eqref{eq:weyl-structure-parallelism-of-contact-form-consequence}, we can conclude that
	$\tensor{\weylcomp{X}}{_\sigma_0}=\tensor{\weylcomp{A}}{_\gamma^\gamma_\sigma}
	=\tensor{\weylcomp{A}}{_\sigma^\gamma_\gamma}=\tensor{\weylcomp{a}}{_\sigma}=0$.
	Then we moreover obtain from \eqref{eq:weyl-form-homogeneity-one-equation-2} that
	$2\tensor{\weylcomp{A}}{_\beta^\alpha_\sigma}-\tensor{\weylcomp{A}}{_\sigma^\alpha_\beta}=0$,
	which implies that $\tensor{\weylcomp{A}}{_\beta^\alpha_\sigma}=0$.
\end{proof}

Note that curvature form of \eqref{eq:normal-Weyl-form-up-to-homogeneity-one} is of the form
\begin{equation}
	\label{eq:normal-Weyl-curvature-up-to-homogeneity-one}
	\underline{K}=
	\begin{pmatrix}
		* & * & * \\
		\frac{1}{2}\tensor{N}{^\alpha_{\conj{\sigma}}_{\conj{\tau}}}\theta^{\conj{\sigma}}\wedge\theta^{\conj{\tau}} \bmod\theta & * & * \\
		& -\frac{1}{2}\tensor{N}{_\beta_\sigma_\tau}\theta^\sigma\wedge\theta^\tau \bmod\theta & *
	\end{pmatrix},
\end{equation}
as the computation in the above proof shows.

We also remark that the soldering form $\tau_-$ is completely determined by
Lemma \ref{lem:Weyl-form-and-Tanaka-Webster-connection};
the pullback of $\tau_-$ to the base manifold by any local unitary frame $\bm{Z}=\set{Z_\alpha}$
equals $\eta^{(0)}$ in \eqref{eq:weyl-form-modulo-terms-of-homogeneity-at-least-1}.
Thus, differential forms on $U$ with values in $\mathfrak{g}$ can be thought of as functions on $U$
with values in $\bigwedge^p\mathfrak{g}_-^*\otimes\mathfrak{g}$ via $\eta^{(0)}$.
This in particular implies that the \emph{homogeneity $l$ component} of
a $\mathfrak{g}$-valued differential form on $U$ makes sense.

\subsection{Determination of the induced normal Weyl form (continued)}
\label{subsec:normal-Weyl-form-continued}

We continue our discussion to determine the rest of the components of the normal Weyl form $\tau$
induced by an exact Weyl structure $\sigma_\theta\colon\mathcal{G}_0\to\mathcal{G}$.
Actually, it suffices to determine $\underline{\tau}$ only modulo terms of homogeneity $\ge 3$
for our purpose of identifying the first BGG operators in Section \ref{sec:computation-of-CR-BGG-operators}.
However, we derive the full formula for future reference.

\subsubsection{Homogeneity $2$ components}

Our task here is to determine $\underline{\tau}$ modulo terms of homogeneity $\ge 3$.
Lemma \ref{lem:Weyl-form-and-Tanaka-Webster-connection} implies that, if we set
\begin{equation}
	\eta^{(1)}
	=
	\begin{pmatrix}
		-\frac{1}{n+2}\tensor{\omega}{_\gamma^\gamma} & & \\
		\theta^\alpha & \tensor{\omega}{_\beta^\alpha}-\frac{1}{n+2}\tensor{\omega}{_\gamma^\gamma}\tensor{\delta}{_\beta^\alpha} & \\
		i\theta & -\theta_\beta & -\frac{1}{n+2}\tensor{\omega}{_\gamma^\gamma}
	\end{pmatrix},
\end{equation}
then $\underline{\tau}$ can be expressed as
\begin{equation}
	\underline{\tau}=
	\eta^{(1)}+
	\underbrace{
		\begin{pmatrix}
			\weylcomp{a}_0\theta & \weylcomp{Z}_{\beta\sigma}\theta^\sigma+\weylcomp{Z}_{\beta\conj{\sigma}}\theta^{\conj{\sigma}} \bmod\theta & * \\
			& \tensor{\weylcomp{A}}{_\beta^\alpha_0}\theta & -\tensor{\weylcomp{Z}}{^\alpha_\sigma}\theta^\sigma-\tensor{\weylcomp{Z}}{^\alpha_{\conj{\sigma}}}\theta^{\conj{\sigma}} \bmod\theta \\
			& & -\conj{\weylcomp{a}}_0\theta
		\end{pmatrix}
	}_{\text{homogeneity $\ge 2$}},
\end{equation}
where $\conj{\weylcomp{a}_0}=\conj{\weylcomp{a}}_0$ and
$\weylcomp{a}_0+\tensor{\weylcomp{A}}{_\gamma^\gamma_0}-\conj{\weylcomp{a}}_0=0$.
Furthermore, as before, since the Weyl connection makes $\theta$ parallel,
$\weylcomp{a}_0$ should be purely imaginary and hence
\begin{equation}
	\label{eq:parallelism-of-contact-form-homogeneity-two}
	\tensor{\weylcomp{A}}{_\gamma^\gamma_0}+2\weylcomp{a}_0=0.
\end{equation}
The computation of the curvature form $K$ pulled back to the base
can be carried out modulo terms of homogeneity $\ge 3$ as follows:
\begin{equation}
	\label{eq:curvature-modulo-homogeneity-3}
	\begin{multlined}
		\underline{K}
		\equiv d\eta^{(1)}+\frac{1}{2}[\eta^{(1)}\wedge\eta^{(1)}]
		+\begin{pmatrix}
			\weylcomp{a}_0 & & \\
			& \tensor{\weylcomp{A}}{_\beta^\alpha_0} & \\
			& & -\conj{\weylcomp{a}}_0
		\end{pmatrix}d\theta \\
		+\left[\eta^{(0)}\wedge
			\begin{pmatrix}
				\weylcomp{a}_0\theta & \weylcomp{Z}_{\beta\sigma}\theta^\sigma+\weylcomp{Z}_{\beta\conj{\sigma}}\theta^{\conj{\sigma}} & \\
				& \tensor{\weylcomp{A}}{_\beta^\alpha_0}\theta & -\tensor{\weylcomp{Z}}{^\alpha_\sigma}\theta^\sigma-\tensor{\weylcomp{Z}}{^\alpha_{\conj{\sigma}}}\theta^{\conj{\sigma}} \\
				& & -\conj{\weylcomp{a}}_0\theta
			\end{pmatrix}
		\right].
	\end{multlined}
\end{equation}
The first two terms of the right-hand side is given by \eqref{eq:normal-Weyl-curvature-up-to-homogeneity-one}
modulo terms of homogeneity $\ge 2$,
and here we need to compute it one homogeneity higher. By a direct computation, we get
\begin{equation}
	d\eta^{(1)}+\frac{1}{2}[\eta^{(1)}\wedge\eta^{(1)}]
	=
	\begin{pmatrix}
		\Pi^{(1)} & * & * \\
		\smash{\Pi^{(1)}}^\alpha & \tensor{\smash{\Pi^{(1)}}}{_\beta^\alpha} & * \\
		& -{\Pi^{(1)}}_\beta & -\conj{\Pi^{(1)}}
	\end{pmatrix},
\end{equation}
which is a $\mathfrak{g}$-valued 2-form, where
\begin{align}
	{\smash{\Pi^{(1)}}}^\alpha
	&=\tfrac{1}{2}\tensor{N}{^\alpha_{\conj{\sigma}}_{\conj{\tau}}}\theta^{\conj{\sigma}}\wedge\theta^{\conj{\tau}}
	+\tensor{A}{^\alpha_{\conj{\sigma}}}\theta\wedge\theta^{\conj{\sigma}}, \\
	\Pi^{(1)}
	&\equiv-\tfrac{1}{n+2}(R_{\sigma\conj{\tau}}\theta^\sigma\wedge\theta^{\conj{\tau}}
	+\tfrac{1}{2}V_{\sigma\tau}\theta^\sigma\wedge\theta^\tau
	+\tfrac{1}{2}V_{\conj{\sigma}\conj{\tau}}\theta^{\conj{\sigma}}\wedge\theta^{\conj{\tau}})
	\bmod\theta,\\
	\begin{split}
		\tensor{\smash{\Pi^{(1)}}}{_\beta^\alpha}
		&\equiv(\tensor{R}{_\beta^\alpha_\sigma_{\conj{\tau}}}-\tfrac{1}{n+2}\tensor{\delta}{_\beta^\alpha}R_{\sigma\conj{\tau}})\theta^\sigma\wedge\theta^{\conj{\tau}} \\
		&\qquad+\tfrac{1}{2}(\tensor{V}{_\beta^\alpha_\sigma_\tau}-\tfrac{1}{n+2}\tensor{\delta}{_\beta^\alpha}V_{\sigma\tau})\theta^\sigma\wedge\theta^\tau
		+\tfrac{1}{2}(\tensor{V}{_\beta^\alpha_{\conj{\sigma}}_{\conj{\tau}}}-\tfrac{1}{n+2}\tensor{\delta}{_\beta^\alpha}V_{\conj{\sigma}\conj{\tau}})\theta^{\conj{\sigma}}\wedge\theta^{\conj{\tau}}\bmod\theta.
	\end{split}
\end{align}
(Note that $\Pi^{(1)}$ is purely imaginary because $\eta^{(1)}$ induces a flat connection of $H^\perp$.)
So if we write
\begin{equation}
	\underline{K}
	=\begin{pmatrix}
		\bm{a} & * & * \\
		\bm{X}^\alpha & \tensor{\bm{A}}{_\beta^\alpha} & * \\
		& \bm{X}_\beta & -\conj{\bm{a}}
	\end{pmatrix},
\end{equation}
then, by calculating the right-hand side of \eqref{eq:curvature-modulo-homogeneity-3}, we obtain
\allowdisplaybreaks
\begin{align}
	\tensor{\bm{X}}{^\alpha_\sigma_\tau}
	&=\tensor{\bm{X}}{^\alpha_\sigma_{\conj{\tau}}}=0,\\
	\tensor{\bm{X}}{^\alpha_{\conj{\sigma}}_{\conj{\tau}}}
	&=\tensor{N}{^\alpha_{\conj{\sigma}}_{\conj{\tau}}},\\
	\tensor{\bm{X}}{^\alpha_0_\sigma}
	&=\tensor{\weylcomp{A}}{_\sigma^\alpha_0}
	-\tensor{\delta}{_\sigma^\alpha}\weylcomp{a}_0+i\tensor{\weylcomp{Z}}{^\alpha_\sigma},\\
	\tensor{\bm{X}}{^\alpha_0_{\conj{\sigma}}}
	&=\tensor{A}{^\alpha_{\conj{\sigma}}}+i\tensor{\weylcomp{Z}}{^\alpha_{\conj{\sigma}}},\\
	\tensor{\bm{A}}{_\beta^\alpha_\sigma_\tau}
	&=\tensor{V}{_\beta^\alpha_\sigma_\tau}
	-\tfrac{1}{n+2}\tensor{\delta}{_\beta^\alpha}\tensor{V}{_\sigma_\tau}
	+\tensor{\delta}{_\sigma^\alpha}\tensor{\weylcomp{Z}}{_\beta_\tau}
	-\tensor{\delta}{_\tau^\alpha}\tensor{\weylcomp{Z}}{_\beta_\sigma}, \\
	\tensor{\bm{A}}{_\beta^\alpha_\sigma_{\conj{\tau}}}
	&=\tensor{R}{_\beta^\alpha_\sigma_{\conj{\tau}}}
	-\tfrac{1}{n+2}\tensor{\delta}{_\beta^\alpha}\tensor{R}{_\sigma_{\conj{\tau}}}
	+i\tensor{h}{_\sigma_{\conj{\tau}}}\tensor{\weylcomp{A}}{_\beta^\alpha_0}
	+\tensor{\delta}{_\sigma^\alpha}\tensor{\weylcomp{Z}}{_\beta_{\conj{\tau}}}
	+\tensor{h}{_\beta_{\conj{\tau}}}\tensor{\weylcomp{Z}}{^\alpha_\sigma}, \\
	\tensor{\bm{A}}{_\beta^\alpha_{\conj{\sigma}}_{\conj{\tau}}}
	&=\tensor{V}{_\beta^\alpha_{\conj{\sigma}}_{\conj{\tau}}}
	-\tfrac{1}{n+2}\tensor{\delta}{_\beta^\alpha}\tensor{V}{_{\conj{\sigma}}_{\conj{\tau}}}
	-\tensor{h}{_\beta_{\conj{\sigma}}}\tensor{\weylcomp{Z}}{^\alpha_{\conj{\tau}}}
	+\tensor{h}{_\beta_{\conj{\tau}}}\tensor{\weylcomp{Z}}{^\alpha_{\conj{\sigma}}}, \\
	\bm{a}_{\sigma\tau}
	&=-\tfrac{1}{n+2}V_{\sigma\tau}-\weylcomp{Z}_{\sigma\tau}+\weylcomp{Z}_{\tau\sigma}, \\
	\bm{a}_{\sigma\conj{\tau}}
	&=-\tfrac{1}{n+2}R_{\sigma\conj{\tau}}+ih_{\sigma\conj{\tau}}\weylcomp{a}_0-\weylcomp{Z}_{\sigma\conj{\tau}}, \\
	\bm{a}_{\conj{\sigma}\conj{\tau}}
	&=-\tfrac{1}{n+2}V_{\conj{\sigma}\conj{\tau}}.
\end{align}
\allowdisplaybreaks[0]%

The normality condition \eqref{eq:normality-of-Weyl-form-pulled-back} implies that
\begin{subequations}
\begin{gather}
	\label{eq:normal-weyl-form-homogeneity-two-equation-1}
	i\tensor{\bm{A}}{_\beta^\alpha_\gamma^\gamma}-\tensor{\bm{X}}{^\alpha_0_\beta}+\tensor{\bm{X}}{_\beta_0^\alpha}=0,\\
	\label{eq:normal-weyl-form-homogeneity-two-equation-2}
	\tensor{\bm{X}}{^\gamma_0_\gamma}+i\tensor{\bm{a}}{_\gamma^\gamma}=0,\\
	\label{eq:normal-weyl-form-homogeneity-two-equation-3}
	\tensor{\bm{A}}{_\beta^\gamma_\sigma_\gamma}-i\tensor{\bm{X}}{_\beta_0_\sigma}+\tensor{\bm{a}}{_\beta_\sigma}=0,\\
	\label{eq:normal-weyl-form-homogeneity-two-equation-4}
	-\tensor{\bm{A}}{_\beta^\gamma_\gamma_{\conj{\sigma}}}-i\tensor{\bm{X}}{_\beta_0_{\conj{\sigma}}}
	+\tensor{\bm{a}}{_\beta_{\conj{\sigma}}}=0.
\end{gather}
\end{subequations}
It follows from \eqref{eq:normal-weyl-form-homogeneity-two-equation-1} that
\begin{subequations}
\begin{equation}
	i\tensor{{R'}}{_\beta^\alpha}-\frac{i}{n+2}R\tensor{\delta}{_\beta^\alpha}
	-(n+2)\tensor{\weylcomp{A}}{_\beta^\alpha_0}+2\tensor{\delta}{_\beta^\alpha}\weylcomp{a}_0=0.
\end{equation}
Therefore, using \eqref{eq:parallelism-of-contact-form-homogeneity-two} together, we can deduce that
\begin{equation}
	\label{eq:normal-weyl-form-homogeneity-two-term-1}
	\weylcomp{a}_0=-\frac{i}{n+2}P,\qquad
	\tensor{\weylcomp{A}}{_\beta^\alpha_0}
	=i\tensor{\smash{P'}}{^\alpha_\beta}-\frac{i}{n+2}\tensor{\delta}{^\alpha_\beta}P.
\end{equation}
Next, \eqref{eq:normal-weyl-form-homogeneity-two-equation-3} implies that
\begin{equation}
	\tensor{V}{_\beta^\gamma_\sigma_\gamma}-i\tensor{A}{_\beta_\sigma}
	-(n+1)\tensor{\weylcomp{Z}}{_\beta_\sigma}+\tensor{\weylcomp{Z}}{_\sigma_\beta}=0,
\end{equation}
and hence
\begin{equation}
	\label{eq:normal-weyl-form-homogeneity-two-term-2}
	\begin{split}
		\weylcomp{Z}_{\beta\sigma}
		&=-iA_{\beta\sigma}
		+\frac{1}{n}(\nabla^*N)^\text{sym}_{\beta\sigma}
		+\frac{1}{n+2}(\nabla^*N)^\text{skew}_{\beta\sigma} \\
		&=-iA_{\beta\sigma}
		+\frac{n+1}{n(n+2)}(\nabla^*N)_{\beta\sigma}
		+\frac{1}{n(n+2)}(\nabla^*N)_{\sigma\beta}.
	\end{split}
\end{equation}
Finally, by \eqref{eq:normal-weyl-form-homogeneity-two-equation-4},
\begin{equation}
	-R''_{\beta\conj{\sigma}}
	-(n+2)\weylcomp{Z}_{\beta\conj{\sigma}}
	-h_{\beta\conj{\sigma}}\tensor{\weylcomp{Z}}{_\gamma^\gamma}=0,
\end{equation}
and thus we obtain
\begin{equation}
	\label{eq:normal-weyl-form-homogeneity-two-term-3}
	\weylcomp{Z}_{\beta\conj{\sigma}}=-P''_{\beta\conj{\sigma}}.
\end{equation}
\end{subequations}
Then \eqref{eq:normal-weyl-form-homogeneity-two-equation-2} is also satisfied.

\subsubsection{Homogeneity $3$ components}

Similarly, in the next step, we determine the normal Weyl form $\tau$ modulo terms of homogeneity $\ge 4$.
We have already shown that $\underline{\tau}$ equals
\begin{equation}
	\eta^{(2)}
	=
	\begin{pmatrix}
		-\frac{1}{n+2}\tensor{\omega}{_\gamma^\gamma}+\weylcomp{a}_0\theta & \weylcomp{Z}_{\beta\sigma}\theta^\sigma+\weylcomp{Z}_{\beta\conj{\sigma}}\theta^{\conj{\sigma}} & \\
		\theta^\alpha & \tensor{\omega}{_\beta^\alpha}-\frac{1}{n+2}\tensor{\omega}{_\gamma^\gamma}\tensor{\delta}{_\beta^\alpha}+\tensor{\weylcomp{A}}{_\beta^\alpha_0}\theta & -\tensor{\weylcomp{Z}}{^\alpha_\sigma}\theta^\sigma-\tensor{\weylcomp{Z}}{^\alpha_{\conj{\sigma}}}\theta^{\conj{\sigma}} \\
		i\theta & -\theta_\beta & -\frac{1}{n+2}\tensor{\omega}{_\gamma^\gamma}-\conj{\weylcomp{a}}_0\theta
	\end{pmatrix}
\end{equation}
modulo terms of homogeneity $\ge 3$, where the homogeneity $2$ components are given by
\eqref{eq:normal-weyl-form-homogeneity-two-term-1}, \eqref{eq:normal-weyl-form-homogeneity-two-term-2},
and \eqref{eq:normal-weyl-form-homogeneity-two-term-3}.
So we can write
\begin{equation}
	\underline{\tau}=
	\eta^{(2)}+
	\underbrace{
		\begin{pmatrix}
			\phantom{M} & \weylcomp{Z}_{\beta 0}\theta & i\weylcomp{z}_\sigma\theta^\sigma+i\weylcomp{z}_{\conj{\sigma}}\theta^{\conj{\sigma}} \bmod\theta \\
			& & -\tensor{\weylcomp{Z}}{^\alpha_0}\theta \\
			& &
		\end{pmatrix}
	}_{\text{homogeneity $\ge 3$}}.
\end{equation}
Its curvature is given modulo terms of homogeneity $\ge 4$ by
\begin{equation}
	\label{eq:curvature-modulo-homogeneity-4}
	\begin{multlined}
		\underline{K}
		\equiv d\eta^{(2)}+\frac{1}{2}[\eta^{(2)}\wedge\eta^{(2)}]
		+\begin{pmatrix}
			\phantom{M} & \weylcomp{Z}_{\beta 0} & \\
			& & -\tensor{\weylcomp{Z}}{^\alpha_0} \\
			& &
		\end{pmatrix} d\theta \\
		+\left[\eta^{(0)}\wedge
			\begin{pmatrix}
				\phantom{M} & \weylcomp{Z}_{\beta 0}\theta & i\weylcomp{z}_\sigma\theta^\sigma+i\weylcomp{z}_{\conj{\sigma}}\theta^{\conj{\sigma}} \\
				& & -\tensor{\weylcomp{Z}}{^\alpha_0}\theta \\
				& &
			\end{pmatrix}
		\right].
	\end{multlined}
\end{equation}

Note that, to determine $\weylcomp{Z}_{\beta 0}$ and $\weylcomp{z}_\sigma$ using
the normality condition \eqref{eq:normality-of-Weyl-form-pulled-back},
only the homogeneity $3$ components of $\underline{K}$ are involved.
So we set
\begin{equation}
	d\eta^{(2)}+\frac{1}{2}[\eta^{(2)}\wedge\eta^{(2)}]
	=
	\begin{pmatrix}
		\Pi^{(2)} & {\Pi^{(2)}}_\beta & * \\
		* & \tensor{\smash{\Pi^{(2)}}}{_\beta^\alpha} & -\smash{\Pi^{(2)}}^\alpha \\
		& * & -\conj{\Pi^{(2)}}
	\end{pmatrix}
\end{equation}
and
\begin{gather}
	\tensor{\smash{\Pi^{(2)}}}{_\beta^\alpha}
	=\tfrac{1}{2}\tensor{\smash{\Pi^{(2)}}}{_\beta^\alpha_K_L}\theta^K\wedge\theta^L,\qquad
	\smash{\Pi^{(2)}}
	=\tfrac{1}{2}\tensor{\smash{\Pi^{(2)}}}{_K_L}\theta^K\wedge\theta^L,\\
	{\smash{\Pi^{(2)}}}_\beta
	=\tfrac{1}{2}\tensor{\smash{\Pi^{(2)}}}{_\beta_K_L}\theta^K\wedge\theta^L,
\end{gather}
where the indices $K$ and $L$ run through $\set{0,1,\dots,n,\conj{1},\dots,\conj{n}}$,
and compute its homogeneity $3$ components,
namely $\tensor{\smash{\Pi^{(2)}}}{_\beta^\alpha_0_\sigma}$,
$\tensor{\smash{\Pi^{(2)}}}{_\beta^\alpha_0_{\conj{\sigma}}}$,
${\smash{\Pi^{(2)}}}_{0\sigma}$, ${\smash{\Pi^{(2)}}}_{0\conj{\sigma}}$,
and ${\smash{\Pi^{(2)}}}_{\beta\sigma\conj{\tau}}$, ${\smash{\Pi^{(2)}}}_{\beta\sigma\tau}$,
${\smash{\Pi^{(2)}}}_{\beta\conj{\sigma}\conj{\tau}}$.
By a direct calculation, we obtain
\allowdisplaybreaks
\begin{align}
	\tensor{\smash{\Pi^{(2)}}}{_\beta^\alpha_0_\sigma}
	&=-\tensor{W}{_\beta^\alpha_\sigma}
	+\tfrac{1}{n+2}\tensor{\delta}{_\beta^\alpha}\tensor{W}{_\gamma^\gamma_\sigma}
	-i\nabla_\sigma\tensor{\weylcomp{A}}{_\beta^\alpha_0}, \\
	\tensor{\smash{\Pi^{(2)}}}{_\beta^\alpha_0_{\conj{\sigma}}}
	&=-\tensor{W}{_\beta^\alpha_{\conj{\sigma}}}
	+\tfrac{1}{n+2}\tensor{\delta}{_\beta^\alpha}\tensor{W}{_\gamma^\gamma_{\conj{\sigma}}}
	-i\nabla_{\conj{\sigma}}\tensor{\weylcomp{A}}{_\beta^\alpha_0}, \\
	{\smash{\Pi^{(2)}}}_{0\sigma}
	&=\tfrac{1}{n+2}\tensor{W}{_\gamma^\gamma_\sigma}-\nabla_\sigma \weylcomp{a}_0, \\
	{\smash{\Pi^{(2)}}}_{0\conj{\sigma}}
	&=\tfrac{1}{n+2}\tensor{W}{_\gamma^\gamma_{\conj{\sigma}}}-\nabla_{\conj{\sigma}}\weylcomp{a}_0, \\
	{\smash{\Pi^{(2)}}}_{\beta\sigma\conj{\tau}}
	&=\nabla_\sigma\weylcomp{Z}_{\beta\conj{\tau}}-\nabla_{\conj{\tau}}\weylcomp{Z}_{\beta\sigma}, \\
	\smash{\Pi^{(2)}}_{\beta\sigma\tau}
	&=2\nabla_{[\sigma|}\weylcomp{Z}_{\beta|\tau]}
	+\tensor{N}{^{\conj{\gamma}}_\sigma_\tau}\weylcomp{Z}_{\beta\conj{\gamma}}, \\
	\smash{\Pi^{(2)}}_{\beta\conj{\sigma}\conj{\tau}}
	&=2\nabla_{[\conj{\sigma}|}\weylcomp{Z}_{\beta|\conj{\tau}]}
	+\tensor{N}{^\gamma_{\conj{\sigma}}_{\conj{\tau}}}\weylcomp{Z}_{\beta\gamma}.
\end{align}
\allowdisplaybreaks[0]%
Now if we write
\begin{equation}
	\underline{K}
	=\begin{pmatrix}
		\bm{a} & \bm{Z}_\beta & * \\
		* & \tensor{\bm{A}}{_\beta^\alpha} & -\bm{Z}^\alpha \\
		& * & -\conj{\bm{a}}
	\end{pmatrix},
\end{equation}
then we obtain from \eqref{eq:curvature-modulo-homogeneity-4} that
\allowdisplaybreaks%
\begin{align}
	\tensor{\bm{A}}{_\beta^\alpha_0_\sigma}
	&=\tensor{\smash{\Pi^{(2)}}}{_\beta^\alpha_0_\sigma}
	-\tensor{\delta}{^\alpha_\sigma}\weylcomp{Z}_{\beta 0}, \\
	\bm{a}_{0\sigma}
	&={\smash{\Pi^{(2)}}}_{0\sigma}+\weylcomp{Z}_{\sigma 0}+\weylcomp{z}_\sigma, \\
	\bm{a}_{0\conj{\sigma}}
	&={\smash{\Pi^{(2)}}}_{0\conj{\sigma}}+\weylcomp{z}_{\conj{\sigma}}, \\
	\bm{Z}_{\beta\sigma\conj{\tau}}
	&={\smash{\Pi^{(2)}}}_{\beta\sigma\conj{\tau}}
	-ih_{\beta\conj{\tau}}\weylcomp{z}_\sigma
	+ih_{\sigma\conj{\tau}}\weylcomp{Z}_{\beta 0}, \\
	\bm{Z}_{\beta\sigma\tau}
	&={\smash{\Pi^{(2)}}}_{\beta\sigma\tau}, \\
	\bm{Z}_{\beta\conj{\sigma}\conj{\tau}}
	&={\smash{\Pi^{(2)}}}_{\beta\conj{\sigma}\conj{\tau}}
	+ih_{\beta\conj{\sigma}}\weylcomp{z}_{\conj{\tau}}
	-ih_{\beta\conj{\tau}}\weylcomp{z}_{\conj{\sigma}}.
\end{align}
\allowdisplaybreaks[0]%

The normality condition \eqref{eq:normality-of-Weyl-form-pulled-back} implies that
\begin{subequations}
\begin{gather}
	\label{eq:normal-weyl-form-homogeneity-three-equation-1}
	\tensor{\bm{A}}{_\beta^\gamma_0_\gamma}-\tensor{\bm{a}}{_0_\beta}+i\tensor{\bm{Z}}{_\beta_\gamma^\gamma}=0, \\
	\label{eq:normal-weyl-form-homogeneity-three-equation-2}
	-i\tensor{\bm{a}}{_0_\sigma}-i\tensor{\conj{\bm{a}}}{_0_\sigma}
	+\tensor{\bm{Z}}{_\gamma_\sigma^\gamma}-\tensor{\bm{Z}}{^\gamma_\sigma_\gamma}=0.
\end{gather}
\end{subequations}
It follows from \eqref{eq:normal-weyl-form-homogeneity-three-equation-1} that
\begin{subequations}
\begin{equation}
	\label{eq:normal-weyl-form-homogeneity-three-term-1}
	\begin{split}
		\weylcomp{Z}_{\beta 0}
		&=\frac{1}{2n+1}(\tensor{\smash{\Pi^{(2)}}}{_\beta^\gamma_0_\gamma}
		-\tensor{\smash{\Pi^{(2)}}}{_0_\beta}+i\tensor{\smash{\Pi^{(2)}}}{_\beta_\gamma^\gamma}) \\
		&=\frac{1}{2n+1}(
		-\tensor{W}{_\beta^\gamma_\gamma}
		+\tfrac{1}{n+2}\tensor{W}{_\gamma^\gamma_\beta}
		-i\nabla_\gamma\tensor{\weylcomp{A}}{_\beta^\gamma_0}
		-\tfrac{1}{n+2}\tensor{W}{_\gamma^\gamma_\beta}+\nabla_\beta \weylcomp{a}_0
		+i\nabla_\gamma\tensor{\weylcomp{Z}}{_\beta^\gamma}-i\nabla^\gamma\tensor{\weylcomp{Z}}{_\beta_\gamma})\\
		&=-\frac{i}{2n+1}(\nabla_\gamma\tensor{{P'}}{_\beta^\gamma}+\nabla_\gamma\tensor{{P''}}{_\beta^\gamma}
		-2i\nabla^\gamma\tensor{A}{_\beta_\gamma}
		+\tfrac{1}{n}\nabla^\gamma(\nabla^*N)^\sym_{\beta\gamma}
		+\tfrac{1}{n+2}\nabla^\gamma(\nabla^*N)^\skew_{\beta\gamma}).
	\end{split}
\end{equation}
Likewise, \eqref{eq:normal-weyl-form-homogeneity-three-equation-2} implies that
\begin{equation}
	\label{eq:normal-weyl-form-homogeneity-three-term-2}
	\begin{split}
		\weylcomp{z}_\sigma
		&=-\frac{i}{2n+1}(\tensor{\smash{\Pi^{(2)}}}{_\gamma_\sigma^\gamma}
		-\tensor{\smash{\Pi^{(2)}}}{^\gamma_\sigma_\gamma}) \\
		&=-\frac{i}{2n+1}(
		\nabla_\sigma\tensor{\weylcomp{Z}}{_\gamma^\gamma}
		-\nabla^\gamma\tensor{\weylcomp{Z}}{_\gamma_\sigma}
		-\nabla_\sigma\tensor{\weylcomp{Z}}{^\gamma_\gamma}
		+\nabla_\gamma\tensor{\weylcomp{Z}}{^\gamma_\sigma}
		-\tensor{N}{_{\gamma_1}_\sigma_{\gamma_2}}\tensor{\weylcomp{Z}}{^{\gamma_2}^{\gamma_1}})\\
		&=\frac{i}{2n+1}
		(\nabla_\gamma\tensor{{P''}}{_\sigma^\gamma}
		-i\nabla^\gamma\tensor{A}{_\sigma_\gamma}
		-i\tensor{A}{^{\gamma_1}^{\gamma_2}}\tensor{N}{_{\gamma_1}_{\gamma_2}_\sigma} \\
		&\qquad+\tfrac{1}{n}(\nabla^\gamma(\nabla^*N)^\sym_{\sigma\gamma}
		-(\nabla^*N)^\sym_{\conj{\gamma}_1\conj{\gamma}_2}
		\tensor{N}{^{\conj{\gamma}_1}^{\conj{\gamma}_2}_\sigma})
		-\tfrac{1}{n+2}(\nabla^\gamma(\nabla^*N)^\skew_{\sigma\gamma}
		-(\nabla^*N)^\skew_{\conj{\gamma}_1\conj{\gamma}_2}
		\tensor{N}{^{\conj{\gamma}_1}^{\conj{\gamma}_2}_\sigma})).
	\end{split}
\end{equation}
\end{subequations}

The formulae that we have just obtained may be simplified further by the Bianchi identities
of the Tanaka--Webster connection, but there is no point in doing it here in full generality.
In the integrable case (i.e., if $N=0$), we have (see Lee \cite{Lee-88}*{Lemma 2.2})
\begin{equation}
	\nabla_\gamma\tensor{P}{_\alpha^\gamma}
	=\frac{1}{n+2}((2n+1)\nabla_\alpha P-i(n-1)\nabla^\gamma A_{\alpha\gamma}).
\end{equation}
Therefore, \eqref{eq:normal-weyl-form-homogeneity-three-term-1} becomes
\begin{equation}
	\weylcomp{Z}_{\beta 0}
	=-\frac{2}{2n+1}(\nabla^\gamma A_{\beta\gamma}+i\nabla_\gamma\tensor{P}{_\beta^\gamma})
	=-\frac{2i}{n+2}(\nabla_\beta P-i\nabla^\gamma A_{\beta\gamma}),
\end{equation}
which equals $-2iT_\beta$ in the notation of Gover--Graham \cite{Gover-Graham-05},
and \eqref{eq:normal-weyl-form-homogeneity-three-term-2} becomes
\begin{equation}
	\weylcomp{z}_\sigma
	=\frac{i}{2n+1}(\nabla_\gamma\tensor{P}{_\sigma^\gamma}-i\nabla^\gamma A_{\sigma\gamma})
	=\frac{i}{n+2}(\nabla_\beta P-i\nabla^\gamma A_{\beta\gamma}),
\end{equation}
which is $iT_\beta$ of \cite{Gover-Graham-05}.

\subsubsection{Homogeneity $4$ component}

This is the last step of the determination of the normal Weyl form $\tau$.
We already know that $\underline{\tau}$ equals
\begin{equation}
	\eta^{(3)}
	=
	\begin{pmatrix}
		-\frac{1}{n+2}\tensor{\omega}{_\gamma^\gamma}+\weylcomp{a}_0\theta & \weylcomp{Z}_{\beta\sigma}\theta^\sigma+\weylcomp{Z}_{\beta\conj{\sigma}}\theta^{\conj{\sigma}}+\weylcomp{Z}_{\beta 0}\theta & i\weylcomp{z}_\sigma\theta^\sigma+i\weylcomp{z}_{\conj{\sigma}}\theta^{\conj{\sigma}} \\
		\theta^\alpha & \tensor{\omega}{_\beta^\alpha}-\frac{1}{n+2}\tensor{\omega}{_\gamma^\gamma}\tensor{\delta}{_\beta^\alpha}+\tensor{\weylcomp{A}}{_\beta^\alpha_0}\theta & -\tensor{\weylcomp{Z}}{^\alpha_\sigma}\theta^\sigma-\tensor{\weylcomp{Z}}{^\alpha_{\conj{\sigma}}}\theta^{\conj{\sigma}}-\tensor{\weylcomp{Z}}{^\alpha_0}\theta \\
		i\theta & -\theta_\beta & -\frac{1}{n+2}\tensor{\omega}{_\gamma^\gamma}-\conj{\weylcomp{a}}_0\theta
	\end{pmatrix}
\end{equation}
modulo terms of homogeneity $\ge 4$,
where the homogeneity $\ge 2$ components in the matrix entries are given by
\eqref{eq:normal-weyl-form-homogeneity-two-term-1}, \eqref{eq:normal-weyl-form-homogeneity-two-term-2},
\eqref{eq:normal-weyl-form-homogeneity-two-term-3},
and \eqref{eq:normal-weyl-form-homogeneity-three-term-1}, \eqref{eq:normal-weyl-form-homogeneity-three-term-2}.
We write
\begin{equation}
	\underline{\tau}=
	\eta^{(3)}+
	\underbrace{
		\begin{pmatrix}
			\phantom{M} & \phantom{M} & i\weylcomp{z}_0\theta \\
			& & \\
			& &
		\end{pmatrix}
	}_{\text{homogeneity $\ge 4$}}.
\end{equation}
Then the curvature is given by the following formula modulo terms of homogeneity $\ge 5$:
\begin{equation}
	\label{eq:curvature-modulo-homogeneity-5}
	\underline{K}
	\equiv d\eta^{(3)}+\frac{1}{2}[\eta^{(3)}\wedge\eta^{(3)}]
	+\begin{pmatrix}
			\phantom{M} & \phantom{M} & i\weylcomp{z}_0 \\
			& & \\
			& &
	\end{pmatrix} d\theta
	+\left[\eta^{(0)}\wedge
		\begin{pmatrix}
			\phantom{M} & \phantom{M} & i\weylcomp{z}_0\theta \\
			& & \\
			& &
		\end{pmatrix}
	\right].
\end{equation}
In determining $\weylcomp{z}_0$, only the homogeneity $4$ components of $\underline{K}$ are involved.
So if we set
\begin{equation}
	d\eta^{(3)}+\frac{1}{2}[\eta^{(3)}\wedge\eta^{(3)}]
	=
	\begin{pmatrix}
		* & {\Pi^{(3)}}_\beta & i\Pi^{(3)} \\
		* & * & -{\smash{\Pi^{(3)}}}^\alpha \\
		& * & *
	\end{pmatrix}
\end{equation}
and
\begin{equation}
	{\smash{\Pi^{(3)}}}_\beta
	=\tfrac{1}{2}{\smash{\Pi^{(3)}}}_{\beta KL}\theta^K\wedge\theta^L,\qquad
	\Pi^{(3)}
	=\tfrac{1}{2}{\smash{\Pi^{(3)}}}_{KL}\theta^K\wedge\theta^L,
\end{equation}
where the indices $K$ and $L$ run through $\set{0,1,\dots,n,\conj{1},\dots,\conj{n}}$,
then we only need to compute its homogeneity $4$ components, namely
${\smash{\Pi^{(3)}}}_{\beta 0\sigma}$,
${\smash{\Pi^{(3)}}}_{\beta 0\conj{\sigma}}$,
and ${\smash{\Pi^{(3)}}}_{\sigma\conj{\tau}}$,
${\smash{\Pi^{(3)}}}_{\sigma\tau}$,
${\smash{\Pi^{(3)}}}_{\conj{\sigma}\conj{\tau}}$.
Actually, we do not need all of them as we see as follows. If the curvature of $\underline{\tau}$ is expressed as
\begin{equation}
	\underline{K}
	=\begin{pmatrix}
		* & \bm{Z}_\beta & i\bm{z} \\
		* & * & -\bm{Z}^\alpha \\
		& * & *
	\end{pmatrix},
\end{equation}
then \eqref{eq:curvature-modulo-homogeneity-5} implies
\begin{align}
	\bm{Z}_{\beta 0\sigma}
	&={\smash{\Pi^{(3)}}}_{\beta 0\sigma}, \\
	\bm{Z}_{\beta 0\conj{\sigma}}
	&={\smash{\Pi^{(3)}}}_{\beta 0\conj{\sigma}}-ih_{\beta\conj{\sigma}}\weylcomp{z}_0, \\
	\bm{z}_{\sigma\conj{\tau}}
	&={\smash{\Pi^{(3)}}}_{\sigma\conj{\tau}}+ih_{\sigma\conj{\tau}}\weylcomp{z}_0, \\
	\bm{z}_{\sigma\tau}
	&={\smash{\Pi^{(3)}}}_{\sigma\tau}, \\
	\bm{z}_{\conj{\sigma}\conj{\tau}}
	&={\smash{\Pi^{(3)}}}_{\conj{\sigma}\conj{\tau}},
\end{align}
and the normality condition \eqref{eq:normality-of-Weyl-form-pulled-back} in this degree reads
\begin{equation}
	\tensor{\bm{Z}}{_\gamma_0^\gamma}-\tensor{\bm{Z}}{^\gamma_0_\gamma}-\tensor{\bm{z}}{_\gamma^\gamma}=0,
\end{equation}
which implies that
\begin{equation}
	\weylcomp{z}_0=\frac{i}{3n}(-\tensor{\smash{\Pi^{(3)}}}{_\gamma_0^\gamma}
	+\tensor{\smash{\Pi^{(3)}}}{^\gamma_0_\gamma}+\tensor{\smash{\Pi^{(3)}}}{_\gamma^\gamma}).
\end{equation}
Therefore, it suffices to compute $\tensor{\smash{\Pi^{(3)}}}{_\beta_0_{\conj{\sigma}}}$ and
$\tensor{\smash{\Pi^{(3)}}}{_\sigma_{\conj{\tau}}}$.
By a direct computation, we obtain
\begin{align}
	{\smash{\Pi^{(3)}}}_{\beta 0\conj{\sigma}}
	&=-\nabla_{\conj{\sigma}}\weylcomp{Z}_{\beta 0}
	+\nabla_0\weylcomp{Z}_{\beta\conj{\sigma}}
	+\tensor{A}{_{\conj{\sigma}}^\gamma}\weylcomp{Z}_{\beta\gamma}
	-\tensor{\weylcomp{A}}{_\beta^\gamma_0}\weylcomp{Z}_{\gamma\conj{\sigma}}
	+\weylcomp{a}_0\weylcomp{Z}_{\beta\conj{\sigma}},\\
	{\smash{\Pi^{(3)}}}_{\sigma\conj{\tau}}
	&=\nabla_\sigma \weylcomp{z}_{\conj{\tau}}-\nabla_{\conj{\tau}}\weylcomp{z}_\sigma
	-i\tensor{\weylcomp{Z}}{^\gamma_\sigma}\tensor{\weylcomp{Z}}{_\gamma_{\conj{\tau}}}
	+i\tensor{\weylcomp{Z}}{_\gamma_\sigma}\tensor{\weylcomp{Z}}{^\gamma_{\conj{\tau}}}.
\end{align}
Consequently, we get
\begin{equation}
	\label{eq:normal-weyl-form-homogeneity-four-term}
	\begin{multlined}
		\weylcomp{z}_0
		=\frac{i}{3n}
		(\nabla^\gamma\weylcomp{Z}_{\gamma 0}-\nabla^{\conj{\gamma}}\weylcomp{Z}_{\conj{\gamma}0}
		-A^{\gamma_1\gamma_2}\weylcomp{Z}_{\gamma_1\gamma_2}
		+A^{\conj{\gamma}_1\conj{\gamma}_2}\weylcomp{Z}_{\conj{\gamma}_1\conj{\gamma}_2} \\
		+\tensor{\weylcomp{A}}{_{\gamma_1}^{\gamma_2}_0}\tensor{\weylcomp{Z}}{_{\gamma_2}^{\gamma_1}}
		-\tensor{\weylcomp{A}}{_{\conj{\gamma}_1}^{\conj{\gamma}_2}_0}\tensor{\weylcomp{Z}}{_{\conj{\gamma}_2}^{\conj{\gamma}_1}}
		-2\weylcomp{a}_0\tensor{\weylcomp{Z}}{_\gamma^\gamma}
		-\nabla^\gamma \weylcomp{z}_\gamma+\nabla^{\conj{\gamma}}\weylcomp{z}_{\conj{\gamma}}
		-i\weylcomp{Z}_{\gamma_1\conj{\gamma}_2}\weylcomp{Z}^{\gamma_1\conj{\gamma}_2}
		+i\weylcomp{Z}_{\gamma_1\gamma_2}\weylcomp{Z}^{\gamma_1\gamma_2}).
	\end{multlined}
\end{equation}
The explicit formula in terms of the Tanaka--Webster invariants, which we omit,
can be obtained by
putting \eqref{eq:normal-weyl-form-homogeneity-two-term-1}, \eqref{eq:normal-weyl-form-homogeneity-two-term-2},
\eqref{eq:normal-weyl-form-homogeneity-two-term-3},
and \eqref{eq:normal-weyl-form-homogeneity-three-term-1}, \eqref{eq:normal-weyl-form-homogeneity-three-term-2}
into the above.
In the integrable case, it becomes
\begin{equation}
	\weylcomp{z}_0
	=\frac{1}{n}
	(\nabla^\gamma T_\gamma+\nabla^{\conj{\gamma}}T_{\conj{\gamma}}
	+P_{\alpha\conj{\beta}}P^{\alpha\conj{\beta}}-A_{\alpha\beta}A^{\alpha\beta}),
\end{equation}
which is $-S$ in the notation of \cite{Gover-Graham-05}.

\section{Computation of CR BGG operators}
\label{sec:computation-of-CR-BGG-operators}

Having written down the CR normal Weyl form explicitly in the previous section,
we are now able to express the normal tractor connection associated with
the standard representation of $G^\sharp=\SU(n+1,1)$
and the one associated with the adjoint representation of $G=\PSU(n+1,1)$
using Proposition \ref{prop:tractor-connection-in-terms-of-Weyl-form}.
Then we can also derive the formula of the modified adjoint tractor connection $\tilde{\nabla}$.
We are going to compute the first BGG operator associated with these three tractor connections.

\subsection{The case of the normal standard tractor connection}

Let $\sigma_\theta$ be the Weyl structure associated with a contact form $\theta$.
Then, the conclusion from the previous section is that the associated normal Weyl form $\tau$ is given by
\begin{equation}
	\underline{\tau}
	=
	\begin{pmatrix}
		-\frac{1}{n+2}\tensor{\omega}{_\gamma^\gamma}+\weylcomp{a}_0\theta &
		\weylcomp{Z}_{\beta\sigma}\theta^\sigma
		+\weylcomp{Z}_{\beta\conj{\sigma}}\theta^{\conj{\sigma}}
		+\weylcomp{Z}_{\beta 0}\theta &
		i\weylcomp{z}_\sigma\theta^\sigma+i\weylcomp{z}_{\conj{\sigma}}\theta^{\conj{\sigma}}+i\weylcomp{z}_0\theta \\
		\theta^\alpha &
		\tensor{\omega}{_\beta^\alpha}
		-\frac{1}{n+2}\tensor{\omega}{_\gamma^\gamma}\tensor{\delta}{_\beta^\alpha}
		+\tensor{\weylcomp{A}}{_\beta^\alpha_0}\theta &
		-\tensor{\weylcomp{Z}}{^\alpha_\sigma}\theta^\sigma
		-\tensor{\weylcomp{Z}}{^\alpha_{\conj{\sigma}}}\theta^{\conj{\sigma}}
		-\tensor{\weylcomp{Z}}{^\alpha_0}\theta \\
		i\theta & -\theta_\beta & -\frac{1}{n+2}\tensor{\omega}{_\gamma^\gamma}-\conj{\weylcomp{a}}_0\theta
	\end{pmatrix},
\end{equation}
where \eqref{eq:normal-weyl-form-homogeneity-two-term-1},
\eqref{eq:normal-weyl-form-homogeneity-two-term-2},
\eqref{eq:normal-weyl-form-homogeneity-two-term-3},
\eqref{eq:normal-weyl-form-homogeneity-three-term-1},
\eqref{eq:normal-weyl-form-homogeneity-three-term-2},
and \eqref{eq:normal-weyl-form-homogeneity-four-term} are observed.

Proposition \ref{prop:tractor-connection-in-terms-of-Weyl-form} implies that
the normal tractor connection $\nabla$ for the standard tractor bundle
$\mathcal{V}=\mathcal{G}^\sharp\times_{P^\sharp}\mathbb{V}\cong\mathcal{G}^\sharp_0\times_{G^\sharp_0}\mathbb{V}$,
the latter identification being given by the Weyl structure $\sigma_\theta$, can be expressed as
\begin{subequations}
\begin{align}
	\label{eq:normal-standard-tractor-connection-1}
	\nabla_\sigma
	\begin{pmatrix}
		s \\ t^\alpha \\ u
	\end{pmatrix}
	&=
	\begin{pmatrix}
		\nabla_\sigma s+\weylcomp{Z}_{\beta\sigma}t^\beta+i\weylcomp{z}_\sigma u \\
		\nabla_\sigma t^\alpha+\tensor{\delta}{_\sigma^\alpha}s-\tensor{\weylcomp{Z}}{^\alpha_\sigma}u \\
		\nabla_\sigma u
	\end{pmatrix},\\
	\label{eq:normal-standard-tractor-connection-2}
	\nabla_{\conj{\sigma}}
	\begin{pmatrix}
		s \\ t^\alpha \\ u
	\end{pmatrix}
	&=
	\begin{pmatrix}
		\nabla_{\conj{\sigma}}s+\weylcomp{Z}_{\beta\conj{\sigma}}t^\beta+i\weylcomp{z}_{\conj{\sigma}}u \\
		\nabla_{\conj{\sigma}}t^\alpha-\tensor{\weylcomp{Z}}{^\alpha_{\conj{\sigma}}}u \\
		\nabla_{\conj{\sigma}}u-t_{\conj{\sigma}}
	\end{pmatrix},\\
	\label{eq:normal-standard-tractor-connection-3}
	\nabla_0
	\begin{pmatrix}
		s \\ t^\alpha \\ u
	\end{pmatrix}
	&=
	\begin{pmatrix}
		\nabla_0s+\weylcomp{a}_0s+\weylcomp{Z}_{\beta 0}t^\beta+i\weylcomp{z}_0u \\
		\nabla_0t^\alpha+\tensor{\weylcomp{A}}{_\beta^\alpha_0}t^\beta-\tensor{\weylcomp{Z}}{^\alpha_0}u \\
		\nabla_0u-\conj{\weylcomp{a}}_0u+is
	\end{pmatrix}.
\end{align}
\end{subequations}
For later purpose, let us also write down the normal tractor connection associated with
the dual representation of $\mathbb{V}$ (the normal standard cotractor connection).
It is given by
\begin{subequations}
\begin{align}
	\label{eq:normal-standard-cotractor-connection-1}
	\nabla_\sigma \begin{pmatrix} \sigma & \tau_\alpha & \rho \end{pmatrix}
	&=
	\begin{pmatrix}
		\nabla_\sigma\sigma-\tau_\sigma &
		\nabla_\sigma\tau_\alpha-\weylcomp{Z}_{\alpha\sigma}\sigma &
		\nabla_\sigma\rho+\tensor{\weylcomp{Z}}{^\gamma_\sigma}\tau_\gamma-i\weylcomp{z}_\sigma\sigma
	\end{pmatrix},\\
	\label{eq:normal-standard-cotractor-connection-2}
	\nabla_{\conj{\sigma}} \begin{pmatrix} \sigma & \tau_\alpha & \rho \end{pmatrix}
	&=
	\begin{pmatrix}
		\nabla_{\conj{\sigma}}\sigma &
		\nabla_{\conj{\sigma}}\tau_\alpha+h_{\alpha\conj{\sigma}}\rho-\weylcomp{Z}_{\alpha\conj{\sigma}}\sigma &
		\nabla_{\conj{\sigma}}\rho+\tensor{\weylcomp{Z}}{^\gamma_{\conj{\sigma}}}\tau_\gamma-i\weylcomp{z}_{\conj{\sigma}}\sigma
	\end{pmatrix},\\
	\label{eq:normal-standard-cotractor-connection-3}
	\nabla_0 \begin{pmatrix} \sigma & \tau_\alpha & \rho \end{pmatrix}
	&=
	\begin{pmatrix}
		\nabla_0\sigma-\weylcomp{a}_0\sigma-i\rho &
		\nabla_0\tau_\alpha-\smash{\tensor{\weylcomp{A}}{_\alpha^\beta_0}}\tau_\beta-\weylcomp{Z}_{\alpha 0}\sigma &
		\nabla_0\rho+\conj{\weylcomp{a}}_0\rho+\smash{\tensor{\weylcomp{Z}}{^\beta_0}}\tau_\beta-i\weylcomp{z}_0\sigma
	\end{pmatrix},
\end{align}
\end{subequations}
which boils down to \cite{Gover-Graham-05}*{Equation (3.3)} in the integrable case.

In order to compute the associated first BGG operator $D_0$,
we need to identify the lifting operator
$L\colon \mathcal{H}_0(\mathfrak{p}_+,\mathbb{V})\to\mathcal{V}$ satisfying
that $\nabla Lu$ lies in the kernel of $\partial^*\colon \Omega^1(\mathcal{V})\to\mathcal{V}$
for every $u=\mathcal{H}_0(\mathfrak{p}_+,\mathbb{V})\cong\mathcal{E}(0,1)$.
If we write
\begin{equation}
	Lu=
	\begin{pmatrix}
		s \\ t^\alpha \\ u
	\end{pmatrix},
\end{equation}
then computations in Section \ref{subsec:homology-with-values-in-standard} shows that
$\partial^*\nabla Lu=0$ implies
\begin{equation}
	t^\alpha=\nabla^\alpha u,\qquad
	s=\frac{1}{n+1}(-\nabla_\gamma t^\gamma+i\nabla_0 u-i\conj{\weylcomp{a}}_0u+\tensor{\weylcomp{Z}}{^\gamma_\gamma}u).
\end{equation}
Therefore, $\nabla Lu$ projects down to the element
$(\nabla_\sigma u,\nabla_{(\conj{\alpha}}t_{\conj{\beta})}-\weylcomp{Z}_{(\conj{\alpha}\conj{\beta})}u)$ of
$\mathcal{H}_1(\mathfrak{p}_+,\mathbb{V})=\mathcal{E}_\sigma(0,1)\oplus\mathcal{E}_{(\conj{\alpha}\conj{\beta})}(0,1)$.
As a consequence, we can conclude that the first BGG operator $D_0$ is given by
\begin{equation}
	\label{eq:standard-first-BGG-operator}
	D_0u=
	(\nabla_\alpha u,
	\nabla_{(\conj{\alpha}}t_{\conj{\beta})}-\weylcomp{Z}_{(\conj{\alpha}\conj{\beta})}u)
	=\left(\nabla_\alpha u,
	\nabla_{(\conj{\alpha}}\nabla_{\conj{\beta})}u-iA_{\conj{\alpha}\conj{\beta}}u
	-\frac{1}{n}(\nabla^*N)^\sym_{\conj{\alpha}\conj{\beta}}u\right).
\end{equation}

In the integrable case, the above operator reduces to
\begin{equation}
	\label{eq:standard-first-BGG-operator-in-integrable-case}
	D_0\colon u\mapsto
	(\nabla_\alpha u,
	\nabla_{\conj{\alpha}}\nabla_{\conj{\beta}}u-iA_{\conj{\alpha}\conj{\beta}}u).
\end{equation}
This system of equations for $u\in\mathcal{E}(0,1)$ is essentially
discussed in \v{C}ap--Gover \cite{Cap-Gover-08}*{Section 4.14} in terms of the standard cotractor connection
(see also \cite{Cap-Gover-08}*{Section 4.2} carefully).
Their computation in the proof of \cite{Cap-Gover-08}*{Proposition 4.13}
shows that the standard tractor connection $\nabla$ is a prolongation of
the operator \eqref{eq:standard-first-BGG-operator-in-integrable-case}.

Therefore, it is a remarkable observation that if we relax the integrability condition,
then $D_0$ no longer prolongs to $\nabla$.
Formally, we can formulate this fact as follows.

\begin{prop}
	Let $\nabla$ be the standard tractor connection and
	$D_0\colon\mathcal{E}(0,1)\to\mathcal{E}_\sigma(0,1)\oplus\mathcal{E}_{(\conj{\alpha}\conj{\beta})}(0,1)$
	the associated first BGG operator
	\eqref{eq:standard-first-BGG-operator}.
	Assume that the Nijenhuis tensor $N$ is nonzero at a point $p\in M$,
	and there exists a nontrivial jet solution of the equation $D_0u=0$ at $p$.
	Then there are no linear differential operators
	$u\mapsto t^\alpha$ and $u\mapsto s$ in any neighborhood of $p$ such that
	\begin{equation}
		D_0u=0\qquad\text{at the level of jets at $p$}
	\end{equation}
	implies
	\begin{equation}
		\nabla
		\begin{pmatrix}
			s \\ t^\alpha \\ u
		\end{pmatrix}
		=0\qquad\text{at the level of jets at $p$}.
	\end{equation}
\end{prop}

\begin{proof}
	Suppose that there were such differential operators $u\mapsto t^\alpha$ and $u\mapsto s$.
	Then, if $D_0u=0$ is satisfied,
	then in view of \eqref{eq:normal-standard-tractor-connection-2} and \eqref{eq:normal-standard-tractor-connection-3},
	$t^\alpha$ and $s$ must be given by
	\begin{equation}
		t^\alpha=\nabla^\alpha u\qquad\text{and}\qquad
		s=i\nabla_0u-i\conj{\weylcomp{a}}_0u=i\nabla_0u+\frac{1}{n+2}Pu.
	\end{equation}
	Moreover, \eqref{eq:normal-standard-tractor-connection-1} implies that such $u$ should also satisfy
	$\nabla_\beta t^\alpha+\tensor{\delta}{_\beta^\alpha}s-\tensor{\weylcomp{Z}}{^\alpha_\beta}u=0$,
	or equivalently,
	\begin{equation}
		\nabla_\beta\nabla^\alpha u
		+\tensor{\delta}{_\beta^\alpha}\left(i\nabla_0u+\frac{1}{n+2}Pu\right)
		+\tensor{\smash{P''}}{^\alpha_\beta}u=0.
	\end{equation}
	Then it follows from Proposition \ref{prop:TW-commutation-on-densities} that
	\begin{equation}
		\nabla^\alpha\nabla_\beta u
		-\frac{1}{n+2}\tensor{R}{_\beta^\alpha}u
		+\frac{1}{n+2}\tensor{\delta}{_\beta^\alpha}Pu
		+\tensor{\smash{P''}}{^\alpha_\beta}u=0
	\end{equation}
	and hence, since $D_0u=0$ implies $\nabla_\beta u=0$,
	\begin{equation}
		-\frac{1}{n+2}\tensor{R}{_\beta^\alpha}u
		+\frac{1}{n+2}\tensor{\delta}{_\beta^\alpha}Pu
		+\tensor{\smash{P''}}{^\alpha_\beta}u=0.
	\end{equation}
	Taking the trace and $R=2(n+1)P$ yields $(P''-P)u=0$, which implies $\abs{N}^2u=0$ by
	\eqref{eq:relationship-of-various-P}.
	Note that our argument so far works as well at the level of jets.
	Therefore, there must be a nontrivial jet $u$ at $p$ satisfying $\abs{N}^2u=0$,
	but this contradicts our assumption that $N$ does not vanish at $p$.
\end{proof}

\subsection{The case of the normal adjoint tractor connection}
\label{subsec:BGG-opertor-for-normal-adjoint-tractor-connection}

Note that
\begin{equation}
	\mathfrak{g}=\Re\mathfrak{sl}(n+2,\mathbb{C})=\Re\tf\End(\mathbb{V})
\end{equation}
as a $G^\sharp$-module. Consequently, the adjoint tractor bundle
$\mathcal{A}M=\mathcal{G}\times_P\mathfrak{g}=\mathcal{G}^\sharp\times_{P^\sharp}\mathfrak{g}$ can be seen as
\begin{equation}
	\mathcal{A}M=\Re\tf\End(\mathcal{V}),
\end{equation}
and hence it follows from the results in the previous subsection
that the normal tractor connection of $\mathcal{A}M$ reads as
in Figure \ref{fig:normal-adjoint-tractor-connection}.

\setlength{\rotFPtop}{0pt plus 1fil}
\setlength{\rotFPbot}{0pt plus 1fil}
\begin{sidewaysfigure}
	\footnotesize
	\centering
	\begin{align}
		\nabla_\sigma
		&\begin{pmatrix}
			a & Z_\beta & iz \\
			X^\alpha & \tensor{A}{_\beta^\alpha} & -Z^\alpha \\
			ix & -X_\beta & -\conj{a}
		\end{pmatrix}
		=
		\begin{pmatrix}
			\nabla_\sigma a-Z_\sigma+\weylcomp{Z}_{\gamma\sigma}X^\gamma-\weylcomp{z}_\sigma x &
			\nabla_\sigma Z_\beta+\weylcomp{Z}_{\gamma\sigma}\tensor{A}{_\beta^\gamma}-\weylcomp{Z}_{\beta\sigma}a-i\weylcomp{z}_\sigma X_\beta &
			i\nabla_\sigma z+\tensor{\weylcomp{Z}}{^\gamma_\sigma}Z_\gamma-\weylcomp{Z}_{\gamma\sigma}Z^\gamma-2i\weylcomp{z}_\sigma\Re a \\
			\nabla_\sigma X^\alpha-\tensor{A}{_\sigma^\alpha}+\tensor{\delta}{^\alpha_\sigma}a-i\tensor{\weylcomp{Z}}{^\alpha_\sigma}x &
			\nabla_\sigma\tensor{A}{_\beta^\alpha}+\tensor{\delta}{_\sigma^\alpha}Z_\beta+\tensor{\weylcomp{Z}}{^\alpha_\sigma}X_\beta-\weylcomp{Z}_{\beta\sigma}X^\alpha &
			-\nabla_\sigma Z^\alpha+i\tensor{\delta}{_\sigma^\alpha}z+\tensor{\weylcomp{Z}}{^\gamma_\sigma}\tensor{A}{_\gamma^\alpha}+\tensor{\weylcomp{Z}}{^\alpha_\sigma}\conj{a}-i\weylcomp{z}_\sigma X^\alpha \\
			i\nabla_\sigma x+X_\sigma &
			-\nabla_\sigma X_\beta-i\weylcomp{Z}_{\beta\sigma}x &
			-\nabla_\sigma\conj{a}-\tensor{\weylcomp{Z}}{^\gamma_\sigma}X_\gamma+\weylcomp{z}_\sigma x
		\end{pmatrix}\\
		\nabla_{\conj{\sigma}}
		&\begin{pmatrix}
			a & Z_\beta & iz \\
			X^\alpha & \tensor{A}{_\beta^\alpha} & -Z^\alpha \\
			ix & -X_\beta & -\conj{a}
		\end{pmatrix}
		=
		\begin{pmatrix}
			\nabla_{\conj{\sigma}}a+\weylcomp{Z}_{\gamma\conj{\sigma}}X^\gamma-\weylcomp{z}_{\conj{\sigma}}x &
			\nabla_{\conj{\sigma}}Z_\beta+ih_{\beta\conj{\sigma}}z+\weylcomp{Z}_{\gamma\conj{\sigma}}\tensor{A}{_\beta^\gamma}-\weylcomp{Z}_{\beta\conj{\sigma}}a-i\weylcomp{z}_{\conj{\sigma}}X_\beta &
			i\nabla_{\conj{\sigma}}z-\tensor{\weylcomp{Z}}{_\gamma_{\conj{\sigma}}}Z^\gamma+\tensor{\weylcomp{Z}}{^\gamma_{\conj{\sigma}}}Z_\gamma-2i\weylcomp{z}_{\conj{\sigma}}\Re a \\
			\nabla_{\conj{\sigma}}X^\alpha-i\tensor{\weylcomp{Z}}{^\alpha_{\conj{\sigma}}}x &
			\nabla_{\conj{\sigma}}\tensor{A}{_\beta^\alpha}-h_{\beta\conj{\sigma}}Z^\alpha-\weylcomp{Z}_{\beta\conj{\sigma}}X^\alpha+\tensor{\weylcomp{Z}}{^\alpha_{\conj{\sigma}}}X_\beta &
			-\nabla_{\conj{\sigma}}Z^\alpha+\tensor{\weylcomp{Z}}{^\gamma_{\conj{\sigma}}}\tensor{A}{_\gamma^\alpha}+\tensor{\weylcomp{Z}}{^\alpha_{\conj{\sigma}}}\conj{a}-i\weylcomp{z}_{\conj{\sigma}}X^\alpha \\
			i\nabla_{\conj{\sigma}}x-X_{\conj{\sigma}} &
			-\nabla_{\conj{\sigma}}X_\beta-A_{\beta\conj{\sigma}}-h_{\beta\conj{\sigma}}\conj{a}-i\weylcomp{Z}_{\beta\conj{\sigma}}x &
			-\nabla_{\conj{\sigma}}\conj{a}+Z_{\conj{\sigma}}-\tensor{\weylcomp{Z}}{^\gamma_{\conj{\sigma}}}X_\gamma+\weylcomp{z}_{\conj{\sigma}}x
		\end{pmatrix}\\
		\nabla_0
		&\begin{pmatrix}
			a & Z_\beta & iz \\
			X^\alpha & \tensor{A}{_\beta^\alpha} & -Z^\alpha \\
			ix & -X_\beta & -\conj{a}
		\end{pmatrix} \\
		&=
		\begin{pmatrix}
			\nabla_0a+z+\weylcomp{Z}_{\gamma 0}X^\gamma-\weylcomp{z}_0x &
			\nabla_0Z_\beta-\tensor{\weylcomp{A}}{_\beta^\gamma_0}Z_\gamma+\weylcomp{a}_0Z_\beta+\weylcomp{Z}_{\gamma 0}\tensor{A}{_\beta^\gamma}-\weylcomp{Z}_{\beta 0}a-i\weylcomp{z}_0X_\beta &
			i\nabla_0z+\tensor{\weylcomp{Z}}{^\gamma_0}Z_\gamma-\weylcomp{Z}_{\gamma 0}Z^\gamma-2i\weylcomp{z}_0\Re a \\
			\nabla_0X^\alpha+iZ^\alpha+\tensor{\weylcomp{A}}{_\gamma^\alpha_0}X^\gamma-\weylcomp{a}_0X^\alpha-i\tensor{\weylcomp{Z}}{^\alpha_0}x &
			\nabla_0\tensor{A}{_\beta^\alpha}+\tensor{\weylcomp{A}}{_\gamma^\alpha_0}\tensor{A}{_\beta^\gamma}-\tensor{\weylcomp{A}}{_\beta^\gamma_0}\tensor{A}{_\gamma^\alpha}-\weylcomp{Z}_{\beta 0}X^\alpha+\tensor{\weylcomp{Z}}{^\alpha_0}X_\beta &
			-\nabla_0Z^\alpha-\tensor{\weylcomp{A}}{_\gamma^\alpha_0}Z^\gamma+\weylcomp{a}_0Z^\alpha+\tensor{\weylcomp{Z}}{^\gamma_0}\tensor{A}{_\gamma^\alpha}+\tensor{\weylcomp{Z}}{^\alpha_0}\conj{a}-i\weylcomp{z}_0X^\alpha \\
			i\nabla_0x+2i\Re a &
			-\nabla_0X_\beta+iZ_\beta+\tensor{\weylcomp{A}}{_\beta^\gamma_0}X_\gamma-\weylcomp{a}_0X_\beta-i\weylcomp{Z}_{\beta 0}x &
			-\nabla_0\conj{a}-z-\tensor{\weylcomp{Z}}{^\gamma_0}X_\gamma+\weylcomp{z}_0x
		\end{pmatrix}
	\end{align}
	\caption{The normal adjoint tractor connection}
	\label{fig:normal-adjoint-tractor-connection}
\end{sidewaysfigure}

The lifting operator $L$ acting on $\mathcal{H}_0(\mathfrak{p}_+,\mathfrak{g})\cong\Re\mathcal{E}(1,1)$
is defined so that $\nabla Lx$ lies in the kernel of
$\partial^*\colon\Omega^1(\mathcal{A}M)\to\mathcal{A}M$ for every $x\in\Re\mathcal{E}(1,1)$.
This implies that $Lx\in\mathcal{A}M$ is of the form
\begin{equation}
	Lx=
	\begin{pmatrix}
		* & * & * \\
		i\nabla^\alpha x & * & * \\
		ix & i\nabla_\beta x & *
	\end{pmatrix}.
\end{equation}
Then
\begin{equation}
	\begin{split}
		\nabla_\sigma Lx
		&=
		\begin{pmatrix}
			* & * & * \\
			* & * & * \\
			0 & i\nabla_\sigma\nabla_\beta x-i\weylcomp{Z}_{\beta\sigma}x & *
		\end{pmatrix} \\
		&=
		\begin{pmatrix}
			* & * & * \\
			* & * & * \\
			0 &
			i\nabla_\sigma\nabla_\beta x-A_{\beta\sigma}x
			-\tfrac{i}{n}(\nabla^*N)^\sym_{\beta\sigma}
			-\tfrac{i}{n+2}(\nabla^*N)^\skew_{\beta\sigma} &
			*
		\end{pmatrix}.
	\end{split}
\end{equation}
This implies that the first BGG operator is given by
\begin{equation}
	x\mapsto
	\left(-i\nabla_{(\alpha}\nabla_{\beta)}x
	+A_{\alpha\beta}x
	+\frac{i}{n}(\nabla^*N)^\sym_{\alpha\beta},
	i\nabla_{(\conj{\alpha}}\nabla_{\conj{\beta})}x
	+A_{\conj{\alpha}\conj{\beta}}x
	-\frac{i}{n}(\nabla^*N)^\sym_{\conj{\alpha}\conj{\beta}}\right).
\end{equation}

\subsection{The case of the modified adjoint tractor connection}

Recall from \eqref{eq:modified-tractor-connection} that, in order to compute the action
\begin{equation}
	\tilde{\nabla}
	\begin{pmatrix}
		a & Z_\beta & iz \\
		X^\alpha & \tensor{A}{_\beta^\alpha} & -Z^\alpha \\
		ix & -X_\beta & -\conj{a}
	\end{pmatrix}
\end{equation}
of the modified adjoint tractor connection $\tilde{\nabla}$,
we need to know the interior product of the curvature function $\underline{\kappa}$ and the vector field given by
\begin{equation}
	\begin{pmatrix}
		* & * & * \\
		X^\alpha & * & * \\
		ix & -X_\beta & *
	\end{pmatrix},
\end{equation}
which is $X^\alpha Z_\alpha+X^{\conj{\alpha}}Z_{\conj{\alpha}}+xT$.
The computation in the previous subsection implies that only
the $\mathfrak{g}_{-2}$- and the $\mathfrak{g}_{-1}$-valued components
are relevant to the first BGG operator associated with $\tilde{\nabla}$.

We know from \eqref{eq:normal-Weyl-curvature-up-to-homogeneity-one} that
the curvature of the normal Weyl form is of the form
\begin{equation}
	\underline{K}=
	\begin{pmatrix}
		* & * & * \\
		\smash{\Pi^{(2)}}^\alpha & * & * \\
		0 & -\smash{\Pi^{(2)}}_\beta & *
	\end{pmatrix}
\end{equation}
where
\begin{equation}
	\smash{\Pi^{(2)}}^\alpha
	=\tfrac{1}{2}\tensor{N}{^\alpha_{\conj{\sigma}}_{\conj{\tau}}}\theta^{\conj{\sigma}}\wedge\theta^{\conj{\tau}}
	+\tensor{\smash{\Pi^{(2)}}}{^\alpha_0_\sigma}\theta\wedge\theta^\sigma
	+\tensor{\smash{\Pi^{(2)}}}{^\alpha_0_{\conj{\sigma}}}\theta\wedge\theta^{\conj{\sigma}}.
\end{equation}
The curvature components $\tensor{\smash{\Pi^{(2)}}}{^\alpha_0_\sigma}$ and
$\tensor{\smash{\Pi^{(2)}}}{^\alpha_0_{\conj{\sigma}}}$ can be computed
using \eqref{eq:curvature-modulo-homogeneity-3}, and it follows that
\begin{equation}
	\iota_{X^\alpha Z_\alpha+X^{\conj{\alpha}}Z_{\conj{\alpha}}+xT}\smash{\Pi^{(2)}}^\alpha
	\equiv -\tensor{N}{^\alpha_{\conj{\sigma}}_{\conj{\tau}}}X^{\conj{\tau}}\theta^{\conj{\sigma}}
	+(\tensor{\weylcomp{A}}{_\sigma^\alpha_0}-\tensor{\delta}{_\sigma^\alpha}\weylcomp{a}_0
	+i\tensor{\weylcomp{Z}}{^\alpha_\sigma})x\theta^\sigma
	+(\tensor{A}{^\alpha_{\conj{\sigma}}}+i\tensor{\weylcomp{Z}}{^\alpha_{\conj{\sigma}}})x\theta^{\conj{\sigma}}
	\mod\theta.
\end{equation}
Consequently, the lifting operator $L$ remains the same as in the previous subsection
and the first BGG operator is given by
\begin{equation}
	\begin{split}
		D_0^{\tilde{\nabla}}x
		&=D_0^\nabla x+\proj(\iota_{\Pi(Lx)}\underline{K}) \\
		&=D_0^\nabla x
		+(-iN^\sym_{\alpha\beta\gamma}\nabla^\gamma x
		+A_{\alpha\beta}x-i\weylcomp{Z}_{(\alpha\beta)}x,
		iN^\sym_{\conj{\alpha}\conj{\beta}\conj{\gamma}}\nabla^{\conj{\gamma}}x
		+A_{\conj{\alpha}\conj{\beta}}x+i\weylcomp{Z}_{(\conj{\alpha}\conj{\beta})}x) \\
		&=
		\left(-i\nabla_{(\alpha}\nabla_{\beta)}x
		+A_{\alpha\beta}x
		-iN^\sym_{\alpha\beta\gamma}\nabla^\gamma x,
		i\nabla_{(\conj{\alpha}}\nabla_{\conj{\beta})}x
		+A_{\conj{\alpha}\conj{\beta}}x
		+iN^\sym_{\conj{\alpha}\conj{\beta}\conj{\gamma}}\nabla^{\conj{\gamma}}x\right).
	\end{split}
\end{equation}
This shows that $D_0^{\tilde{\nabla}}$ is nothing but the CR Killing operator.

\bibliography{myrefs}

\begin{bibdiv}
\begin{biblist}

\bib{Akahori-Garfield-Lee-02}{article}{
      author={Akahori, Takao},
      author={Garfield, Peter~M.},
      author={Lee, John~M.},
       title={Deformation theory of 5-dimensional {CR} structures and the
  {R}umin complex},
        date={2002},
        ISSN={0026-2285},
     journal={Michigan Math. J.},
      volume={50},
      number={3},
       pages={517\ndash 549},
         url={http://dx.doi.org/10.1307/mmj/1039029981},
}

\bib{Bailey-Eastwood-Gover-94}{article}{
      author={Bailey, T.~N.},
      author={Eastwood, M.~G.},
      author={Gover, A.~R.},
       title={Thomas's structure bundle for conformal, projective and related
  structures},
        date={1994},
        ISSN={0035-7596},
     journal={Rocky Mountain J. Math.},
      volume={24},
      number={4},
       pages={1191\ndash 1217},
         url={http://dx.doi.org/10.1216/rmjm/1181072333},
}

\bib{Biquard-00}{article}{
      author={Biquard, Olivier},
       title={M\'etriques d'{E}instein asymptotiquement sym\'etriques},
        date={2000},
        ISSN={0303-1179},
     journal={Ast\'erisque},
      number={265},
       pages={vi+109},
}

\bib{Biquard-Herzlich-14}{article}{
      author={Biquard, Olivier},
      author={Herzlich, Marc},
       title={Analyse sur un demi-espace hyperbolique et
  poly-homog\'{e}n\'{e}it\'{e} locale},
        date={2014},
        ISSN={0944-2669},
     journal={Calc. Var. Partial Differential Equations},
      volume={51},
      number={3-4},
       pages={813\ndash 848},
         url={https://mathscinet.ams.org/mathscinet-getitem?mr=3268872},
      review={\MR{3268872}},
}

\bib{Blair-76}{book}{
      author={Blair, David~E.},
       title={Contact manifolds in {R}iemannian geometry},
      series={Lecture Notes in Mathematics, Vol. 509},
   publisher={Springer-Verlag, Berlin-New York},
        date={1976},
         url={https://mathscinet.ams.org/mathscinet-getitem?mr=0467588},
      review={\MR{0467588}},
}

\bib{Calderbank-Diemer-01}{article}{
      author={Calderbank, David M.~J.},
      author={Diemer, Tammo},
       title={Differential invariants and curved
  {B}ernstein-{G}elfand-{G}elfand sequences},
        date={2001},
        ISSN={0075-4102},
     journal={J. Reine Angew. Math.},
      volume={537},
       pages={67\ndash 103},
         url={https://mathscinet.ams.org/mathscinet-getitem?mr=1856258},
      review={\MR{1856258}},
}

\bib{Cap-02}{article}{
      author={{\v{C}}ap, Andreas},
       title={Parabolic geometries, {CR}-tractors, and the {F}efferman
  construction},
        date={2002},
        ISSN={0926-2245},
     journal={Differential Geom. Appl.},
      volume={17},
      number={2-3},
       pages={123\ndash 138},
         url={https://mathscinet.ams.org/mathscinet-getitem?mr=1925761},
        note={8th International Conference on Differential Geometry and its
  Applications (Opava, 2001)},
      review={\MR{1925761}},
}

\bib{Cap-05}{article}{
      author={{\v{C}}ap, Andreas},
       title={Correspondence spaces and twistor spaces for parabolic
  geometries},
        date={2005},
        ISSN={0075-4102},
     journal={J. Reine Angew. Math.},
      volume={582},
       pages={143\ndash 172},
         url={https://mathscinet.ams.org/mathscinet-getitem?mr=2139714},
      review={\MR{2139714}},
}

\bib{Cap-08}{article}{
      author={{\v{C}}ap, Andreas},
       title={Infinitesimal automorphisms and deformations of parabolic
  geometries},
        date={2008},
        ISSN={1435-9855},
     journal={J. Eur. Math. Soc. (JEMS)},
      volume={10},
      number={2},
       pages={415\ndash 437},
         url={https://mathscinet.ams.org/mathscinet-getitem?mr=2390330},
      review={\MR{2390330}},
}

\bib{Cap-Gover-03}{article}{
      author={{\v{C}}ap, Andreas},
      author={Gover, A.~Rod},
       title={Standard tractors and the conformal ambient metric construction},
        date={2003},
        ISSN={0232-704X},
     journal={Ann. Global Anal. Geom.},
      volume={24},
      number={3},
       pages={231\ndash 259},
         url={https://mathscinet.ams.org/mathscinet-getitem?mr=1996768},
      review={\MR{1996768}},
}

\bib{Cap-Gover-08}{article}{
      author={{\v{C}}ap, Andreas},
      author={Gover, A.~Rod},
       title={C{R}-tractors and the {F}efferman space},
        date={2008},
        ISSN={0022-2518},
     journal={Indiana Univ. Math. J.},
      volume={57},
      number={5},
       pages={2519\ndash 2570},
         url={https://mathscinet.ams.org/mathscinet-getitem?mr=2463976},
      review={\MR{2463976}},
}

\bib{Cap-Schichl-00}{article}{
      author={{\v{C}}ap, Andreas},
      author={Schichl, Hermann},
       title={Parabolic geometries and canonical {C}artan connections},
        date={2000},
        ISSN={0385-4035},
     journal={Hokkaido Math. J.},
      volume={29},
      number={3},
       pages={453\ndash 505},
         url={http://dx.doi.org/10.14492/hokmj/1350912986},
}

\bib{Cap-Slovak-09}{book}{
      author={{\v{C}}ap, Andreas},
      author={Slov\'{a}k, Jan},
       title={Parabolic geometries. {I}},
      series={Mathematical Surveys and Monographs},
   publisher={American Mathematical Society, Providence, RI},
        date={2009},
      volume={154},
        ISBN={978-0-8218-2681-2},
         url={https://mathscinet.ams.org/mathscinet-getitem?mr=2532439},
        note={Background and general theory},
      review={\MR{2532439}},
}

\bib{Cap-Slovak-Soucek-01}{article}{
      author={{\v{C}}ap, Andreas},
      author={Slov\'{a}k, Jan},
      author={Sou\v{c}ek, Vladim\'{\i}r},
       title={Bernstein-{G}elfand-{G}elfand sequences},
        date={2001},
        ISSN={0003-486X},
     journal={Ann. of Math. (2)},
      volume={154},
      number={1},
       pages={97\ndash 113},
         url={https://mathscinet.ams.org/mathscinet-getitem?mr=1847589},
      review={\MR{1847589}},
}

\bib{Chern-Moser-74}{article}{
      author={Chern, S.~S.},
      author={Moser, J.~K.},
       title={Real hypersurfaces in complex manifolds},
        date={1974},
        ISSN={0001-5962},
     journal={Acta Math.},
      volume={133},
       pages={219\ndash 271},
      review={\MR{0425155 (54 \#13112)}},
}

\bib{Chrusciel-Delay-Lee-Skinner-05}{article}{
      author={Chru{\'s}ciel, Piotr~T.},
      author={Delay, Erwann},
      author={Lee, John~M.},
      author={Skinner, Dale~N.},
       title={Boundary regularity of conformally compact {E}instein metrics},
        date={2005},
        ISSN={0022-040X},
     journal={J. Differential Geom.},
      volume={69},
      number={1},
       pages={111\ndash 136},
         url={http://projecteuclid.org/euclid.jdg/1121540341},
}

\bib{Fefferman-Graham-85}{article}{
      author={Fefferman, Charles},
      author={Graham, C.~Robin},
       title={Conformal invariants},
        date={1985},
        ISSN={0303-1179},
     journal={Ast\'erisque Numero Hors Serie},
       pages={95\ndash 116},
        note={The mathematical heritage of {\'E}lie Cartan (Lyon, 1984)},
}

\bib{Fefferman-Graham-12}{book}{
      author={Fefferman, Charles},
      author={Graham, C.~Robin},
       title={The ambient metric},
      series={Annals of Mathematics Studies},
   publisher={Princeton University Press, Princeton, NJ},
        date={2012},
      volume={178},
        ISBN={978-0-691-15313-1},
}

\bib{Fefferman-76}{article}{
      author={Fefferman, Charles~L.},
       title={Monge-{A}mp\`ere equations, the {B}ergman kernel, and geometry of
  pseudoconvex domains},
        date={1976},
        ISSN={0003-486X},
     journal={Ann. of Math. (2)},
      volume={103},
      number={2},
       pages={395\ndash 416},
}

\bib{Gover-Graham-05}{article}{
      author={Gover, A.~Rod},
      author={Graham, C.~Robin},
       title={C{R} invariant powers of the sub-{L}aplacian},
        date={2005},
        ISSN={0075-4102},
     journal={J. Reine Angew. Math.},
      volume={583},
       pages={1\ndash 27},
         url={http://dx.doi.org/10.1515/crll.2005.2005.583.1},
}

\bib{Graham-Hirachi-05}{incollection}{
      author={Graham, C.~Robin},
      author={Hirachi, Kengo},
       title={The ambient obstruction tensor and {$Q$}-curvature},
        date={2005},
   booktitle={Ad{S}/{CFT} correspondence: {E}instein metrics and their
  conformal boundaries},
      series={IRMA Lect. Math. Theor. Phys.},
      volume={8},
   publisher={Eur. Math. Soc., Z\"{u}rich},
       pages={59\ndash 71},
         url={https://mathscinet.ams.org/mathscinet-getitem?mr=2160867},
      review={\MR{2160867}},
}

\bib{Hammerl-09-Thesis}{thesis}{
      author={Hammerl, Matthias},
       title={Natural {P}rolongations of {BGG}-{O}perators},
        type={Ph.D. thesis},
organization={Universit{\"a}t Wien},
        date={2009},
}

\bib{Hammerl-Somberg-Soucek-Silhan-12}{article}{
      author={Hammerl, Matthias},
      author={Somberg, Petr},
      author={Sou\v{c}ek, Vladim\'{\i}r},
      author={\v{S}ilhan, Josef},
       title={On a new normalization for tractor covariant derivatives},
        date={2012},
        ISSN={1435-9855},
     journal={J. Eur. Math. Soc. (JEMS)},
      volume={14},
      number={6},
       pages={1859\ndash 1883},
         url={https://mathscinet.ams.org/mathscinet-getitem?mr=2984590},
      review={\MR{2984590}},
}

\bib{Herzlich-09}{article}{
      author={Herzlich, Marc},
       title={The canonical {C}artan bundle and connection in {CR} geometry},
        date={2009},
        ISSN={0305-0041},
     journal={Math. Proc. Cambridge Philos. Soc.},
      volume={146},
      number={2},
       pages={415\ndash 434},
         url={https://mathscinet.ams.org/mathscinet-getitem?mr=2475975},
      review={\MR{2475975}},
}

\bib{Hirachi-Marugame-Matsumoto-17}{article}{
      author={Hirachi, Kengo},
      author={Marugame, Taiji},
      author={Matsumoto, Yoshihiko},
       title={Variation of total {$Q$}-prime curvature on {CR} manifolds},
        date={2017},
        ISSN={0001-8708},
     journal={Adv. Math.},
      volume={306},
       pages={1333\ndash 1376},
         url={https://mathscinet.ams.org/mathscinet-getitem?mr=3581332},
      review={\MR{3581332}},
}

\bib{Lee-86}{article}{
      author={Lee, John~M.},
       title={The {F}efferman metric and pseudohermitian invariants},
        date={1986},
        ISSN={0002-9947},
     journal={Trans. Amer. Math. Soc.},
      volume={296},
      number={1},
       pages={411\ndash 429},
         url={http://dx.doi.org/10.2307/2000582},
}

\bib{Lee-88}{article}{
      author={Lee, John~M.},
       title={Pseudo-{E}instein structures on {CR} manifolds},
        date={1988},
        ISSN={0002-9327},
     journal={Amer. J. Math.},
      volume={110},
      number={1},
       pages={157\ndash 178},
         url={http://dx.doi.org/10.2307/2374543},
}

\bib{Matsumoto-14}{article}{
      author={Matsumoto, Yoshihiko},
       title={Asymptotics of {ACH}-{E}instein metrics},
        date={2014},
        ISSN={1050-6926},
     journal={J. Geom. Anal.},
      volume={24},
      number={4},
       pages={2135\ndash 2185},
         url={http://dx.doi.org/10.1007/s12220-013-9411-z},
}

\bib{Matsumoto-16}{article}{
      author={Matsumoto, Yoshihiko},
       title={G{JMS} operators, {$Q$}-curvature, and obstruction tensor of
  partially integrable {CR} manifolds},
        date={2016},
        ISSN={0926-2245},
     journal={Differential Geom. Appl.},
      volume={45},
       pages={78\ndash 114},
         url={https://doi.org/10.1016/j.difgeo.2016.01.002},
      review={\MR{3457389}},
}

\bib{Matsumoto-21}{article}{
      author={Matsumoto, Yoshihiko},
       title={Canonical almost complex structures on {ACH} {E}instein
  manifolds},
        date={2021},
        ISSN={0030-8730},
     journal={Pacific J. Math.},
      volume={314},
      number={2},
       pages={375\ndash 410},
         url={https://mathscinet.ams.org/mathscinet-getitem?mr=4337468},
      review={\MR{4337468}},
}

\bib{Sharpe-97}{book}{
      author={Sharpe, R.~W.},
       title={Differential geometry},
      series={Graduate Texts in Mathematics},
   publisher={Springer-Verlag, New York},
        date={1997},
      volume={166},
        ISBN={0-387-94732-9},
        note={Cartan's generalization of Klein's Erlangen program, With a
  foreword by S. S. Chern},
      review={\MR{1453120}},
}

\bib{Tanaka-62}{article}{
      author={Tanaka, Noboru},
       title={On the pseudo-conformal geometry of hypersurfaces of the space of
  {$n$} complex variables},
        date={1962},
        ISSN={0025-5645},
     journal={J. Math. Soc. Japan},
      volume={14},
       pages={397\ndash 429},
         url={https://mathscinet.ams.org/mathscinet-getitem?mr=145555},
      review={\MR{145555}},
}

\bib{Yamaguchi-93}{incollection}{
      author={Yamaguchi, Keizo},
       title={Differential systems associated with simple graded {L}ie
  algebras},
        date={1993},
   booktitle={Progress in differential geometry},
      series={Adv. Stud. Pure Math.},
      volume={22},
   publisher={Math. Soc. Japan, Tokyo},
       pages={413\ndash 494},
      review={\MR{1274961}},
}

\end{biblist}
\end{bibdiv}

\end{document}